\documentclass[12pt]{amsart}

\usepackage[top=1in, bottom=1in, left=1in, right=1in]{geometry}
\usepackage{times}
\usepackage[french,english]{babel}

\usepackage{amssymb,amsmath,amsthm}
\usepackage{mathrsfs} 
\usepackage{bbm} 
\usepackage{enumitem}

\usepackage{graphicx}
\usepackage{color}
\usepackage[hyphens]{url}
\usepackage[colorlinks=true, linkcolor=blue, citecolor=green, urlcolor=red]{hyperref}   

\definecolor{red}{rgb}{1,0,0}
\definecolor{orange}{rgb}{0.7,0.3,0}
\definecolor{blue}{rgb}{.2,.6,.75}
\definecolor{green}{rgb}{.4,.7,.4}

\numberwithin{equation}{section}

\newtheorem{thm}{Theorem}[section]
\newtheorem{lem}[thm]{Lemma}
\newtheorem{prop}[thm]{Proposition}

\theoremstyle{definition}
\newtheorem{dfn}[thm]{Definition}

\theoremstyle{remark}
\newtheorem{rem}{Remark}[section]
\newtheorem*{rem*}{Remark}

\newcommand{\N}{\mathbb{N}}
\newcommand{\Z}{\mathbb{Z}}
\newcommand{\Q}{\mathbb{Q}}
\newcommand{\R}{\mathbb{R}}
\newcommand{\C}{\mathbb{C}}

\newcommand{\F}{\mathbb{F}}

\newcommand{\CA}{\mathcal{A}}
\newcommand{\SA}{\mathscr{A}}

\newcommand{\CC}{\mathcal{C}}
\newcommand{\CD}{\mathcal{D}}
\newcommand{\CE}{\mathcal{E}}

\newcommand{\CI}{\mathcal{I}}
\newcommand{\CJ}{\mathcal{J}}

\newcommand{\CL}{\mathcal{L}}
\newcommand{\CM}{\mathcal{M}}
\newcommand{\CN}{\mathcal{N}}

\newcommand{\CP}{\mathcal{P}}

\newcommand{\CS}{\mathcal{S}}

\newcommand{\CX}{\mathcal{X}}

\newcommand{\CZ}{\mathcal{Z}}

\newcommand{\abs}[1]{\left\lvert #1 \right\rvert}

\newcommand{\ds}{\displaystyle}
\newcommand{\bs}\boldsymbol{}

\newcommand{\eq}[2]{ \begin{equation} \label{#1}\begin{split} #2 \end{split} \end{equation} }
\newcommand{\al}[1]{\begin{align} #1 \end{align} }
\newcommand{\als}[1]{\begin{align*} #1 \end{align*} }

\newcommand{\nn}{\nonumber \\}

\newcommand{\dee}{\mathrm{d}}

\DeclareMathOperator{\Span}{Span}
\DeclareMathOperator{\vol}{vol}

\newcommand{\fl}[1]{\left\lfloor#1\right\rfloor}

\renewcommand{\hat}{\widehat}
\renewcommand{\tilde}{\widetilde}

\renewcommand{\phi}{\varphi}

\renewcommand{\Re}{{\rm Re}}
\renewcommand{\Im}{{\rm Im}}

\renewcommand{\mod}[1]{\,({\rm mod}\,#1)}

\definecolor{red}{rgb}{1,0,0}

\definecolor{orange}{rgb}{0.7,0.3,0}

\definecolor{blue}{rgb}{.2,.6,.75}

\definecolor{green}{rgb}{.4,.7,.4}

\makeatletter
\newenvironment{abstracts}{%
	\ifx\maketitle\relax
	\ClassWarning{\@classname}{Abstract should precede
		\protect\maketitle\space in AMS document classes; reported}%
	\fi
	\global\setbox\abstractbox=\vtop \bgroup
	\normalfont\Small
	\list{}{\labelwidth\z@
		\leftmargin3pc \rightmargin\leftmargin
		\listparindent\normalparindent \itemindent\z@
		\parsep\z@ \@plus\p@
		
		\itemsep\medskipamount
	}%
}{%
	\endlist\egroup
	\ifx\@setabstract\relax \@setabstracta \fi
}

\newcommand{\abstractin}[1]{%
	\otherlanguage{#1}%
	\item[\hskip\labelsep\scshape\abstractname.]%
}
\makeatother

\begin{document}
	
	\title{Sieve weights and their smoothings  \\\hfill \\Poids de crible et leurs lissages}
	
	\author[A. Granville]{Andrew Granville}
	\address{AG: D\'epartement de math\'ematiques et de statistique\\
		Universit\'e de Montr\'eal\\
		CP 6128 succ. Centre-Ville\\
		Montr\'eal, QC H3C 3J7\\
		Canada; 
		and Department of Mathematics \\
		University College London \\
		Gower Street \\
		London WC1E 6BT \\
		England.
	}
	\email{{\tt andrew@dms.umontreal.ca}}
	
	\author[D. Koukoulopoulos]{Dimitris Koukoulopoulos}
	\address{DK: D\'epartement de math\'ematiques et de statistique\\
		Universit\'e de Montr\'eal\\
		CP 6128 succ. Centre-Ville\\
		Montr\'eal, QC H3C 3J7\\
		Canada}
	\email{{\tt koukoulo@dms.umontreal.ca}}
	
	\author[J. Maynard]{James Maynard}
	\address{JM: Mathematical Institute, Radcliffe Observatory quarter, Woodstock Road, Oxford OX2 6GG, England}
	\email{{\tt james.alexander.maynard@gmail.com}}

	
	\date{\today}

	\begin{abstracts}
		\abstractin{french}
		
		On obtient des formules asymptotiques pour les $2k$-i\`emes moments de quelques sommes partiellement liss\'ees de la fonction de M\"obius sur les diviseurs d'un entier. Quand $2k$ est petit en comparaison avec $A$, qui est le niveau de lissage, alors la contribution principale aux moments provient des entiers ayant que de grands facteurs premiers, comme on l'esp\'erait pour un poids de crible. Cependant, si $2k$ est plus grand en comparaison avec $A$, alors la contribution principale aux moments provient des entiers ayant beaucoup de facteurs premiers, ce qui n'est pas l'intention quand on cr\'ee des poids de crible. La valeur seuil pour ``petit'' est $A=\frac 1{2k} \binom{2k}{k}-1$.
		
		On peut aussi poser des questions analogues pour les polyn\^omes sur des corps finis et pour les permutations, et dans ces cas les moments se comportent de fa\c con assez diff\'erente, avec moins d'annulations dans les sommes de diviseurs. On donne, on esp\`ere, une explication plausible pour ce ph\'enom\`ene, en \'etudiant les sommes analogues pour les caract\`eres de Dirichlet, et en obtenant chaque type de comportement selon si le caract\`ere est ``exceptionnel'' ou non.
		\\
		\abstractin{English}
		We obtain asymptotic formulas for the $2k$th moments of partially smoothed divisor sums of the M\"obius function.
		When $2k$ is small compared with $A$, the level of smoothing,  then the main contribution to the moments come  from integers with only large  prime factors, as one would hope for in   sieve weights.  However if $2k$ is any larger, compared with $A$, then the main contribution to the moments come  from integers with quite a few prime factors, which is not the intention when designing sieve weights. The   threshold for ``small'' occurs when $A=\frac 1{2k} \binom{2k}{k}-1$.  
		
		One can ask  analogous questions for polynomials over finite fields and for permutations, and in these cases the moments  behave rather differently, with even less cancellation in the divisor sums. We give, we hope, a plausible explanation for this phenomenon, by studying the analogous sums for Dirichlet characters, and obtaining each type of behaviour depending on whether or not the character is ``exceptional''.
	\end{abstracts}

\thanks{We are grateful to the referee of the paper for an extraordinarily thorough reading of the paper that improved significantly the level of exposition. Our thanks also go to Maksym Radziwill for supplying several key references, and to him, Ben Green and Henryk Iwaniec for helpful discussions. AG is is funded by the National Science and Engineering Research Council of Canada and by the European Research Council. DK is funded by the National Science and Engineering Research Council of Canada and by the Fonds de recherche du Qu\'ebec -- Nature et technologies. JM is funded by a Clay research fellowship and a fellowship by examination of Magdalen College, Oxford.}

\maketitle

\setcounter{tocdepth}{2}
\tableofcontents

\section{Introduction}

Sieve methods are a set of techniques which give upper and lower bounds for the number of elements of a set of integers $\mathcal{A}$ which have no `small' prime factors. Their key benefit is that they are very flexible - one can obtain bounds of the correct order of magnitude for many interesting sets $\mathcal{A}$, even though obtaining asymptotic formulae looks completely hopeless. In particular, they are typically very effective at obtaining upper bounds for the number of primes in sets $\mathcal{A}$ of interest which are only worse than the conjectured truth by a constant factor.

One of the most important sieves in the Selberg sieve. Selberg's approach \cite{Sel} starts with the inequality
\eq{Selberg}{
\Bigl( \sum_{\substack{ d|n \\ P^+(d)\le z }} \lambda_d\Bigr)^2 \ge \sum_{\substack{d|n\\ P^+(d)\le z}} \mu(d)=\begin{cases}
1,\qquad & P^-(n)\ge z,\\
0,&\text{otherwise,}
\end{cases}
}
which is valid for any real numbers $\lambda_d$ with $\lambda_1=1$. Here $P^+(n)$ and $P^-(n)$ are the largest and smallest prime factors of $n$ respectively. 
Summing \eqref{Selberg} over $n\in\mathcal{A}$ gives
\[
\#\{n\in\mathcal{A}:P^-(n)\ge z\}\le \sum_{n\in\CA}  \Bigl( \sum_{\substack{d|n \\ P^+(d)\le z}} \lambda_d\Bigr)^2 
	=  \sum_{P^+(d_1),P^+(d_2)\le z} \lambda_{d_1}\lambda_{d_2} \cdot \#\{n\in\CA:[d_1,d_2]|n\},
\]
which is a quadratic form in the variables $\lambda_d$. Provided $d_1$ and $d_2$ are not too large, say at most $R$, one can hope to get a reasonable estimate for the coefficients $\#\{n\in\CA:[d_1,d_2]|n\}$ of this quadratic form. The best upper bound stemming from this method then comes from minimizing the quadratic form over all choices of $\lambda_d\in \mathbb{R}$ with $\lambda_1=1$ and $\lambda_d=0$ for $d>R$.

For typical sets $\CA$ that arise in arithmetic problems, one finds that the optimal choice for the $\lambda_d$ takes the form
\[
\lambda_d \approx \mu(d) \cdot \left(\frac{\log(R/d)}{\log R}\right)^A \quad (d\le R),
\]
where $A$ is some positive constant. We note that the weights $\lambda_d$ decay to 0, and the larger the value of $A$, the higher the level of smoothness at the truncation point $R$. In the optimal choice, the exponent $A$ is taken to be $\kappa$, the {\sl dimension} of the sieve problem. However, for a given dimension $\kappa$, it is known \cite[pg. 154]{SelCollected} that any exponent $A>\kappa-1/2$ yields weights $\lambda_d$ whose dominant contribution comes from numbers almost coprime to $m$, whereas this fails to be true for smaller $A$. See \cite[ch. 10]{opera} for further discussion.

More generally, one can consider the smoothed sieve weight
\[
M_f(n;R):= \sum_{d|n} \mu(d) f\left( \frac{\log d}{\log R}\right) ,
\]
where $f:\R\to\R$ is a function supported on $(-\infty,1]$, which corresponds to taking $\lambda_d=\mu(d)f(\log{d}/\log{R})$ for $d\le R$. In Selberg sieve arguments one typically chooses $f$ to be a polynomial in $[0,1]$, perhaps of high degree. Such an example is offered by the `GPY sieve' of Goldston-Pintz-Y\i ld\i r\i m \cite{GPY, Zhang}. In more recent  developments on gaps between primes by the third author \cite{maynard} and Tao \cite{tao} one works with general smooth functions $f$.

The main motivation of this paper is to understand the exact role of the smoothing in the structure of the Selberg sieve weights. To this end, we consider their moments
\[
\sum_{n\le x} M_f(n;R)^k
\]
as a tool of gaining additional insight on the distribution of the values of $M_f(n;R)$. On the practical side, higher moments naturally appear when applying H\"older's inequality, so it would be useful to know their behaviour\footnote{For example, Lemma 3.5 in Pollack's paper \cite{Pollack} is an example of a case where a fourth moment occurs because of the use of Cauchy's inequality, and a similar issue is encountered in Friedlander's work \cite{Friedlander} for the combinatorial sieve instead of the Selberg sieve.}.

From the discussion above, in the case $f(x)=\max(1-x,0)^A$ and $k=2$, we have seen that if $A$ is sufficiently large, then $M_f(n;R)^2$ `behaves like a sieve weight' in the sense that the sum $\sum_{n<x}M_f(n;R)^2$ is $O_f(x/\log{R})$ and the main contribution to this comes from numbers with few prime factors less than $R$. If $A$ is too small and so $f$ is not smooth enough, however, then $M_f(n;R)$ exhibits qualitatively different behavior; the sum is larger than $x/\log{R}$, and the main contribution is no longer from numbers with few prime factors $\le R$. 

How smooth should $f$ be so that $M_f(n;R)^{2k}$ behaves like a sieve weight when $k$ varies, that is to say the main contribution to the $2k$-th moment\footnote{We are typically interested in how large sieve weights get. If we took odd powers there might be an irrelevant cancellation, so we focus on even moments.} of $M_f(n;R)$ comes from integers $a$ that have very few prime factors $\le R$? What happens in the extreme case where $f$ is the discontinuous function ${\bf 1}_{(-\infty,1]}$? These are the types of questions that we will study in this paper.

\subsection{Some smoothing is necessary to behave like a sieve weight}

In order to gain a first understanding of the importance of smoothing, let us consider the sharp cut-off function
\[
f_0:= {\bf 1}_{(-\infty,1]} .
\]
If $n=2m$ with $m$ odd, then we have that
\eq{dyadic}{
M_{f_0}(n;R) 
	= \sum_{\substack{d|n \\ d\le R}}\mu(d)  
	=  \sum_{\substack{d|m \\ d\le R}}\mu(d) 
	+ \sum_{\substack{d|m \\ 2d\le R}}\mu(2d) 
	= \sum_{\substack{d|m \\ R/2<d\le R}}\mu(d) 
	= M_{\tilde{f_0}}(m;R),
}
where, with a slight abuse of notation, we have put
\eq{f_0}{
\tilde{f_0} := {\bf 1}_{(1-\log2/\log R,1]} .
}
In particular, if $m$ is square-free and has exactly one divisor $d\in(R/2,R]$, then $M_{f_0}(n;R)=\pm1$. An easy generalization of a deep result of Ford \cite[Theorem 4]{Ford} implies that\footnote{The key estimates in the proof of the lower bound of Theorem 4 in \cite{Ford} are the second part of Lemma 4.1, Lemma 4.3 (the parameters are $z=R\ll R/2=y$), Lemma 4.5, Lemma 4.8 and Lemma 4.9. A key observation is that only square-free integers are considered in Lemma 4.8, so that a stronger version of the lower bound of Theorem 4 of \cite{Ford} can be immediately deduced by the same proof, that counts square-free integers with exactly one divisor in $(R/2,R]$.}  the proportion of such $m\le x/2$ is $\gg (\log R)^{-\delta}(\log\log R)^{-3/2}$ with
\[
\delta= 1-\frac{1+\log\log 2}{\log 2}=0.086071332\ldots ,
\]
whence we conclude that
\[
\#\{  n\leq x:\  M_{f_0}(n;R)  \neq 0 \}	
	\gg \frac{x}{(\log R)^{\delta}(\log\log R)^{3/2}}  \quad 
		(x \ge R^{1+\epsilon} )  .
\]
In particular, we find that $M_{f_0}(n;R)$ is non-zero too often to behave like a sieve weight. This indicates that part of the importance of smoothing is to reduce the contribution of isolated divisors of $n$ to $M_f(n;R)$.

We will prove in Section \ref{support} that
\eq{NumberOfnonZeros}{
\#\{  n\leq x:\  M_{f_0}(n;R)  \neq 0 \}	
	\asymp \frac{x}{(\log R)^{\delta}(\log\log R)^{3/2}}  \quad (x \ge R^{5/2} )  .
}
This sharpens a result by Hall and Tenenbaum \cite{divisors}, who used a very similar argument and the best results about divisors of integers available at that time.

\subsection{A heuristic argument}
Going back to the study of $M_f(n;R)$ for a smooth function $f$, it is reasonable to believe that the smoother $f$ is, the larger the $k$ are for which $M_f(n;R)^{2k}$ behaves like a sieve weight. One way to explain this phenomenon is by noticing that various integral transformations have faster decay for smooth weights, which can help to tame the arithmetic issues at play. (See, for example, Section \ref{mellin}.) Nevertheless, we prefer to give a number theoretic explanation in terms of the underlying sieve questions rather than an analytic one focused more on the technical issues. Assume that $n=p_1^{\alpha_1}\cdots p_r^{\alpha_r}m$, where $p_1<\cdots<p_r$, $\alpha_i\ge 1$ and all of the prime divisors of $m$ are $>p_r$. Then
\eq{differencing1}{
M_f(n;R) &= \sum_{d|p_2\cdots p_rm} \mu(d) f\left( \frac{\log d}{\log R}\right) 
			+ \sum_{d|p_2\cdots p_rm} \mu(p_1d) f\left( \frac{\log(p_1d)}{\log R}\right)  \\
	&= \sum_{d|p_2\cdots p_rm} \mu(d) \left\{ f\left( \frac{\log d}{\log R}\right) 
			-  f\left( \frac{\log p_1}{\log R} + \frac{\log d}{\log R}\right) \right\} .
}
Continuing as above, we find that
\eq{differencing2}{
M_f(n;R) &= (-1)^r 
	\sum_{d|m} \mu(d) \Delta^{(r)} f\left( \frac{\log d}{\log R} ; \frac{\log p_1}{\log R},\dots,\frac{\log p_r}{\log R} \right) ,
}
where $\Delta^{(r)} f(x;h_1,\dots,h_r)$ denotes the multi-difference operator defined by
\[
\Delta^{(1)} f(x;h) = f(x+h)-f(x)
\]
and
\[
\Delta^{(r)} f(x;h_1,\ldots,h_r)
	= \Delta^{(r-1)} (x+h_r;h_1,\dots,h_{r-1}) - \Delta^{(r-1)}(x;h_1,\dots,h_{r-1}) .
\]
In particular, if $f \in C^r(\R)$, then
\eq{differencing3}{
\Delta^{(r)} f(x; h_1,\ldots,h_r) = 
	 \int_0^{h_r} \int_0^{h_{r-1}}\ldots \int_0^{h_1}  f^{(r)}(x+t_1+t_2+\ldots +t_r)dt_1\ldots dt_r,
}
Returning to \eqref{differencing2}, we see that if $f\in C^A(\R)$ and $n=p_1^{\alpha_1}\cdots p_r^{\alpha_r}m$, $r\le A$, is as above, then $M_f(n;R)$ should heuristically be $\ll M_{f^{(r)}}(m;R)\prod_{j=1}^r(\log p_j/\log R)$. Loosely, this indicates each additional degree of smoothness of the weight function $f$ cuts the average size of $M_f(n;R)$ by about a factor of $1/\log R$. 

The above discussion leads us to conjecture that if $f\in C^A(\R)$ with $f(0)\neq0$, then 
\eq{conj}{
\sum_{n\le x} M_f(n;R)^{2k}\ll 
	\max\left\{ \frac{x}{\log R} , \frac{1}{(\log R)^{2kA}} \sum_{n\le x} M_{f_0}(n;R)^{2k} \right\} .
}
Notice that the factor $x/\log R$ is necessary because $M_f(n;R)=f(0)$ for all integers $n$ that are free of prime factors $\le R$. 

Naturally, for this relation to be useful, we need to understand the asymptotics of $\sum_{n\le x} M_{f_0}(n;R)^{2k}$. Recall the relation \eqref{smoothedmain}. Expanding the $k$-th power and swapping the order of summation, we find that
\[
\sum_{n\leq x} M_f(n;R)^k
	= x\cdot \CM_{f,k}(R) +O((\|f\|_\infty R)^k) 
\]
for any $f:\R\to\R$ supported on $(-\infty,1]$, where
\eq{smoothedmain}{
\CM_{f,k}(R) := 
	\sum_{d_1,\dots,d_k\ge1} \frac{\prod_{j=1}^k \mu(d_j)f(\log d_j/\log R)}{[d_1,\dots,d_k]} 
	= \prod_{p\le R}\left(1-\frac{1}{p}\right) \sum_{p|n\, \Rightarrow\, p\le R} \frac{M_f(n;R)^k}{n} .
}
We are generally interested in the situation when $R$ is bounded by a small power of $x$, so that the error term $O((\|f\|_\infty R)^k)$ is negligible. Thus our focus is on the main term $\CM_{f,k}(R)$, which no longer depends on $x$. When $k=1$, Dress, Iwaniec and Tenenbaum \cite{DIT} showed
\eq{DDT}{
\CM_{f_0,2}(R) \sim c_1\quad(R\to\infty)
}
for some constant $c_1>0$, and when $k=2$, Motohashi \cite{Moto} showed that
\eq{motohashi}{
\CM_{f_0,4}(R)  \sim c_2(\log R)^2 \quad(R\to\infty) 
}
for some constant $c_2>0$. In general, Balazard,  Naimi, and P\'etermann \cite{BNP} proved that 
\[
\CM_{f_0,2k}(R)=P_k(\log R) + O(e^{-c(\log R)^{3/5}(\log\log R)^{-1/5}}) ,
\] 
for some polynomial $P_k$ and some constant $c=c(k)>0$. This built on work of de la Bret\`eche \cite{delaB}, who showed how a wide class of related sums can be evaluated asymptotically. However, when applying his technique to this question, one would need some strong understanding of the growth of $\zeta(s)$ near to $s=1$ to recover the result of  \cite{BNP} (which, for example, follows from the Riemann Hypothesis). 

Notice that if $\CE_k=\deg(P_k)$, so that $\CE_1=0$ and $\CE_2=2$, then \eqref{conj} becomes
\eq{conj2}{
\sum_{n\le x} M_f(n;R)^{2k} \ll 
	\max\left\{ \frac{x}{\log R} , x(\log R)^{\CE_k-2kA}  \right\} .
}
This suggests that $M_f(n;R)^{2k}$ acts like a sieve weight as long as $A>\CE_k/2k$. The big issue with the result of Balazard, Naimi and P\'etermann is that the degree $\CE_k$ is not determined for general $k$, and that is essential if one wishes to gain a better understanding of how the Selberg sieve weights work. Our attention thus turns to calculating $\CE_k$. 

But first, we study seemingly analogous questions (in different settings), that one might guess would be easier and indicate what kind of estimate we should be looking for

\subsection{Analogous settings}

It is well-known that many of the analytic properties of integers are shared by both polynomials of finite fields (c.f. \cite{Rosenbook}), and by permutations (c.f. \cite{Gra}). Moreover, polynomials and permutations are usually easier objects to understand, so in order to gain an understanding of the exponent $\CE_k$, it would be natural to consider what happens in these analogous settings first.

\medskip

\noindent {\it Permutations.} 
The easiest analogy to analyze concerns permutations. Every $\sigma \in S_N$ (the permutations on $N$ letters) can be decomposed in a unique way into a product of disjoint cycles. Those cycles cannot be decomposed any further and play the role of irreducibles. Divisors of $\sigma$ are precisely the set of possible products of cycles. If those cycles act on the subset $T$ of $[N]$, then $\sigma$ fixes $T$. Moreover, if $\sigma$ fixes $T$, then $\sigma$ is a product of cycles, a subproduct of which fixes $T$. Hence ``divisors'' correspond to sets $T\subset [N]$ for which $\sigma(T)=T$

To ``calibrate'' our understandings of the properties of integers and permutations, we note that for a typical integer $n$, its $j$-th largest prime factor is about $e^{e^j}$, whereas for a typical permutation $\sigma\in S_N$, its $j$-th largest cycle has length about $e^j$. Thus, the inequality $R/2<d\le R$ for a divisor $d$ of $n$ corresponds to having a set $T$ that is fixed by $\sigma$ of size $\#T=\log R+O(1)$. Hence we will study
\eq{perm-def}{
\text{\rm Perm}(N,m;k) :=\frac 1{N!} \sum_{\sigma\in S_N} \Bigl| \sum_{\substack{\ T\subset  [N],\ \#T=m,\\ \sigma(T)=T  }}\mu(\sigma\big|_T)\Bigr|^{2k} ,
}
where 
\[
[N]:=\{1,\dots,N\},
\]
and if $\sigma\big|_T=C_1C_2\ldots C_\ell$ is the product of $\ell$ disjoint cycles, then we have set $\mu(\sigma\big|_T)=(-1)^\ell$. We claim that $\text{\rm Perm}(N,m;k)$ is more natural than it appears at first sight. A usual function of permutations is the signature $\epsilon(\sigma)$ which counts the number of transpositions (i.e.~the  number of interchanges of two elements) needed to create $\sigma$.  For a cycle $C$, one knows that $\epsilon(C)=(-1)^{\#C-1}$ and hence $\epsilon(\sigma\big|_T)=\epsilon(C_1)\epsilon(C_2)\ldots \epsilon(C_\ell)=(-1)^{\#T-\ell}=(-1)^m \mu(\sigma|_T)$, since $\#T=m$ here. Therefore
\[ 
 \sum_{\substack{\sigma\in S_N,\ T\subset  [N],\\ \#T=m,\ \sigma(T)=T  }}\mu(\sigma\big|_T)
 =(-1)^m \sum_{\substack{\sigma\in S_N,\ T\subset  [N],\\ \#T=m,\ \sigma(T)=T  }}\epsilon(\sigma\big|_T),
 \]
 whence
\[
\text{\rm Perm}(N,m;k) =\frac 1{N!} \sum_{\sigma\in S_N} \Bigl| \sum_{\substack{\ T\subset  [N],\ \#T=m,\\ \sigma(T)=T  }}\epsilon(\sigma\big|_T)\Bigr|^{2k} .
\]
Arguing as in the work of Eberhard, Ford and Green \cite{EFG} that establishes the analogue for permutations of Ford's results \cite{Ford} for integers, it is possible to show that the summands on the right hand of \eqref{perm-def} (and, hence, of the above formula) are non-zero for a proportion $\asymp 1/m^\delta(\log m)^{3/2}$ of the permutations in $S_N$. The following theorem provides a formula and an asymptotic estimate for $\text{\rm Perm}(N,m;k)$.

\begin{thm} \label{thm-perm} For each integer $k\geq 1$ and each integer $m\geq 1$, if $N\geq 2mk$ then
\[
\text{\rm Perm}(N,m;k) = c(m,k) ,
\]
where $c(m,k)$ is the number of $(2^{2k}-1)$-tuples $(r_I)_{\emptyset\neq I\subset\{1,\dots,2k\}}$ of non-negative integers such that 	
\begin{itemize}
\item $r_I \in\{0,1\}$ for $\#I$ odd;
\item $\sum_{I:\ i\in I}\ r_I = m$, for each $i\in\{1,\dots,2k\}$.
\end{itemize}
Moreover, for fixed $k\in\Z_{\ge 1}$, the function $c(m,k)$ is increasing in $m$ and satisfies the estimate
\[
  c(m,k)  \asymp_k m^{2^{2k-1}-2k-1} +1.
\]
\end{thm}

\begin{proof}[Proof of the formula for $\text{\rm Perm}(N,m;k)$]
Given sets $T_1,\dots,T_{2k}$, the sets
\[
R_I := \Bigl(\bigcap_{i\in I} T_i \Bigr)\setminus \Bigl( \bigcup_{i\in[2k]\setminus I}T_i\Bigr)
	\quad(I\subset[2k]) 
\]
form a partition of $[N]$; that is to say $[N]$ equals $\sqcup_I R_I$, the disjoint union of the sets $R_I$. Using this with $T_1,\dots T_{2k}$ fixed sets of $\sigma$ (i.e. $\sigma(T_i)=T_i$, so the $R_I$ are all fixed by $\sigma$ as well), we find
\[
\frac 1{N!} \sum_{\sigma\in S_N} \Bigl| \sum_{\substack{\ T\subset  [N],\ \#T=m,\\ \sigma(T)=T  }}\epsilon(\sigma\big|_T)\Bigr|^{2k} =
\sum_{\substack{r_I\geq 0\ \forall I \\ \sum_{I:\ i\in I}\ r_I = m }} \quad
\sum_{\substack{ [N]=\sqcup_I R_I \\ \#R_I=r_I\  \forall I }}
\frac 1{N!}
\prod_{I\subset [2k]}\Bigl( \sum_{\rho_I\in S_{r_I}} \epsilon(\rho_I)^{\#I}\Bigr)  .
\]
The inner sums are each $r_I!$ unless $\#I$ is odd and $r_I>1$, in which case we get $0$. Additionally, we get that the number of choices of sets of the given sizes is $N!/\prod_I r_I!$, and hence the above equals $c(m,k)$.

The bounds for $c(m,k)$ will be proven in Section \ref{permsect}.
\end{proof} 

Evidently, the above results suggest that $\CE_k=\max\{0,2^{2k-1}-2k-1\}$. Relation \eqref{DDT} implies that $\CE_1=0$, but relation \eqref{motohashi} implies that $\CE_2=2\neq 2^3-5$. This suggests that there is a discrepancy between the integer and the permutation setting, a very rare exception.

\medskip

\noindent
{\it Polynomials over finite fields}.
Positive integers are uniquely identifiable by their factorization into primes (the Fundamental Theorem of Arithmetic). Note though that every non-zero integer equals a unit (that is $1$ or $-1$) times one of those positive integers.  We will work with polynomials in $\mathbb F_q[t]$. Monic polynomials in $\mathbb F_q[t]$ are uniquely identifiable by their factorization into monic irreducible polynomials of degree $\geq 1$. Again, note that every non-zero polynomial in $\mathbb F_q[t]$ equals a unit (that is, any element $a\in \F_q\setminus\{0\}$) times a monic polynomial. We will work only with monic polynomials, for example when considering divisors of a given polynomial (rather like we only consider positive integer divisors of a given integer). The M\"obius function of a given polynomial is a multiplicative function, where $\mu(P)=-1$, and $\mu(P^k)=0$ if $k\geq 2$, whenever $P$ is irreducible. 

To ``calibrate'' our understandings of the arithmetic properties of integers and polynomials, we note that $\sim 1/\log x$ of integers around $x$ are prime, whereas $\sim 1/m$ of monic polynomials of degree $m$ are irreducible in $\mathbb F_q[t]$. Here the ``$\sim$'' symbol means as $q\to\infty$ running through prime powers. Thus, wherever we see $\log x$ in an estimate about the integers, we try to replace it with $m$ in an  estimate about degree $m$ polynomials.  Similarly  a divisor $d$ of $n$ that is close to $R$ is analogous to a polynomial divisor of $F(t)$ of degree $m$, where $m$ replaces $\log R$ in estimates. Hence we will study
\[
\text{Poly}_q(n,m;k):= \frac 1{q^n} \sum_{\substack{\text{ monic } N\in \mathbb F_q[t]\\ \deg N=n}} \Bigl|  \sum_{\substack{\text{monic } M|N  \\ \deg M=m}} \mu(M) \Bigr|^{2k},
\]
Here we have divided by $q^n$ because this is how many monic polynomials $N$ of degree $n$ are contained in $\F_q[t]$, which is the analogue of 
\[
\frac{1}{x} \sum_{n\le x} \Bigl( \sum_{\substack{d|n \\ R/2<d\le R}} \mu(d)\Bigr)^{2k} ,
\]
a quantity directly related to $\frac{1}{x} \sum_{n\le x} M_{f_0}(n;R)^{2k}$ via \eqref{dyadic}. We will prove below the following estimate:

\begin{thm} \label{thm-poly} For integers $k,m\ge1$ and $n\ge2mk$, we have that
\[
\text{\rm Poly}_q(n,m;k) = c(m,k) (1+O_k(1/q)) \asymp_k 1+m^{2^{2k-1}-2k-1} .
\]
\end{thm}

We thus see that polynomials behave similarly to permutations (and thus differently than integers).


\subsection{Two worlds apart and a bridge between them}

Our discussion of the permutation and polynomial analogues, rather than shedding more light on the value of the exponent $\CE_k$, gave rise to even more questions. It turns out that the integer setting is substantially more complicated than the permutation and polynomial settings. We now state our main results about integers. First, given $A\in\Z_{\ge1}$, we let 
\[
f_A(t):= 
	\begin{cases} 
		(1-t)^A 		& \text{for } t\le 1; \\
		0 			& \text{otherwise} ,
	\end{cases}
\]
an extension of the definition of $f_0$. Note that $f_A \in C^{A-1}(\R)\setminus C^A(\R)$ for all $A\ge1$. We then have the following result that determines the value of $\CE_k$: 

\begin{thm} \label{mainthm}
For fixed integers $k\ge1$ and $A\ge0$, there is a constant $c_{k,A}>0$ such that
\eq{f-asymp}{
 \CM_{f_A,2k}(R) = c_{k,A} (\log R)^{\CE_{k,A}}  
 	+ O((\log R)^{\CE_{k,A}-1}) ,
}
where 
\[
\CE_{k,A} := \max\left\{ \binom{2k}{k}-2k(A+1),-1\right\} .
\]
In particular, $\CE_k=\CE_{k,0}=\binom{2k}{k}-2k$. Additionally, we find that there is a constant $c_k'>0$ such that for $R^{2k}\le x$ we have
\eq{mainthm-dyadic}{
\frac{1}{x}\sum_{n\leq x} \Bigl(\sum_{\substack{d|n \\ R/2<d\leq R}}\mu(d)\Bigr)^{2k}
	= c_k' (\log R)^{\binom{2k}k -2k} + O\left((\log R)^{\binom{2k}k -2k-1}\right) .
}
All implied constants depend at most on $k$ and $A$.
\end{thm}

\begin{rem}
We have no nice formula for the constants $c_{k,A}$ and $c_k'$ appearing in Theorem \ref{mainthm}; we only know how to write them as an enormous rational linear combination of complicated integrals, and leave it as a challenge to come up with an easy explicit description. 
\end{rem}

\begin{rem}
If the moments of a distribution grow slowly, then the distribution can be determined via its Laplace transform. However, in our case the moments are of rapidly increasing magnitude, indeed with different powers of $\log R$, so one cannot immediately deduce from them the distribution of the weights $M_{f_A}(n;R)$ as $n$ varies over the integers. 
\end{rem}

\begin{rem}
In this paper we only consider integral $A$, but we would expect analogous results to hold for all real $A>0$.
\end{rem}

\begin{rem}
In this paper we only consider Selberg-style sieve weights. We would expect something somewhat analogous to hold for combinatorial-style sieve weights (such as those used in the $\beta$-sieve) but we do not consider such situations here.
\end{rem}
For general functions $f$, we prove that $M_f(n;R)^{2k}$ behaves like a sieve weight as long $f\in C^A(\R)$ with $A>\binom{2k}{k}/2k=\CE_{2k}/2k+1$. Notice that this confirms a weak version of the heuristic estimate \eqref{conj2}.

\begin{thm} \label{thm-smooth}
Let $k\in\Z_{\ge 1}$, $\epsilon\in(0,1)$ and $f:\R\to\R$ be supported in $(-\infty,1]$. Assume further that for some integer $A\ge2$, $f\in C^A(\R)$ and that all functions $f,f',\dots,f^{(A)}$ are bounded.
\begin{enumerate}
\item If $A>\frac{1}{2k} \binom{2k}{k}$, then for $x\ge R\ge 2$ and $1\ge \eta \ge \log2/\log R$, we have that
\[
\sum_{\substack{n\le x \\ \exists p|n,\ p\le R^\eta }}  M_f(n;R)^{2k}
	\ll \frac{\eta x}{\log{R}}.
\]
If, in addition, $f(0)\neq0$, then there is a constant $c_{k,f}>0$ such that for $x\ge R^{2k}\log^2 R$ we have that
\[
\frac{1}{x}\sum_{n<x} M_f(n;R)^{2k}
	= \frac{c_{k,f}}{\log R} + O\left(\frac{1}{(\log R)^{2-\epsilon}} \right) .
\]
\item If $A\le \frac{1}{2k} \binom{2k}{k}$, then for $x \ge R\ge 2$ we have that
\[
\sum_{n\le x} 
	M_f(n;R)^{2k} \ll  x(\log R)^{\binom{2k}{k} -2kA} .
\]
\end{enumerate}
All implied constants depend at most on $f$, $k$ and $\epsilon$.
\end{thm}

The value of $\CE_k=\binom{2k}{k}-2k$ given by Theorem \ref{mainthm} is significantly smaller than the exponent $2^{2k-1}-2k-1$ in the polynomial/permutation setting. So we see the usual analogy breaking down in quite a severe way, something surprising. We devote Section \ref{heuristics} to the analysis of this discrepancy. In particularly, we will see that the underlying reason is the relation
\eq{mobius-sign}{
 \sum_{\substack{\sigma\in S_N,\ T\subset  [N],\\ \#T=m,\ \sigma(T)=T  }}\mu(\sigma\big|_T)
	 =(-1)^m \sum_{\substack{\sigma\in S_N,\ T\subset  [N],\\ \#T=m,\ \sigma(T)=T  }}\epsilon(\sigma\big|_T)
}
that we saw before. Notice here that while $\mu(\rho)=-1$ for all cycles $\rho$, we have that $\epsilon(\rho)$ takes the values $\pm1$ with equal probability as $\rho$ ranges over cycles of all possible lengths. The simplest example of a multiplicative function over $\Z$ demonstrating this kind of behaviour is that of a real Dirichlet character. To this end, we consider
\[
\CX_{2k}(R) =
	\prod_{p\le R}\left(1-\frac{1}{p}\right)
	\sum_{P^+(n)\le R} \frac{1}{n} \Bigl( \sum_{\substack{ d|n \\ R/2<d\le R}} \chi(d) \Bigr)^{2k} ,
\]
which, as in \eqref{smoothedmain}, is the main term of
\[
\frac{1}{x}\sum_{n<x}\Bigl|\sum_{\substack{d|n\\ R/2<d\le R}}\chi(d)\Bigr|^{2k}.
\]
We then have the following theorem, which shows that it is possible to bridge the gap between the two worlds of integers and of permutations/polynomials.  
All implied constants below depend at most on $k$, and we have set
\[
\CS^+(2k) := \{I\subset\{1,2,\dots,2k\}: \#I\ \text{even}\}\setminus\{\emptyset\} .
\]
 
\begin{thm}\label{thm-chars}
Let $\chi\mod q$ be a real non-principal character and $k\in\Z_{\ge 1}$. 
\begin{enumerate}
\item If $k=1$, then 
\[
\CX_2(R) = \frac{1}{2\pi}\int_{-\infty}^\infty \frac{P(t,\chi)|L(1+it,\chi)|^2\sin^2(t(\log 2)/2)}{t^2} \dee t
	+ O\left(\frac{1}{(\log R)^{100}}\right),
\]
where $P(\cdot,\chi)$ is a real-valued Euler product whose factors are $1+O(1/p^2)$. In particular, $P(t,\chi)\asymp1$ for all $t$, uniformly in $\chi$.
\item
Assume that $k\ge2$. Let $V_k(m)$  be the Lebesgue volume in $\R^{2^{2k-1}-1}$ given by 
\[
V_k(m)=\vol\{(x_I)_{I\in\CS^+(2k)}: x_I\ge0,\ m-\log 2\le \sum_{I\ni i} x_I \le m\},
\]
and let $\mathfrak{S}_k(\chi)$ be the singular series
\[
\mathfrak{S}_k(\chi)
	=\prod_p \left(1-\frac{1}{p}\right)^{2^{2k-1}}f_p,
\]
where 
\[
f_p=\begin{cases} \sum_{j\ge0} (j+1)^{2k}/p^j,\qquad&\text{if }\chi(p)=1,\\
(1-1/p^2)^{-1},&\text{if }\chi(p)=-1,\\
(1-1/p)^{-1},&\text{if }p|q.
\end{cases}
\]
Then $V_k(m)\asymp_k m^{2^{2k-1}-2k-1}$, 	$\mathfrak{S}_k(\chi)\asymp_k L(1,\chi)^{2^{2k-1}}$, and
\[
\CX_{2k}(R) 
	=\mathfrak{S}_k(\chi)  \cdot V_k(\log R)
		 + O\left((\log R)^{2^{2k-1}-2k-2}(\log(q\log R))^{O(1)}\right) .
\]
\item Assume that $k\ge2$ and that $L(\beta,\chi)=0$ for some $\beta>1-1/(100\log q)$. If $Q=e^{1/(1-\beta)}$ and $e^{(\log q)^C}\le R\le Q$ for some large enough $C=C(k)$, then there is a constant $c_k(\chi)=(\log q)^{O(1)}$ such that
\[
\CX_{2k}(R) = 
	 c_k(\chi) (\log R)^{\binom{2k}{k}-2k} 
	 \left( 1+O \left(\frac{(\log(q\log R))^{O(1)}}{\log R}\right) \right) .
\]
\end{enumerate} 
\end{thm}
In the case of our polynomial and permutation models, we have an exponent of $2^{2k-1}-2k-1$ for the $2k^{th}$ moment with $k\ge2$, whilst over the integers we have an exponent $\binom{2k}{k}-2k$. We see that our Dirichlet character model interpolates between these two settings. If the Dirichlet $L$-function associated with the character has a zero very close to 1, then $\chi(p)=-1$ for many small primes $p$, and so by multiplicativity $\chi$ behaves similarly to $\mu$ (at least in appropriate ranges). This is represented by our exponent of $\binom{2k}{k}-2k$ in this case. On the other hand, $\chi$ is a periodic character, and if the $L$ function does not have a zero very close to 1, we see that we have an exponent $2^{2k-1}-2k-1$, matching the exponent of our polynomial and permutation models.  Notice that if $L(s,\chi)$ does have an exceptional zero, then the asymptotic of case (c) for $\CX_{2k}(R)$ holds for small $R$, and transitions into the asymptotic of case (b) as $R$ grows.

\medskip

\begin{rem}
Relation \eqref{mobius-sign} has a polynomial analogue whose consequences are worth exploring further. Given $I\subset\{ F\in \F_q[t] : \deg(F)=n\}$, we consider the sum
\[
\sum_{F\in  I}\mu(F) .
\]
For example, we could take $I=\{ F\in \F_q[t] : \deg(F)=n\}$, or $I=\{F\in \F_q[t] : \deg(F-F_0)\le h\}$ for some $F_0\in\F_q[t]$ of degree $n$ and for some integer $h\in[1,n-1]$, which can be seen as the polynomial analogue of a short interval. Then 
\[
\sum_{F\in  I}\mu(F) 
	= (-1)^n \sum_{F\in  I}\chi(F), 
\]
where $\chi(F)=(-1)^{\deg(F)}\mu(F)$, which is also a multiplicative function. However, we note that, even though $\mu(P)=-1$ for all irreducibles, we have that $\chi(P)=1$ for about half of the irreducibles $P$, and $\chi(P)=-1$ for the other half, that is to say $\chi$ behaves on average much more like a real Dirichlet character rather than the M\"obius function. 

This phenomenon is striking and sharply different than what happens over $\Z$, where there is a dichotomy between multiplicative functions that look like the M\"obius functions and other ones whose average prime value is 0, as is exemplified by Theorem \ref{thm-chars} (see, also, \cite{kou}).
\end{rem}

\subsection{Further analysis of truncated M\"obius divisor sums}
As we saw in Theorem \ref{thm-smooth}, if $f\in C^A(\R)$ with $A>\frac{1}{2k} \binom{2k}{k}=\CE_k/2k+1$, then $M_f(n;R)^{2k}$ behaves like a sieve weight. When $f=f_A$, we can be more precise: 

If $A> \frac{1}{2k} \binom{2k}{k}-1=\CE_k/2k$, then $\CE_{k,A}=-1$ in Theorem \ref{mainthm}, and so $\CM_{f_A,2k}(R)\asymp (\log{R})^{-1}$. Since integers $n\le x$ with no prime factors less than $R$ contribute a total $\gg x(\log{R})^{-1}$ to $\sum_{n\le x}M_{f_A}(n;R)^{2k}$, we see that $M_{f_A}(n;R)^{2k}$ is behaving like a sieve weight in this case.

If $A\leq \frac{1}{2k} \binom{2k}{k}-1=\CE_k/2k$, then $\CE_{k,A}\ge 0$, and so $\CM_{f_A,2k}(R)\gg 1$. In particular, $M_{f_A}(n;R)^{2k}$ no longer behaves like a sieve weight, and the main contribution is from numbers with several prime factors in $[1,R]$.

The following theorem illustrates further this distinction.

\begin{thm} \label{thm-factors}
Let $x\ge R\ge2$, $k\in\Z_{\ge 1}$ and $A\in\Z_{\ge0}$. Moreover, let $\Omega(n;R)$ denote the number of prime factors of $n$ in $[1,R]$, counted with multiplicity.
\begin{enumerate} 
\item If $A > \frac{1}{2k} \binom{2k}{k}-1$, then
\[
\sum_{\substack{n<x \\ \Omega(n;R)\ge C}}M_{f_A} (n;R)^{2k}\ll_{k,A} \frac{x}{C\log{R}}.
\]
\item If $A\le \frac{1}{2k} \binom{2k}{k}-1$ and $\epsilon>0$ is fixed, then there is a $\delta=\delta(\epsilon,k)>0$ such that
\[
\sum_{\substack{n<x \\ |\Omega(n;R)/\log\log{R}-\binom{2k}{k}|\ge \epsilon }} M_{f_A}(n;R)^{2k}
	\ll_{k,A}  x (\log{R})^{\binom{2k}{k}-2k(A+1)-\delta} .
\]
\end{enumerate}
\end{thm}

In other words, if $A > \frac{1}{2k} \binom{2k}{k}-1$, then the main contribution to the sum defining 
$ \CM_{f_A,2k}(R)$ comes from integers with a bounded number of prime factors $\leq R$; whereas if 
$A\le \frac{1}{2k} \binom{2k}{k}-1$,   then the main contribution to the sum   comes from integers with 
$\left( \binom{2k}{k} +o(1)\right) \log\log{R}$ 
prime factors $\leq R$.

Analogous results hold with $\Omega(n;R)$ replaced by the function $\#\{p|n:p\le R\}$. We note that typically one requires $x>R^{2k}$, as in Theorem \ref{mainthm}, to estimate a $2k^{th}$ moment of a sum of divisors of size at most $R$, but the estimates of Theorem \ref{thm-factors} hold in the much wider range $x\ge R$. We can show similar (but slightly weaker) results for general weights $f$:

\begin{thm} \label{thm-smooth-factors}
Let $k\in\Z_{\ge 1}$ and $f:\R\to\R$ be supported in $(-\infty,1]$. Assume further that $f\in C^A(\R)$ and that all functions $f,f',\dots,f^{(A)}$ are uniformly bounded for some integer $A\ge2$, and fix some $\epsilon\in(0,1)$.
\begin{enumerate}
\item Assume that $A>\frac{1}{2k} \binom{2k}{k}$. For $x\ge R\ge 2$ and $C\ge1$, we have that
\[
\sum_{\substack{n<x \\ \Omega(n;R)\ge C}}M_f(n;R)^{2k}\ll_{k,f} \frac{x}{C\log{R}}.
\]

\item If $A\le \frac{1}{2k} \binom{2k}{k}$ and $\epsilon>0$ is fixed, then there is a $\delta=\delta(\epsilon,k)>0$ such that 
\[
\sum_{\substack{n<x \\ |\Omega(n;R)/\log\log{R}-\binom{2k}{k}|\ge \epsilon}} M_f(n;R)^{2k} 
	\ll_{k,f,\epsilon}  x(\log R)^{\binom{2k}{k} -2kA-\delta} \quad (x\ge R\ge 2) .
\]
\end{enumerate}
\end{thm}


\subsection{Outline of the paper}
 
We start the paper in Section \ref{heuristics} with a discussion of the discrepancy between the exponent of $\log R$ in \eqref{mainthm-dyadic} and the exponent of $m$ in Theorem \ref{thm-perm}, which is surprising at first sight. 

Sections \ref{permsect} and \ref{polysect} study the analogies for permutations and polynomials over finite fields, respectively. These analogies are considerably easier to analyze than the integer case.

The main body of the paper is then dedicated to the study of moments of $M_{f_A}(n;R)$ over Sections \ref{support} -- \ref{anatomy}. Specifically, in Section \ref{support} we establish relation \eqref{NumberOfnonZeros} for the size of the support of $M_{f_0}(n;R)$, and in Section \ref{mellin} we study inversion formulas for our divisor sums $M_f(n;R)$ that will be essential when dealing with their moments.  The proof of Theorem \ref{mainthm} is separated over three sections: in Section \ref{combinatorial}, we establish certain combinatorial inequalities that will be instrumental in understanding the leading term in the asymptotics for $\CM_{f_A,2k}(R)$. Then, in Section \ref{contour} we establish Theorem \ref{mainthm} by a multidimensional contour shifting argument, except for showing the positivity of the constants $c_{k,A}$ and $c_k'$. The latter will be accomplished with a different argument in Section \ref{lb}. Section \ref{anatomy} contains an analysis of the anatomy of the integers that give the main contribution to moments of $M_f(n;R)$. Specifically, we prove Theorems \ref{thm-smooth}, \ref{thm-factors} and \ref{thm-smooth-factors} there.
 
Finally, in Sections \ref{nonexcepsec} and \ref{excepsec} we study the moments of the sum weighted by Dirichlet characters, and establish Theorem \ref{thm-chars}, first for non-exceptional Dirichlet characters (where the proof is similar to Theorem \ref{thm-poly}), and then for exceptional Dirichlet characters (where the proof is similar to Theorem \ref{mainthm}).


\subsection{Notation}\label{notation}
Given an integer $N\ge1$, we set throughout the paper
\[
[N] := \{1,2,\dots,N\},
\]
\[
\CS^+(N) := \{\emptyset\neq I\subset[N]  : \#I\ \text{even}\},
\quad
\CS^-(N) := \{I\subset[N]  : \#I\ \text{odd}\},
\]
\[
\CS(N) := \{I\subset[N] \} 
\quad\text{and}\quad
\CS^*(N) := \CS^+(N)\cup\CS^-(N) . 
\]
Also, we recall that, given a integer $n\ge1$, we write $P^+(n)$ and $P^-(n)$ for its largest and smallest prime divisors, respectively, with the convention that $P^+(1)=1$ and $P^-(1)=\infty$.

Finally, given $2k$ variables $s_1,\dots,s_{2k}$ and $I\subset[2k]$, we will use the notation $s_I=\sum_{i\in I} s_i$. 


\section{The discrepancy between integers and polynomials}\label{heuristics}


The goal of this section is to analyze in detail why we have such a different behaviour when considering integers vs. polynomials or permutations

\subsection{Integer setting}\label{heuristics-integers}
Assume that $k\ge2$. We mimic the proof of Theorem \ref{thm-perm}. Recall the definition of $\tilde{f}_0$ in \eqref{f_0}. Given square-free integers $d_1,\dots,d_{2k}$ and $I\subset\CS^*(2k)$, we let $D_I$ be the product of those primes $p$ that divide each of the $d_i$'s with $i\in I$ but do not divide $\prod_{i\in[2k]\setminus I}d_i$. Then the integers $D_I$ for $I\in\CS^*(2k)$ are pairwise coprime and $d_i=\prod_{I:\, i\in I} D_I$ for each $i$, so that
\als{
\CM_{\tilde{f}_0,2k}(R) 
	&= \sum_{R/2<d_1,\dots,d_{2k}<R}
		\frac{\mu(d_1)\dots\mu(d_{2k})}{[d_1,\dots,d_{2k}]}\\
	&= \sideset{}{^\flat}\sum_{\substack{D_I\ (I\in\CS^*(2k)) \\ R/2<\prod_{I\ni i}D_I \le R\ (1\le i\le 2k)}}
		\left(\prod_{I\in\CS^-(2k)}\frac{\mu(D_I)}{D_I}\right)
			\left(\prod_{I\in\CS^+(2k)} \frac{1}{D_I} \right),
}
where the notation $\sideset{}{^\flat}\sum$ means that the summation is running over squarefree and pairwise coprime integers $D_I$. Set $L=e^{(\log\log R)^3}$. The contribution of those tuples with $D_I>L$ for some $I$ odd to $\CM_{\tilde{f}_0,2k}(R)$ can be seen to be $\ll 1/e^{c(\log\log R)^{3/2}}$ for some $c=c(k)>0$, by the Prime Number Theorem. So assume that $D_I\le L$ for all $I$ odd. Then it is natural to write
\[
\CM_{\tilde{f}_0,2k}(R) \sim 
	\sideset{}{^\flat}\sum_{\substack{D_I\le L\ (I\in\CS^*(2k))\\ D=\prod_{I\in \CS^-(2k)} D_I }} \frac{\mu(D)}{D} \cdot 
		T_{2k}(R_1,\dots,R_{2k} ; D ),
\]
where $R_i= R/\prod_{I\ni i,\ I\in \CS^+(2k)} D_I$ and 
\[
T_{2k}(\bs R;a)	
	= \sideset{}{^\flat}\sum_{\substack{(D_I,a)=1\ (I\in \CS^+(2k))  \\ 
		R_i/2<\prod_{I\in\CS^+(2k):\, i\in I} D_I \le R_i \ (1\le i\le 2k) }}
					\prod_{I\in\CS^+(2k)} \frac{1}{D_I} .
\]
When $\log R_i=\log R+O(\log L) = \log R+O((\log\log R)^{3/2})$, as it is here, we should be expecting that $T_{2k}(\bs R;a)$ has an asymptotic formula of the form
\[
T_{2k}(\bs R;a) = g(a) (\log R)^{2^{2k-1}-2k-1} + O\left((\log R)^{2^{2k-1}-2k-2} (\log\log R)^{O(1)} \right),
\]
since we have $2^{2k-1}-1$ variables on a logarithmic scale and $2k$ multiplicative constraints in dyadic intervals, where $g(a)$ is a multiplicative function with $g(p)=1+O(1/p)$. Since $\sum_{n=1}^\infty \mu(n)/n=0$, we then find that the total contribution of the main terms to $\CM_{\tilde{f}_0,2k}(R)$ is 
\[
(\log{R})^{2^{2k-1}-2k-1}\sideset{}{^\flat}\sum_{\substack{D_I\le L\ (I\in \CS^-(2k))  \\ D=\prod_{I\in \CS^-(2k)} D_I }}  \frac{\mu(D)g(D)}{D} 
	\ll e^{-c'(\log\log R)^{3/2}} .
\]
for some $c'=c'(k)>0$, which is negligible. Consequently,
\[
\CM_{\tilde{f}_0,2k}(R)
	\ll (\log\log R)^{O(1)} (\log R)^{2^{2k-1}-2k-2} ,
\]
whereas the power of $m$ in Theorem \ref{thm-perm} is $2^{2k-1}-2k-1$. So this heuristic indicates that we should get more cancellation in the integer setting than we will obtain in the analogous permutation question, as established in Theorem \ref{thm-perm}.

\subsection{Polynomial analogue}
The reader might be sceptical of the argument presented above, because a direct analogue exists for polynomials over finite fields too. Specifically, expanding the $2k$-th power in $\text{Poly}_q(n,m;k)$, we find that
\[
\text{Poly}_q(n,m;k)
	= \sum_{G_1,\dots,G_{2k}} \frac{\mu(G_1)\cdots\mu(G_{2k})}{q^{\deg([G_1,\dots,G_{2k}])}} 
\]
for $n\ge 2mk$. Given square-free, monic polynomials $G_1,\dots,G_{2k}$ over $\F_q[t]$ and $I\subset\CS^*(2k)$, we let $G_I$ be the product of those monic irreducibles $P$ that divide each of the $G_i$'s with $i\in I$ but do not divide $\prod_{i\in[2k]\setminus I}G_i$. Then the polynomials $G_I$ for $I\in\CS^*(2k)$ are pairwise coprime and $G_i=\prod_{I:\, i\in I} G_I$ for each $i$, so that
\als{
\text{Poly}_q(n,m;k)
	&=  \sideset{}{^\flat}\sum_{\substack{G_I\ (I\in\CS^*(2k)) \\
			\sum_{I\ni i}\deg(G_I) = m \ (1\le i\le 2k)}}
		\left(\prod_{I\in\CS^-(2k)}\frac{\mu(G_I)}{q^{\deg(G_I)}}\right)
			\left(\prod_{I\in\CS^+(2k)} \frac{1}{q^{\deg(G_I)}} \right).
}
where the notation $\sideset{}{^\flat}\sum$ means that the summation is running over squarefree and pairwise monic polynomials $G_I$. As in the integer case, the contribution to $\text{Poly}_q(n,m;k)$ of those tuples $(G_I)_{I\in\CS^*(2k)}$ such that $\deg(G_I)$ is large for some $I\in\CS^-(2k)$ is negligible, by the Prime Number Theorem over $\F_q[t]$. Hence, we may assume that $\deg(G_I)\le\log m$ for all $I\in\CS^-(2k)$. Then it is natural to write
\[
\text{Poly}_q(n,m;k) 
	= \sideset{}{^\flat}\sum_{\substack{\deg(G_I)\le \log m \ (I\in \CS^-(2k)) \\ G=\prod_{I\in \CS^-(2k)} G_I }} \frac{\mu(G)}{q^{\deg(G)}} 
		\cdot  \tilde{T}_{q,2k}(m_1,\dots,m_{2k};G),
\]
where $m_i=m-\sum_{I\ni i,\, I\in \CS^+(2k)} \deg(G_I)$ and
\[
\tilde{T}_{q,2k}(\bs m ;A)	
	:= \sideset{}{^\flat}\sum_{\substack{(G_I,A)=1\ (I\in \CS^+(2k)) \\ 
		\sum_{I\in\CS^+(2k):\, i\in I} \deg(G_I) = m_i \ (1\le i\le 2k) }}
					\prod_{I\in\CS^+(2k)} \frac{1}{q^{\deg(G_I)}} .
\]
As before, when $\ell_i=m+O(\log m)$, as above, we should be expecting that $\tilde{T}_{q,2k}(\bs\ell;A)$ has an asymptotic formula of the form
\eq{poly-count}{
\tilde{T}_{q,2k}(\bs\ell;A) = \tilde{g}(A) m^{2^{2k-1}-2k-1} + O\left( m^{2^{2k-1}-2k-2} (\log m)^{O(1)} \right),
}
where $\tilde{g}(A)$ is a multiplicative function with $\tilde{g}(P)=1+O(1/q^{\deg(P)})$ for irreducibles $P$. 

The above argument suggests that we should have an asymptotic behaviour of $\text{Poly}_q(n,m;k)$ that is smaller than what Theorem \ref{thm-poly} states, which is absurd. The problem is that if $\sum_{I\ni i}\deg(G_I) =m$ for all $i$, then we also have that
\[
2km = \sum_{i=1}^{2k} \sum_{I:\,i\in I}\deg(G_I) = \sum_I \#I\deg(G_I) .
\]
Reducing this formula mod 2, we find that
\eq{poly-condition}{
\sum_{I\in\CS^-(2k)} \deg(G_I) \equiv 0\mod 2,
}
a local constraint that is not present in the integer analogue. In particular, we see that \eqref{poly-count} is true only when \eqref{poly-condition} is satisfied. We thus find that the main term for $\text{Poly}_q(n,m;k)$ equals
\als{
&m^{2^{2k-1}-2k-1} \cdot  \sideset{}{^\flat}\sum_{\substack{\deg(G_I)\le \log m\ (I\in \CS^-(2k)) \\ G=\prod_{I\in \CS^-(2k)} G_I ,\ 2|\deg(G) }} \frac{\mu(G)\tilde{g}(G)}{q^{\deg(G)}} \\
		&\quad=\frac{m^{2^{2k-1}-2k-1}}{2} \cdot 
			\sideset{}{^\flat}\sum_{\substack{\deg(G_I)\le \log m\ (I\in \CS^-(2k))\\
				 G=\prod_{I\in \CS^-(2k)} G_I   }} \frac{\mu(G)(1+(-1)^{\deg(G)})}{q^{\deg(G)}} +O\Bigl(\frac{1}{q}\Bigr)\\
	&\quad \asymp m^{2^{2k-1}-2k-1},
}
because $\sum_F \mu(F)/q^{\deg(F)}=0$ and $\sum_F \mu(F)(-1/q)^{\deg(F)}=\prod_P (1-(-1/q)^{\deg(P)})>0$. Thus we see the local constraint associated to the discreteness of degrees in the polynomial setting means we have genuinely different asymptotic behavior.

\subsection{Further analysis}

The above arguments suggest a possible route to proving Theorem \ref{mainthm}, by working out the full asymptotic expansion of $T_{2k}(\bs x;a)$. Controlling the coefficients in this expansion is a highly non-trivial problem. Instead, we take another route, using a high-dimensional contour shifting argument. Our starting point is Perron's inversion formula which, ignoring convergence issues, yields
\[
T_{2k}(\bs x;a)	 \sim
	\frac{1}{(2\pi i)^{2k}} \idotsint\limits_{\substack{\Re(s_j)=1/\log R \\1\le j\le 2k}}
		\ \sideset{}{^\flat}\sum_{\substack{(D_I,a)=1 \\ I\in \CS^+(2k) }}
					\prod_{I\in\CS^+(2k)} \frac{1}{D_I^{1+s_I}} 
					\prod_{j=1}^{2k} \frac{x_j^{s_j}(1-2^{-s_j})}{s_j} \dee s_1\cdots \dee s_{2k} ,
\]
with the notational convention that $s_I=\sum_{j\in I} s_j$. Therefore
\[
\CM_{\tilde{f}_0,2k}(R) \sim
	\frac{1}{(2\pi i)^{2k}} \idotsint\limits_{\substack{\Re(s_j)=1/\log R \\1\le j\le 2k}}
		F(\bs s) \frac{\prod_{I\in\CS^+(2k)} \zeta(1+s_I)}{\prod_{I\in\CS^-(2k)} \zeta(1+s_I)}
					\prod_{j=1}^{2k} \frac{1-2^{-s_j}}{s_j} \dee s_1\cdots \dee s_{2k} ,
\]
where $F(\bs s)$ is analytic and non-zero when $\Re(s_j)>-1/4k$ for all $j$. As we will see in Section \ref{contour}, shifting contours, we pick up poles any time $s_I=0$ for some $I\in \CS^+(2k)$. What is the difficulty in proving Theorem \ref{mainthm} is that some of these poles can get annihilated by poles of the zeta factors in the denominator, which is an analytic way of saying that the higher order terms in the asymptotic expansions of $T_{2k}(\bs x;a)$ are cancelled out.

\medskip

It is clear from the above discussion that the underlying reason why we got a genuinely smaller main term for $\CM_{\tilde{f}_0,2k}(R)$ is the identity $\sum_{n=1}^\infty \mu(n)/n=0$, that is to say the fact that $1/\zeta$ has a zero at 1. This also explains the phenomenon we see in Theorem \ref{thm-chars}. If we replace $\mu$ by a real valued multiplicative function $f$ whose Dirichlet series $F(s)=\sum_{n=1}^\infty f(n)/n^s$ which is not very small at $s=1$, then the behaviour of the respective divisor sums should be similar as the permutation analogue, whilst if $F(s)$ is close to 0 at $s=1$ (which occurs if $F$ has a zero very close to 1) the behaviour is the same as in the original integer setting.

\subsection{Further obstructions to the analogy}
Is it possible that the local constraints at the prime 2 described above are the only thing separating integers and polynomials? In order to study this question, we consider the variations
\[
\text{Poly}_q(n,m,h;k)
	:= \frac 1{q^n} \sum_{\substack{N\in \F_q[t] \\ \deg N=n}} \Biggl|  \sum_{\substack{ M|N \\ m-h<\deg M\le m}} 
		\mu(M) \Biggr|^{2k} ,
\]
where $h\in\Z\cap[1,m+1]$. If $h\ge2$, then the local problems at the prime 2 should be resolved. However, we will see that this is not sufficient, and that the discrepancy between the integer and the polynomial analogues goes even deeper.

First, let us consider the case $h=m+1$ in order to convince the reader that resolving the constraints at the prime 2 is not sufficient. It is known that a positive proportion of polynomials $N\in \F_q[t]$ of degree at most $r$ have a simple zero over $\F_q$, and that the number of zeroes of such a polynomial over $\F_q$ is, on average, bounded. So we should expect
\[
\text{Poly}_q(n,m,m+1;k)
	\asymp \frac 1{q^n} \sum_{\substack{N\in \F_q[t] \\ \deg N=n}}\sum_{\substack{\alpha\in \F_q \\ N(\alpha)=0 \\ N'(\alpha)\neq0}} 
	\Biggl|  \sum_{\substack{ M|N \\ \deg M\le m}} \mu(M) \Biggr|^{2k}.
\]
If $N$ has a simple zero at $\alpha$, then we can factor $N(x)=(x-\alpha) \tilde{N}(x)$, where $\tilde{N}(\alpha)\neq0$, that is to say $x-\alpha$ and $\tilde{N}$ are co-prime. Then
\[
\sum_{\substack{ M|N \\ \deg M\le m}} \mu(M)
 = \sum_{\substack{ M|\tilde{N} \\ \deg M\le m}} \mu(M)
 	+ \sum_{\substack{ M|\tilde{N} \\ 1+\deg M\le m}} \mu((x-\alpha)\cdot M)
	=  \sum_{\substack{ M|\tilde{N} \\ \deg M=m}} \mu(M) ,
\]
so that 
\[
\text{Poly}_q(n,m,m+1;k)
	\asymp \frac 1{q^n} \sum_{\alpha\in\F_q} 
		\sum_{\substack{\tilde{N}\in \F_q[t] \\ \deg \tilde{N}=n-1 \\ \tilde{N}(\alpha)\neq0}}
	\Biggl|  \sum_{\substack{ M|\tilde{N} \\ \deg M=m}} \mu(M) \Biggr|^{2k}
	\asymp m^{2^{2k-1}-2k-1} +1
\]
for $n\ge2mk$, by an easy variation of Theorem \ref{thm-poly}. This argument can be made rigorous; we leave this task to the interested reader.

\medskip

Let us now study $\text{Poly}_q(n,m,h;k)$ more generally. For any $h\in\Z_{\ge 1}$ and $n\ge 2mk$, we note that
\als{
\text{Poly}_q(n,m,h;k)
	&=  \sum_{\substack{G_1,\dots,G_{2k} \\ m-h<\deg(G_i)\le m \\ 1\le i\le 2k}}
		\frac{\mu(G_1)\cdots \mu(G_{2k})}{q^{\deg([G_1,\dots,G_{2k}])}} \\
	&=  \sideset{}{^\flat}\sum_{\substack{G_I \ (I\in\CS^*(2k)) \\
			m-h< \sum_{I\ni i}\deg(G_I) \le m  \ (1\le i\le 2k)}}
		\left(\prod_{I\in\CS^-(2k)}\frac{\mu(G_I)}{q^{\deg(G_I)}}\right)
			\left(\prod_{I\in\CS^+(2k)} \frac{\mu^2(G_I)}{q^{\deg(G_I)}} \right),
}
as before. Applying Fourier inversion $2k$ times, we find that, for any $r\in(0,1)$, 
\als{
\text{Poly}_q(n,m,h;k)
	&=  \sum_{\substack{m-h<\ell_j\le m \\ 1\le j\le 2k}} 
		\ \sideset{}{^\flat}\sum_{G_I\,(I\in\CS^*(2k))}
		\left(\prod_{I\in\CS^-(2k)}\frac{\mu(G_I)}{q^{\deg(G_I)}}\right)
			\left(\prod_{I\in\CS^+(2k)} \frac{\mu^2(G_I)}{q^{\deg(G_I)}} \right) \\
	&\qquad\times \prod_{j=1}^{2k} 
				\int_0^1 (re(\theta_j))^{-\ell_j+\sum_{I\ni j} \deg(G_I)} \dee\theta_j .
}
So, if we set
\[
\CZ_q(w) = \sum_{G\in \F_q[t]} \left( \frac{w}{q}\right)^{\deg(G)} 
	= \prod_P (1 - (w/q)^{\deg(P)})^{-1} ,
\]
the $\F_q[t]$ analogue of the Riemann zeta function, then
\als{
\text{Poly}_q(n,m,h;k) 
		&= \sum_{\substack{m-h<\ell_j\le m \\ 1\le j\le 2k}}	
			 \int_{[0,1]^{2k}} \tilde{F}_q( (re(\theta_j))_j ) 
			 	\frac{\prod_{I\in \CS^+(2k)} \CZ_q(r^{\#I} e(\theta_I))}
					{\prod_{I\in \CS^-(2k)} \CZ_q(r^{\#I} e(\theta_I))}
					\prod_{j=1}^{2k} \frac{e(-\ell_j\theta_j)}{r^{\ell_j}}  \dee \bs\theta \\
		&= \int_{[0,1]^{2k}} \tilde{F}_q\left((re(\theta_j))_j\right)
			 	\frac{\prod_{I\in \CS^+(2k)} \CZ_q(r^{\#I} e(\theta_I))}
					{\prod_{I\in \CS^-(2k)} \CZ_q(r^{\#I} e(\theta_I))}
					\prod_{j=1}^{2k}\sum_{\ell=m-h+1}^m \frac{e(-\ell \theta_j )}{r^\ell}  \dee \bs\theta ,
}
where $\theta_I=\sum_{j\in I} \theta_j$ and $\tilde{F}_q(\bs w)$ is a certain function that is analytic and non-zero when $|w_j|<\sqrt{q}/2k$ for all $j$. 

We take $r=1-1/m$ and note that the main contribution to $\text{Poly}_q(n,m,h;k)$ should come from those values of $\bs\theta$ for which there are many $I\in\CS^+(2k)$ such that $\theta_I\equiv 0\mod 1$. This is the key difference with the integer case: before, we needed many $I\in\CS^+(2k)$ with $s_I=0$. So we see two different linear algebra problems: one over the group $\R/\Z$, which has torsion, and one over $\R$, which does not. The presence of torsion in $\R/\Z$ is a reflection of the discreteness of the polynomial setting (of the degree of the polynomials, more precisely), and the fact that $\R$ is a field reflects the continuous nature of the integer problem (of the logarithms of integers, more precisely).

When $h=1$, then the integrand is $\asymp 1/m$ when $\theta_j=O(1/m)$ mod 1 for all $j$, much like the integer analogue. However, if we take $\theta_j=1/2+O(1/m)$ for all $j$, then we see that $\theta_I=O(1/m)$ mod 1 for $I\in\CS^+(2k)$, whereas $\theta_I=1/2+O(1/m)$ mod 1 for $I\in\CS^-(2k)$, so the integrand has size $m^{2^{2k-1}-1}$ for such $\bs \theta$. The volume of this region is $\asymp 1/m^{2k}$, leading to a contribution of size $m^{2^{2k-1}-2k-1}$ to $\text{Poly}_q(n,m,1;k)$, which is precisely its order of magnitude for $k\ge2$. Note that the fact the main contribution comes from when $\theta_j\approx 1/2$ and not when $\theta_j\approx 0$ is a reflection of the local constraint at the prime 2 we noticed above.

Similarly to the above case, if $h=2$ and $\theta_j=1/2+O(1/m)$, then the integrand becomes 
\[
\tilde{F}_q\left(1/2,\dots,1/2\right)
			 	\frac{m^{2^{2k-1}-1}(1+O(1/m))}
					{\CZ_q(1/2)^{4^k}} 
					\prod_{j=1}^{2k} (1+e(\theta_j)) .
\]
By Taylor expansion, we have that 
\[
1+e(\theta_j)=1-e(\theta_j-1/2) 
	= -(\theta_j-1/2) -\frac{(\theta_j-1/2)^2}{2} - \cdots
\]
By symmetry, we should then have that 
\als{
&\idotsint\limits_{\substack{|\theta_j-1/2|\le 1/m \\ 1\le j\le 2k}}
	\tilde{F}_q\left(1/2,\dots,1/2\right)
			 	\frac{m^{2^{2k-1}-1}(1+O(1/m))}
					{\CZ_q(1/2)^{4^k}} 
					\prod_{j=1}^{2k} (1+e(\theta_j)) \dee \bs \theta \\
&\quad = (1+O(1/m)) \idotsint\limits_{\substack{|\theta_j-1/2|\le 1/m \\ 1\le j\le 2k}}
	\tilde{F}_q\left(1/2,\dots,1/2\right)
			 	\frac{m^{2^{2k-1}-1}}
					{\CZ_q(1/2)^{4^k}} 
					\prod_{j=1}^{2k} \frac{(\theta_j-1/2)^2}{2} \dee \bs \theta,
}
which leads to a contribution of size $m^{2^{2k-1}-6k-1}$ to $\text{Poly}_q(n,m,2;k)$. This should be the dominant contribution for large $k$, even though for small $k$ other regions can dominate. For example, if $k=2$ and we take $\theta_1,\theta_2\in[0.33,0.34]$, $\theta_3,\theta_4\in[0.66,0.67]$, and $\theta_1+\theta_2,\theta_3+\theta_4,\theta_1+\theta_4=O(1/m)$, then the integrand becomes $\asymp m^5$, and we are integrating over a region of volume $\asymp 1/m^3$, so we see that $\text{Poly}_q(n,m,1;2)\gg m^2$. In fact, this is the exact order of magnitude of $\text{Poly}(n,m,1;2)$. 

We conclude our discussion with another peculiar fact: if $h=3$, then 
\[
\sum_{\ell=m-h+1}^m e(\ell \theta )
	= e(m\theta) (1+e(\theta)+e(2\theta)) 
	= e(m\theta)+O(1/m)
\]
when $\theta=1/2+O(1/m)$. So $\text{Poly}_q(n,m,3;k)$ should have the same size as $\text{Poly}_q(n,m,1;k)$, whereas $\text{Poly}_q(n,m,h;k)$ is a bit smaller, by a factor of size $m^{O(k)}$. In general, no matter how we choose $h$, we cannot make the sum $\sum_{\ell=m-h+1}^m e(\ell \theta )$ small enough to cancel the contribution of the factors $\CZ_q(r^{\#I}e(\theta_I))$ for even $I$ in the region $\theta_j\sim 1/2$, so the quantities $\text{Poly}_q(n,m,h;k)$ do not behave in the same way as $\CM_{\tilde{f}_0,2k}(R)$ for $k$ large.

\section{The analogy for permutations}\label{permsect}

\begin{proof} [Completion of the proof of Theorem \ref{thm-perm}] It remains to prove the two claims for the quantity $c(m,k)$, which we recall is defined as the number of $(2^{2k}-1)$-tuples $(r_I)_{\emptyset\neq I\subset[2k]}$ of non-negative integers such that $r_I \in\{0,1\}$ for $\#I$ odd and such that $\sum_{I:\ i\in I}\ r_I = m$, for each $i\in[2k]$. 

Given any vector $\{ r_I:\ \emptyset \ne I\subset [2k]\}$ counted by $c(m,k)$, the vector
$\{ r_I':\ \emptyset \ne I\subset [2k]\}$ is counted by $c(m+1,k)$ where
$r_{\{ 1,2\}}'= r_{\{ 1,2\}}+1$ and $r_{\{ 3,4,\ldots,2k\}}'= r_{\{ 3,4,\ldots,2k\}}+1$, and $r_I'=r_I$ otherwise. Since $r_I\mapsto r_I'$ is injective, we see
$c(m,k)\leq c(m+1,k)$ for all $m\geq 0$, as claimed.

We now estimate $c(m,k)$. When $k=1$, we find immediately that $c(m,1)= 2$, so there is noting to prove. Assume now that $k\ge2$. Since $c(m,k)$ is increasing in $m$ and $c(0,k)=1$, we may assume that $m$ is even and large enough. We note that that there are $\asymp_k 1$ possibilities for the $r_I$ for the odd-sized $I$.  Otherwise we have to satisfy $2k$ equations with $2^{2k-1}-1$ variables. Hence the number of solutions should be 
\[
\asymp_k  m^{2^{2k-1}-2k-1} +1,
\]
as claimed. Certainly, this argument yields an appropriate upper bound. To prove the lower bound for $k\ge2$ we will construct this number of solutions.  Let
\[
\CI = \{ \{ i,j\}: 1\le i<j\le 4\} \cup \{ \{1,j\} : 5\le j\le 2k\}  ,
\]
so that $\#\CI=2k+2$. Set $r_I=0$ if $I\in\CS^-(2k)$ and, given $\delta>0$ to be chosen later, let $r_I$ be any even integer from the range $[0,\delta m/4^k]$ if $I\in\CS^+(2k)\setminus (\CI\cup\{\{5,6,\dots,2k\}\})$. Finally, if $k>2$, let $r_{\{5,6,\dots,2k\}}$ be an even integer from the range $[m-2\delta m,m-\delta m]$. There are $\asymp_{k,\delta} m^{2^{2k-1}-3-2k}$ such choices of $r_I$, $I\in\CS^+(2k)\setminus\CI$. Then select 
\[
r_{\{1,j\} }:=m -\sum_{I\in\CS^+(2k)\setminus\CI,\ j\in I} r_I \quad(5\le j\le 2k),
\]
which is an even integer lying in the interval $[0,2\delta m]$, so that $\sum_{I\in\CS^+(2k),\,j\in I}r_I=m$ for $5\le j\le 2k$. Now set 
\begin{align*}
\CI_2&=\{\{i,j\}:1\le i<j\le 4\},\\
m_j &= m - \sum_{I\in\CS^+(2k)\setminus \CI_2,\ j\in I} r_I \quad(1\le j\le 4).
\end{align*}
We note that the $m_j$ are even integers lying in the interval $[m-2k\delta m,m]$. It remains to choose $r_I$, $I\in\CI_2$, such that $\sum_{I\in\CI_2,\ j\in I} r_I = m_j$ for $1\le j\le 4$. Then, we select any even integers $r_{\{2,4\}},r_{\{3,4\}}$ from $[\sqrt{\delta} m-\delta m,\sqrt{\delta} m+\delta m]$, and we set 
\[ 
r_{\{1,4\}} = m_4 - r_{\{2,4\}}  -   r_{\{3,4\}} .
\]
Finally, we define $r_{ \{1,2\} }, r_{ \{1,3\} }, r_{\{2,3\} }$ such that
\begin{align*}
r_{\{1,2\}} + r_ {\{1,3\}} &= m_1 - r_{\{1,4\}} = m_1-m_4 +r_{\{2,4\}}+r_{\{3,4\}} ;    \\
r_{\{1,2\}}+ r_{\{2,3\}}  & = m_2  - r_{\{2,4\}} ;\ \text{and} \\ 
r_{\{1,3\}}+ r_{\{2,3\}}&= m_3 -  r_{\{3,4\}} .
\end{align*}
Note that the right-hand sides are all even so there is no parity problem, and the solutions we obtain are non-negative integers for $\delta$ small enough. We have thus constructed $\gg_k m^{2^{2k-1}-2k-1}$ solutions counted by $c(m,k)$. This completes the proof of the lemma.
\end{proof}

\begin{rem}
It should not be too difficult to determine $c(m,k)$ exactly in some special cases. For example, we have that $c(m,1)=2$ and $c(m,2)=\frac 13( 64m^3-135m^2+182m-66)$ for all $m\geq 1$.
\end{rem}

Finally, we prove a probabilistic interpretation for $c(m,k)$. In its statement, we have set with a slight abuse of notation
\eq{M-perm}{
M(\bs c ;  r) := \sum_{\substack{0\le b_j\le c_j \\ 1\le j\le m \\ \sum_j j b_j =r }} (-1)^{b_1+\cdots+b_m}
}
for an $m$-tuple of non-negative integers $\bs c=(c_1,\dots,c_m)$.

\begin{prop}\label{perm-prob} Let $\bs X=(X_1,X_2,\dots,X_m)$ be a vector of pairwise independent Poisson random variables, where $X_j$ has parameter $1/j$. For every $k\in\Z_{\ge 1}$, we have that
\[
c(m,k) = \mathbb{E}[M(\bs X;m)^{2k}] .
\]
\end{prop}
In passing, we note that Proposition \ref{perm-prob} is purely a statement about Poisson random variables and not immediately related to permutations or polynomials over finite fields, but our proof makes use of this connection.

Before we prove Proposition \ref{perm-prob}, we need a lemma.

\begin{lem}\label{perm-sieve} Let $N\ge m\ge1$. The proportion of permutations $\sigma\in S_N$ that have no cycles of length $\le m$ is
\[
\prod_{j=1}^m e^{-1/j} 
		+ O\left(\frac{m^2}{N} \right) .
\]
\end{lem}

\begin{proof} Note that this lemma was proven by the first author in \cite{Gra} for large $m$, but here we are mainly interested in the case when $m$ is very small compared to $N$. We apply inclusion-inclusion. If $\CC_j$ denotes the $j$-cycles in $S_N$, we write $|\pi|=j$ for an element $\pi$ of $\CC_j$, and we let $\CC$ be the union of $\CC_1,\dots,\CC_m$. Then
\als{
&\#\{\sigma\in S_N : \mbox{$\sigma$ has no cycles of length $\le m$}\} \\
	&\quad
		= N! - \sum_{\pi \in \CC} (N-|\pi|)! 
			+ \sum_{\substack{\pi_1,\pi_2\in \CC \\ \pi_1,\pi_2\ \text{disjoint}} }(N-|\pi_1|-|\pi_2|)!
		\mp \cdots \\
	&\quad  =
		\sum_{\substack{c_1,\dots,c_m \ge 0 \\ c_1+2c_2+\cdot+mc_m\le n}} (-1)^{c_1+\cdots+c_m} 
			(N-c_1-2c_2\cdots-mc_m)! 
			\sum_{\substack{ \pi_1,\pi_2,\dots \in \CC\ \text{disjoint} \\ \#\{i:|\pi_i|=j\}=c_j\ \forall j}} 1 .
}
In order to count the inner quantity, we note that if $r=c_1+\cdots+c_m$ is the total number of disjoint cycles we are choosing, and we have fixed our choice for $\pi_1,\pi_2,\dots,\pi_{r-1}$, then there are 
\[
\frac{(N-|\pi_1|-\cdots-|\pi_{r-1}|)!}{|\pi_r|! (N-|\pi_1|-\cdots-|\pi_{r-1}|-|\pi_r|)!}
\] 
choices for the set of size $|\pi_r|$ fixed by $\pi_r$, and then $(|\pi_r|-1)!$ possibilities for a cycle on $|\pi_r|$ given elements. Inductively, we then find that the total number of possibilities for $\pi_1,\dots,\pi_r$ should be
\[
\frac{N!}{(N-|\pi|-\cdots-|\pi_r|)!} \cdot \frac{1}{|\pi_1|\cdots|\pi_r|} 
	= \frac{N!}{(N-c_1-2c_2-\cdots-mc_m)!} \prod_{j=1}^m \frac{1}{j^{c_j}} .
\]
Note though we have overcounted: each possibility of $j$-cycles occurs $c_j!$ times, depending on the order they are picked, so we must divide the above expression by $c_1!\cdots c_m!$. We then find that
\als{
\frac{\#\{\sigma\in S_N : \mbox{$\sigma$ has no cycles of length $\le m$}\}}{N!} 
	&=	\sum_{\substack{c_1,\dots,c_m \ge 0 \\ c_1+2c_2+\cdot+mc_m \le n}} 
		\prod_{j=1}^m \frac{(-1/j)^{c_j}}{c_j!} \\
	&= \prod_{j=1}^m e^{-1/j} 
		+ O\left(\frac{m^2}{N} \right) ,
} 
where the error term is obtained by noting that ${\bf 1}_{c_1+2c_2+\cdots+mc_m\le N} \le (c_1+2c_2+\cdots+mc_m)/N$.
\end{proof}

\begin{proof}[Proof of Proposition \ref{perm-prob}] We recall that we have already proved that
\[
c(m,k) = 
	\frac{1}{N!} \sum_{\sigma\in S_N} \Biggl(\sum_{\substack{T\subset[n] \\ \sigma(T)=T \\ \#T=m}}\mu(\sigma\big|_T) 			\Biggr)^{2k} 
\]
for any $N\ge 2mk$. We will now rewrite the right hand side for $n$ much larger than $m$ and $k$. Note that if $\sigma$ has $c_j$ cycles of length $j$ for each $j\in\{1,\dots,m\}$, then
\[
\sum_{\substack{T\subset[n] \\ \sigma(T)=T \\ \#T=m}}\mu(\sigma\big|_T)
= M(\bs c ;  m) = \sum_{\substack{0\le b_j\le c_j \\ 1\le j\le m \\ \sum_j j b_j =m }} (-1)^{b_1+\cdots+b_m}
\]
where $\bs c=(c_1,\dots,c_m)$. Moreover, a generalization of Cauchy's formula (see Lemma 2.2 in \cite{EFG}) implies that if $t:=c_1+2c_2+\cdots+mc_m\le N$, then
\als{
&\frac{\#\{\sigma\in S_N : \mbox{$\sigma$ has $c_j$ $j$-cycles of length $j$ ($1\le j\le m$)}\}}{N!} \\
	&\quad = \left(\prod_{j=1}^m \frac{1}{j^{c_j} c_j!} \right) 
		\cdot \frac{\#\{\sigma\in S_{N-t} : \mbox{$\sigma$ has no cycles of length $\le m$}\}}{(N-t)!} .
}
Applying Lemma \ref{perm-sieve}, it is then easy to conclude that
\als{
c(m,k)	
	&= \lim_{N\to\infty} \frac{1}{N!} \sum_{\sigma\in S_N} \Biggl(\sum_{\substack{T\subset[n] \\ \sigma(T)=T \\ \#T=m}}\mu(\sigma\big|_T) \Biggr)^{2k} \\
	&=\sum_{\substack{c_1,\dots,c_m \ge 0}} M(\bs c; m)^{2k}  \prod_{j=1}^m \frac{e^{-1/j}}{j^{c_j}c_j!}  .
}
Since $\mathbb{P}(X_j=c_j)=e^{-1/j}/(j^{c_j}c_j!)$, this is $\mathbb{E}[M(\bs X;m)^{2k}]$, and so completes the proof.
\end{proof}
 
\section{The analogy for polynomials over finite fields} \label{polysect}

\begin{proof}[Proof of Theorem \ref{thm-poly}]
Throughout this proof all polynomials we consider are monic, and $P$ denotes a generic monic irreducible polynomial over $\F_q$. Note that
\als{
\text{\rm Poly}_q(n,m;k)
	&= \frac{1}{q^n} \sum_{\deg(F)=n} \Biggl( \sum_{\substack{G|F \\ \deg(G)=m}} \mu(G) \Biggr)^{2k} \\
	&= \prod_{\deg(P)\le m} \Biggl(1-q^{-\deg(P)}\Biggr) 
		\sum_{P|F\ \Rightarrow\ \deg(P)\le m} \frac{1}{q^{\deg(F)}} 
		\Biggl( \sum_{\substack{G|F \\ \deg(G)=m}} \mu(G) \Biggr)^{2k} 
}
for $n\ge 2km$, as can be proven by expanding the $2k$-th power in both sides, and noticing that if $G_j|F$ for each $j\le 2k$, then we may write $F=[G_1,\dots,G_{2k}]H$ for some monic polynomial $H$.

Next, note that if $F=P_1^{n_1}\cdots P_r^{n_r}$ is the factorisation of $F$ into monic irreducible factors, and we write $c_j=\#\{i: \deg(P_i)=j\}$ for $1\le j\le m$, then
\[
\sum_{\substack{G|F \\ \deg(G)=m}} \mu(G) 
	= M(\bs c;m) 
\]
with $M(\bs c;m)$ defined by \eqref{M-perm}. In particular, we see that $\sum_{G|F,\ \deg(G)=m} \mu(G)$ is a function of the vector $c(F):=(c_1,\dots,c_m)$. Moreover, given a fixed vector $\bs c$, we see that
\als{
\prod_{\deg(P)\le m} \left(1-q^{-\deg(P)}\right) 
	 \sum_{\substack{F:\, c(F)= \bs c \\ P|F\ \Rightarrow\ \deg(P)\le m}}
 	\frac{1}{q^{\deg(F)}}
 	& = \prod_{\deg(P)\le m} \left(1-q^{-\deg(P)}\right) 
	  	\prod_{j=1}^m \frac{\binom{N_j}{c_j}}{(q^j-1)^{c_j}} \\
	& = \prod_{j=1}^m \frac{\binom{N_j}{c_j}  (1-q^{-j})^{N_j} }{(q^j-1)^{c_j}} ,
}
where
\[
N_j:=\#\{P\in\F_q[t]: P\ \text{irreducible},\ \deg(P)=j\} .
\]
(Note that we have $(q^j-1)^{c_j}$ and not $q^{jc_j}$ in the denominator because we have to sum over powers of $P_j$ too.) 
Galois theory implies that $q^j = \sum_{j'|j} j'N_{j'}$, whence
\eq{irr-formula}{
N_j	= \frac{1}{j} \sum_{j'|j} \mu(j') q^{j/j'} 
	= \frac{q^j}{j} \left( 1+ O\left(\frac{{\bf1}_{j\ge 2}}{(q^j/j)^{1/2}}\right)\right) 
}
and
\eq{irr-bound}{
q+jN_j\le q^j \quad(j\ge 2) .
}

Our next task is to control the quantity
\[
 \prod_{j=1}^m \frac{\binom{N_j}{c_j}  (1-q^{-j})^{N_j} }{(q^j-1)^{c_j}}
\]
and remove the dependence on $q$. First, note that
\[
\prod_{j=1}^m (1-q^{-j})^{N_j}  = (1+O(1/q)) \prod_{j=1}^m e^{-1/j} .
\]
Furthermore, 
\[
\frac{\binom{N_j}{c_j}}{(q^j-1)^{c_j}}
	= \frac{N_j^{c_j}(1+O(c_j/N_j))^{c_j}}{c_j! (q^j-1)^{c_j}}
	= \frac{1}{c_j! j^{c_j}} \left( 1+ O\left( \frac{{\bf 1}_{j\ge2} \cdot c_j}{(q^j/j)^{1/2}} + \frac{{\bf 1}_{j=1} c_j}{q^j} \right) \right) ,
\]
provided that $c_1\le q$ and that $c_j\le \sqrt{q^j/j}$ if $j\ge2$. Therefore, if $c_1\le q$ and $\sum_{2\le j\le m} c_j \cdot (j/q^j)^{1/2} \le 1$, then 
\[
 \prod_{j=1}^m \frac{\binom{N_j}{c_j}  (1-q^{-j})^{N_j} }{(q^j-1)^{c_j}}
 	= \left(1+O\left(\frac{c_1+1}{q}+ \sum_{j=2}^m \frac{c_j j^{1/2}}{q^{j/2}} \right) \right) 
		\prod_{j=1}^m \frac{e^{-1/j}}{c_j! j^{c_j}} . 
\]
Together with Proposition \ref{perm-prob}, this implies that
\[
\text{\rm Poly}_q(n,m;k)
	= c(m,k) + O(R_1+R_2+R_3) \quad(n\ge 2mk) ,
\]
where
\[
R_1 = \sum_{c_1,\dots,c_m\ge0}
				M(\bs c;m)^{2k}	
			\left(\frac{c_1+1}{q}+ \sum_{j=2}^m \frac{c_j j^{1/2}}{q^{j/2}} \right)  
			\prod_{j=1}^m \frac{e^{-1/j}}{c_j! j^{c_j}} ,
\]
\als{
R_2 	&= \sum_{\substack{c_1,\dots,c_m\ge0 \\ c_1>q\ \text{or}\ \sum_{j>1} c_j j^{1/2}/q^{j/2} >1}} 
					M(\bs c;m)^{2k}	
			\prod_{j=1}^m \frac{e^{-1/j}}{c_j! j^{c_j}}  \\
	&\le \sum_{\substack{c_1,\dots,c_m\ge0}} 
					M(\bs c;m)^{2k}	
			 \left( \frac{c_1}{q}+ \sum_{j=2}^m \frac{c_j j^{1/2}}{q^{j/2}}\right)\prod_{j=1}^m \frac{e^{-1/j}}{c_j! j^{c_j}} 
	\le R_1,
}
and
\als{
R_3 &= \sum_{\substack{c_1,\dots,c_m\ge0 \\ c_1>q\ \text{or}\ \sum_{j>1} c_j j^{1/2}/q^{j/2} >1}} 
					M(\bs c;m)^{2k}	 \prod_{j=1}^m \frac{e^{-1/j} \binom{N_j}{c_j}  }{(q^j-1)^{c_j}}  .
}
For $R_3$, we note that $c_j\le N_j$ in its range; otherwise, $\binom{N_j}{c_j}=0$. In particular, $c_1\le N_1=q$. Moreover, \eqref{irr-bound} implies that
\[
\binom{N_j}{c_j}\le \frac{N_j^{c_j}}{c_j!} 
\le 	\begin{cases}
		\ds \frac{(q^j-1)^{c_j}}{c_j!j^{c_j}} &\text{if}\ j\ge2,\\
		\\
		\ds \frac{q^{c_1}}{c_1} \le (1-1/q)^{-q}\cdot \frac{(q-1)^{c_1}}{c_1!} 
			&\text{if}\ j=1,
		\end{cases}
\]
Therefore
\[
R_3	\ll  \sum_{c_1,\dots,c_m\ge0}
					M(\bs c;m)^{2k}	  \left(\sum_{j=2}^m \frac{c_j j^{1/2}}{q^{j/2}}\right)
					\prod_{j=1}^m \frac{e^{-1/j} }{c_j! j^{c_j}} 
		\le R_1 .
\]

We thus see that Theorem \ref{thm-poly} is reduced to proving that $R_1\ll_k c(m,k)/q$. It suffices to show that
\[
T_i:= \sum_{c_1,\dots,c_m\ge0} c_i
		M(\bs c;m)^{2k} \prod_{j=1}^m \frac{e^{-1/j}}{j^{c_j} c_j!} \ll c(m,k)   \quad(1\le i\le m) .
\]
Indeed, we note that the term with $c_i=0$ does not contribute, and we replace $c_i$ by $c_i+1$ to find that
\[
T_i = \frac{1}{i}\sum_{c_1,\dots,c_m\ge0}
		M(\bs e_i+\bs c;m)^{2k} \prod_{j=1}^m \frac{e^{-1/j}}{j^{c_j} c_j!} ,
\]
where $\bs e_i$ denotes the $m$-th dimensional vector all of whose coordinates are 0 except for the $i$-th coordinate that equals 1. Note that
\[
M(\bs e_i+\bs c;m)
	= \sum_{\substack{0\le b_j\le c_j\  \forall j\neq i \\ 0\le b_i\le c_i+1 \\ \sum_j j b_j =m }} (-1)^{b_1+\cdots+b_m}
	= M(\bs c;m) +(-1)^{c_i+1} M(\bs c_i;m-i(c_i+1)),
\]
where $\bs c_i=(c_1,\dots,c_{i-1},0,c_{i+1},\dots,c_m)$, so that
\[
M(\bs e_i+\bs c;m)^{2k}
	\le 2^{2k-1} \left( M(\bs c;m) +  M(\bs c_i;m-i(c_i+1)) \right)^{2k},
\]
by H\"older's inequality. We thus conclude that
\als{
T_i 	&\le \frac{2^{2k-1}}{i} c(m,k) + \frac{2^{2k-1}}{i} \sum_{c_i=0}^\infty \frac{e^{-1/i}}{c_i! i^{c_i} }
	\sum_{(c_j)_{j\le m,\, j\neq i} }M(\bs c_i;m-i(c_i+1))^{2k}  \prod_{j\neq i} \frac{e^{-1/j}}{c_j!j^{c_j}} \\
	&\le \frac{2^{2k-1}}{i} c(m,k) + \frac{2^{2k-1}}{i} \sum_{c_i=0}^\infty \frac{e^{-1/i}}{c_i! i^{c_i} } c(m-i(c_i+1),k) ,
}
since the $M(\bs c_i;m-i(c_i+1))^{2k}$ is independent of the value of the $c_j$'s with $j>m-i(c_i+1)$. Recalling that $c(\ell,k)$ is an increasing function of $\ell$ by Theorem \ref{thm-perm}, we arrive to the claimed bound $T_i\ll_k c(m,k)$, whence Theorem \ref{thm-poly} follows.
\end{proof}

\section{The support of $M_{f_0}(n;R)$}\label{support}

We prove here \eqref{NumberOfnonZeros}, which we recall is the statement that 
\[
\#\{  n\leq x:\  M_{f_0}(n;R)  \neq 0 \}	
	\asymp \frac{x}{(\log R)^{\delta}(\log\log R)^{3/2}}  \quad (x \ge R^4 )  .
	\]
The lower bound was proven in the introduction, so we are left to show the upper bound. We recall the relation \eqref{differencing1}
\[
	M_f(n;R)= \sum_{d|p_2\cdots p_rm} \mu(d) \left\{ f\left( \frac{\log d}{\log R}\right) 
			-  f\left( \frac{\log p_1}{\log R} + \frac{\log d}{\log R}\right) \right\} ,
\]
where $n=p_1^{\alpha_1}\cdots p_r^{\alpha_r}m$, where $p_1<\cdots<p_r$, $\alpha_i\ge 1$ and all of the prime divisors of $m$ are $>p_r$. Taking $r=2$, letting $q$ be the smallest prime dividing $n$ and writing $n=q^j m$ with $q\nmid m$, we see that
\[
M_{f_0}(n;R) = \sum_{\substack{d|m \\  R/q<d\leq R}} \mu(d) .
\]
Therefore,
\eq{divisors-0}{
\#\{n\le x: M_{f_0}(n;R)\neq0\}
	\le \sum_{q^j \le y} H(x/q^j, q; R/q,R)+ O\left( \frac{x}{\log y} \right),
}
for any parameter $y\le R^{1/3}$ to be chosen later, where
\[
H(X,Y ;Z,W) := \#\{n\le X: P^-(n)>Y,\ \exists d|n\ \text{with}\ Z<d\le W\}  .
\]
We have the following estimate, that is useful in its own right.

\begin{prop}\label{div-prop}
Uniformly for $1\le Y\le Z\le W\le X/(2Z)$ and $2Z\le W\le Z^2$, we have
\[
H(X,Y;Z,W) \ll \frac{X}{\log Y} \cdot \frac{1}{\lambda^\delta(1+\log\lambda)^{3/2}},
\]
where $\lambda$ is defined by the relation $W=Z^{1+1/\lambda}$ and $\delta=1-\frac{1+\log\log2}{\log 2}=0.086071\dots$
\end{prop}

\begin{rem*}
	In the special case when $W=2Z$, Ford \cite{Ford-rough} used a more refined argument and determined the exact order of magnitude of $H(X,Y;Z,W)$. The exact statement is a bit complicated, so we refer the interested reader to Ford's paper.
\end{rem*}

\begin{proof} We adapt the proof of Lemma 6.1 in Ford's paper \cite{Ford}.  By a dyadic decomposition argument, it suffices to upper bound the difference $H(X,Y;Z,W)-H(X/2,Y;Z,W)$. Let $n$ being counted by this difference, so that it can be written as $n=n_1n_2$ with $n_1\in(Z,W]$. We thus have that $n_2\in(X/2W,X/Z]$. If $p=\min\{P^+(n_1),P^+(n_2)\} \in(Y,W]$, then we may write $n=apb$, where:
\begin{itemize}
\item[(i)] all prime factors of $a$ are in $(Y,p)$;
\item[(ii)] all prime factors of $b$ are $\ge p$ (and there is at least one such prime factor);
\item[(iii)] there is a divisor $d|a$ such that $pd\in(Z,W]\cup(X/2W,X/Z]$.
\end{itemize}
If we set $\CL(a;\sigma):=\bigcup_{d|a}[\log d-\sigma,\log d)$ and $\eta=\log(W/Z)$, the last condition can be also written as:
\begin{itemize}
\item[(iii')] either $\log(Z/p)\in \CL(a;\eta)$, or $\log(X/(2Wp)) \in \CL(a;\eta+\log2)$.
\end{itemize}
Let $\eta'=\eta+\log2$, and note that $\eta'\asymp\eta$ by our assumption that $W\ge2Z$. Moreover, let $Z_1=Z$ and $Z_2=X/2W$, so that condition (iii') yields condition
\begin{itemize}
\item[(iii'')] $\log(Z_j/p)\in \CL(a;\eta')$ for some $j\in\{1,2\}$.
\end{itemize}
Finally, note that since there is $d|a$ with $dp>Z_j$, we must have that $p>Z_j/d\ge Z_j/a$. We thus conclude that we must have the condition
\begin{itemize}
\item[(iv)]
$p>Q_j(a):= \max\{P^+(a),Z_j/a\}$.
\end{itemize}

Given $a$ and $p$ satisfying conditions (i), (iii'') and (iv), the number of $b\in(1,X/ap]$ such that $P^-(b)>p$ is $\ll X/(ap\log p)$. Indeed, notice that if there is one such $b$, then $X/ap\ge b\ge p$, so that the claimed estimated follows by a standard sieve bound, such as Theorem 4.3 of \cite{opera}. We thus conclude that
\[
H(X,Y;Z,W)-H(X/2,Y;Z,W) 
	\ll X\sum_{j=1}^2 \sum_{a\in\CP(Y,W)}	\frac{1}{a}
		 \sum_{\substack{p>Q_j(a) \\ \log(Z_j/p)\in \CL(a;\eta')}}
		 	\frac{1}{p\log p} ,
\]
where $\CP(Y,W)$ denotes the set of integers all of whose prime factors are in $(Y,W]$. As in the proof of Lemma 6.1 in \cite{Ford}, we have that the sum over $p$ is $\ll L(a;\eta')/\log^2Q_j(a)$, where $L(a;\sigma)$ denotes the Lebesgue measure of $\CL(a;\sigma)$. We conclude that
\als{
H(X,Y;Z,W)-H(X/2,Y;Z,W) 
		&\ll X\sum_{j=1}^2\sum_{a\in\CP(Y,W)}	\frac{L(a;\eta')}{a\log^2Q_j(a)} \\
		&\ll \frac{X}{\log Y} \sum_{j=1}^2\sum_{P^+(a')\le Y} \sum_{a\in\CP(Y,W)}\frac{L(a;\eta')}{aa'\log^2Q_j(a)} .
}
Since $Q_j(a)\ge Q_j(aa')$ and $L(a;\eta')\le L(aa';\eta')$, we have the estimate
\[
H(X,Y;Z,W)-H(X/2,Y;Z,W) 
	\ll \frac{X}{\log Y} \sum_{j=1}^2\sum_{P^+(m)\le W} \frac{L(m;\eta')}{m\log^2Q_j(m)} .
\]
Since $Z_2=X/2W\ge Z=Z_1$, the contribution for $j=2$ is bounded by the contribution from $j=1$, and so it suffices to just consider $Z_j=Z$. In this case the contribution is $\ll 1/(\lambda^\delta(1+\log\lambda)^{3/2}$ by Lemma 3.3, equation (3.8) and Lemma 3.7 of \cite{Ford}. This completes the proof. 
\end{proof}

Proposition \ref{div-prop} implies that
\[
H(x/q^j, q; R/q,R)
	\ll \frac{x}{q^j\log q} \cdot \left(\frac{\log q}{\log R}\right)^\delta 
	\left(\log\frac{\log R}{\log q}\right)^{-3/2} ,
\]
uniformly in $2\le q^j\le y\le R^{1/2}$ and $x\ge R^{5/2}$. Inserting this bound to \eqref{divisors-0}, we deduce that
\[
\#\{n\le x: M_{f_0}(n;R)\neq0\}
	\ll \frac{x}{(\log R)^\delta(\log\log R)^{3/2}} + \frac{x}{\log y}.
\]
Selecting $y=\exp( (\log R)^{\delta} { (\log\log R)^{3/2}} )$ completes the proof of \eqref{NumberOfnonZeros}. 


\begin{rem}\label{large values}
It is possible to construct integers $n$ for which $M_{f_0}(n;R)$ is quite large. Indeed, let $y\ge3$ and $k\in\Z_{\ge 1}$ be two parameters such that the interval $(y,2^{1/k}y)$ contains at least $2k$ primes, and let $q_1<\cdots <q_{2k}$ be the smallest such primes. Then we set $n=2q_1\cdots q_{2k}$ and $R = 2y^k$. By \eqref{differencing1},
\[
M_{f_0}(n;R) = \sum_{\substack{d|q_1\cdots q_{2k} \\ R/2<d\le R}}\mu(d) .
\]
The choice of $R$ implies that the above sum runs over all divisors $d$ of $q_1,\dots,q_{2k}$ with precisely $k$ prime factors, so that
\[
M_{f_0}(n;R) = (-1)^k \binom{2k}{k}  .
\]
Optimizing the choice of $k$ and $y$, and using the fact that there infinitely many $y$'s such that $\pi(y+\sqrt{y\log y})-\pi(y) \gg\sqrt{y/\log y}$ (see, for example, \cite[Exercice 5, p. 266]{IK}), we find that there exist arbitrarily large integers $n$ such that $|M_f(n;R)|\gg n^{c/\log\log n}$, for any fixed $c<\frac{\log 2}{2}$ with $R\approx n^{1/2}$. 

On the other hand, such extreme values of $M_{f_0}(n;R)$ are very rare, as Theorem \ref{mainthm} indicates.
\end{rem}


\section{Inversion formulas}\label{mellin}

Given $f:\R\to\R$, $R\ge2$ and $s\in\C$, we set
\[
\hat{f}_R(s) = \int_0^\infty f\left( \frac{\log x}{\log R} \right) x^{s-1} \dee x
	= (\log R) \int_{-\infty}^\infty f(u) R^{su} \dee u,
\]
provided that the above integral converges. If $f$ is Lebesgue measurable, supported in $(-\infty,1]$ and bounded, which will always be the case for us, then $\hat{f}_R$ defines an analytic function for $\Re(s)>0$. If, in addition, $f\in C^j(\R)$ for some $j\ge1$ and the derivatives $f',f'',\dots,f^{(j)}$ are all bounded, then integrating by parts $j$ times yields the formula
\eq{IBP-0}{
\hat{f}_R(s) = \frac{(-1)^j}{s^j(\log R)^{j-1}} \int_{-\infty}^\infty f^{(j)}(u) R^{su} \dee u .
}
In particular, we see that 
\eq{IBP}{
\abs{ \hat{f}_R(s) } \le
	 \|f^{(j)}\|_\infty \cdot \frac{R^{\Re(s)}}{\Re(s) (|s|\log R)^j}
}
for $\Re(s)>0$, where we used our assumption that $\text{supp}(f)\subset(-\infty,1]$.

Now, for $m\in\Z_{\ge 1}$, the Mellin inversion formula implies that for $c>0$
\eq{perron}{
f\left( \frac{\log m}{\log R} \right) 
	= \frac{1}{2\pi i} \int_{\Re(s)=c} \hat{f}_R(s) m^{-s} \dee s .
}
In the proof of Theorem \ref{mainthm} with $A\ge1$ and of Theorem \ref{thm-factors}, our assumption that $f$ is a few times differentiable in $\R$ allows us to apply \eqref{IBP} and write $M_f(a;R)$ in terms of an absolutely convergent integral, which can easily be truncated at some appropriate height. However, when $A=0$ in Theorem \ref{mainthm}, we have $f_0=\chi_{(-\infty,1]}$, so that $\hat{(f_0)}_R(s)=R^s/s$. Truncating Perron's formula is still feasible but rather technical. Instead, we perform a technical manoeuvre and smoothen $f_0$ a bit. We consider a smooth function $h:\R\to\R$ such that 
\[
\begin{cases}
h(x)=1 		&\text{if}\ x\le 1-\eta,\\
0\le h(x)\le 1	&\text{if}\ 1-\eta\le x\le 1,\\
h(x)=0 		&\text{if}\ x\ge1,
\end{cases}
\]
where $\eta=1/(\log R)^C$ for some constant $C>0$ that will be chosen appropriately later. We choose $h$ so that $h^{(j)}(x) \ll_j \eta^{-j}$, for all $j\in\Z_{\ge0}$. We claim that, for any fixed $L>0$ and $k\ge1$, there is $C=C(k,L)$ such that
\eq{f_0-g}{
\CM_{f_0,2k}(R) = \CM_{h,2k}(R) + O\left(\frac{1}{(\log R)^L} \right) .
}
Indeed, we have that
\als{
|\CM_{f_0,2k}(R)-\CM_{h,2k}(R)|
	&\le 2k \sum_{\substack{d_1,\dots,d_{2k-1}\le R \\ R^{1-\eta}<d_{2k}\le R}} 
		\frac{\prod_{j=1}^{2k}\mu^2(d_j)}{[d_1,\dots,d_{2k}]}  
	\le 2k \sum_{m\le R^{2k}} \frac{\tau(m)^{2k-1}}{m} 
		\sum_{\substack{d|m \\ R^{1-\eta}<d\le R}}1 ,
}
by setting $m=[d_1,\dots,d_{2k}]$ and $d=d_{2k}$. We split the above sum according to whether $\tau(m)\le (\log R)^B$ or not, where $B$ is some parameter. We then find that
\als{
|\CM_{f_0,2k}(R)-\CM_{h,2k}(R)|
	&\le 2k \sum_{\substack{m\le R^{2k} \\ \tau(m)\le (\log R)^B}} 
		\frac{(\log R)^{(2k-1)B}}{m}
		\sum_{\substack{d|m \\ R^{1-\eta}<d\le R }}1  \\
	&\quad  +2k \sum_{\substack{m\le R^{2k} \\ \tau(m)>(\log R)^B}} 
		\frac{\tau(m)^{2k+1}(\log R)^{-B}}{m} \\
	&\ll_k (\log R)^{(2k-1)B+2-C}+ (\log R)^{2^{2k+1} - B} .
}
We choose $B=L+2^{2k+1}$ and $C\ge (2k-1)2^{2k+1}+2kL+2$ to complete the proof of our claim.  For the purposes of Theorem \ref{mainthm}, we may take $L=1$, so that having $C\ge (2k-1)2^{2k+1}+2k+2$ suffices.  We also note that
\[
\hat{h}_R(s) = -\frac{1}{s} \int_{1-\eta}^1 h'(u) R^{su} \dee u 
		= - \frac{R^s}{s} \int_0^{\eta} h'(1-u) R^{-su} \dee u,
\]
by \eqref{IBP-0} and the fact that $h$ is constant outside $[1-\eta,1]$. In particular, this relation implies that $\hat{h}_R$ has a meromorphic continuation to $\C$ with only a simple pole at $s=0$ of residue $-\int_0^{\eta} h'(1-u) \dee u = 1$. We further note that
\eq{h-bound-prelim}{
\frac{\dee^j}{\dee s^j} \left(\frac{s\hat{h}_R(s)}{R^s}\right) 
	&=  (-1)^{j-1} \int_0^{\eta} h'(1-u) (u\log R)^j R^{-su} \dee u \\
	&\ll \eta \cdot \eta^{-1} \cdot (\eta\log R)^j 
	=  (\eta\log R)^j \quad(s\in \C,\ \Re(s)\ge-1) .
}
Moreover, we have that
\eq{h-mellin}{
\hat{h}_R(s) 
	&= (\log R)\int_{-\infty}^1 R^{us} du + O(\eta (\log R) R^{\Re(s)}) \\
	&= \frac{R^s}{s} +O((\log R)^{-C+1}R^{\Re(s)}) 
	\quad(\Re(s)\ge0) ,
}
a relation that we will use at the very end of the proof of Theorem \ref{mainthm}.


\section{A combinatorial problem in linear algebra}\label{combinatorial}

Recall the notations from Section \ref{notation}. Consider the $2k$-dimensional vector space (over $\Q$) of linear forms in the free variables $s_1,\dots,s_{2k}$, which we denote by $W_k$. Given a subspace $V$ of $W_k$, we define
\[
\SA(V)  =  \sum_{\substack{ J\in \CS^*(2k) \\ s_J \in V}} (-1)^{\#J}  .
\]
(We recall that in our notation $s_J=\sum_{j\in J}s_j$.) We will prove the following result.

\begin{prop} \label{CombProp} Let $k\ge1$ and $V$ be a subspace of $W_k$ containing the form $s_{[2k]}=\sum_{i=1}^{2k}s_i$. 
\begin{enumerate}
\item If $s_j\in V$ for some $j\in[2k]$, then $\SA(V)=-1$.
\item If $\dim(V)=2k-1$, then
\[
\SA(V)-\dim(V) \le \binom{2k}{k} - 2k ,
\]
with equality if, and only if, there is a set $J\subset[2k]$ such that $\#J=k$, $1\in J$, and $V=\Span_\Q(\{s_j-s_1\}_{j\in J},\{s_j+s_1\}_{j\in[2k]\setminus J} )$.
\item If $\dim(V)\le 2k-2$, then 
\[
\SA(V)-\dim(V) \le \binom{2k}{k} - 2k - 2 .
\]
\end{enumerate}
\end{prop}

\begin{proof} (a) If $s_j\in V$, then we immediately see that $\SA(V)=-1$ by pairing $s_J$ with $s_{J\cup\{j\}}$ for each $J\subset[2k]\setminus\{j\}$.

\medskip

(b) We may assume that $s_1,\dots,s_{2k}\notin V$, by part (a). Since $\dim(V)=2k-1$ and $s_1\notin V$, for each $j=1,\dots,2k$ we have that $s_j\equiv r_js_1\mod V$, for some $r_j\in \Q\setminus\{0\}$. We may write $r_j=b_j/q$ for some $b_j\in\Z\setminus\{0\}$ and $q\in\Z_{\ge 1}$,  so that $s_J\in V$ if, and only if, $b_J=\sum_{j\in J}b_j=0$. Therefore
\[
\SA(V)=-1 + \int_0^1 \prod_{j=1}^{2k} (1-e(b_j \theta))\dee \theta
	\le -1 + \int_0^1 \prod_{j=1}^{2k} |1-e(b_j \theta)| \dee \theta.
\]
H\"older's inequality then implies that
\[
\SA(V) \le -1 + \prod_{j=1}^{2k} \left( \int_0^1| 1-e(b_j\theta)|^{2k} \dee \theta \right)^{\frac{1}{2k}} 
	= -1 + \binom{2k}{k} ,
\]
whence
\eq{CS-comb}{
\SA(V) \le \binom{2k}{k}-1  .
}

Finally, we claim that \eqref{CS-comb} is an equality if, and only if, the multiset $\{b_1,\dots,b_{2k}\}$ is of the form $\{b,-b,\dots,b,-b\}$ with $b=b_1$ (which must equal $q$). This claim immediately implies $(b)$ of the Proposition. 

If the multiset $\{b_1,\dots,b_{2k}\}$ is of the form $\{b,-b,\dots,b,-b\}$, then the integral formula for $\SA(V)$ becomes $\SA(V)=-1+\int_0^1|1-e(b\theta)|^{2k}d\theta=\binom{2k}{k}-1$. Conversely, we know that H\"older's inequality above is an equality if, and only if, there exist real numbers $\lambda_1,\dots,\lambda_{2k}$ such that $|1-e(b_j\theta)| = \lambda_j|1-e(b_1\theta)|$ for $\theta\in[0,1]$ and $j\in\{1,\dots,2k\}$. Since $\int_0^1|1-e(b\theta)|d\theta=4/\pi$ for $b\neq0$, we must have that $\lambda_j=1$ for all $j$. Moreover, taking $\theta$ close enough to 0, we find that the condition $|1-e(b_j\theta)|=|1-e(b_1\theta)|$ implies that $|b_j|=|b_1|$ for all $j$. So $\{b_1,\dots,b_{2k}\}$ has $\ell$ copies of $b_1$ and $2k-\ell$ copies of $-b_1$, for some $\ell\in\{1,\dots,2k\}$. Since $b_{[2k]}=0$ by our assumption that $s_{[2k]}\in V$, we must have that $\ell=k$, which completes the proof of our claim.

\medskip

(c) Write $\dim(V)=2k-n$, where $n\ge2$. By part (a), we may assume that $s_1,\dots,s_{2k}\notin V$. 

We first deal with the case $n=2,k=2$ by direct computation. In this case, we have $s_1+s_2+s_3+s_4\in V$ and $s_1,\dots,s_4\notin V$, by assumption. It is thus easy to see that either $V\cap\{s_I:I\in\CS^*(2k)\} =\{s_1+\cdots+s_4\}$ or $V\cap\{s_I:I\in\CS^*(2k)\} = \{s_1+\cdots+s_4,s_J\}$, for some $J$ containing two elements. (Here we recall that $\CS^*(2k)=\{I\subseteq[2k]:\,I\ne 0\}$.) In any case, $\SA(V)\le2$, as required. This completes the proof of part (c) when $n=2$ and $k=2$.

We now assume that either $n>2$ or $k>2$. Choose a maximal subset of linear forms $\{s_{j_1},\dots,s_{j_{n'}}\}$ that are linearly independent when reduced mod $V$. Clearly, $n'=n$. Moreover, a permutation of the variables $s_1,\dots,s_{2k}$ allows to assume without loss of generality that $j_i=i$ for each $i$. Then
\[
s_j \equiv  \sum_{i=1}^n r_{i,j} s_i \quad \mod V
	 \qquad(1\le j\le 2k),
\]
for certain $r_{i,j}\in\Q$. We write $r_{i,j}=b_{i,j}/q$, where $b_{i,j}\in\Z$ and $q\in\Z_{\ge 1}$, so that $s_J\in V$ if, and only if, $b_{i,J}:=\sum_{j\in J} b_{i,j}=0$ for each $i\in[n]$. Thus
\[
\SA(V)+1
	=\int_{[0,1]^n} 	
	\prod_{j=1}^{2k} (1-e(b_{1,j}\theta_1+\cdots+b_{n,j}\theta_n)) 	
			\dee \theta_1\cdots \dee \theta_n .
\]
We set
\[
J_m= \{1\le j\le 2k: b_{n-m+1,j}=\cdots = b_{n,j}=0\} \quad(0\le m\le n)
\]
to be the set of $j$ such that $s_j$  is in the span of $\{s_1,\dots,s_{n-m}\} \mod{V}$. In particular, $J_0=[2k]$ and $J_n=\emptyset$. By construction, $s_i$ is a basis vector of $W_k/V$ for $1\le i\le n$, so for $0\le m\le n-1$ we have $n-m\in J_m$ but $n-m\notin J_{m+1}$. In particular, $\#(J_m\setminus J_{m+1})\ge 1$ for $0\le m\le n-1$. Then
\als{
\SA(V)+1 
	&\le \int_{[0,1]^{n-1}} \prod_{j\in J_1} |1-e(b_{1,j}\theta_1+\cdots+b_{n-1,j}\theta_{n-1}))| \\
	&\quad \times  \left(\int_0^1 \prod_{j\in [2k]\setminus J_1} |1-e(b_{1,j}\theta_1+\cdots+b_{n,j}\theta_n)| 
			\dee \theta_n\right)
			\dee \theta_1\cdots \dee \theta_{n-1} .
}
By H\"older's inequality, the innermost integral is bounded by
\[
\prod_{j\in[2k]\setminus J_1} 
				\left(\int_0^1 |1-e(b_{1,j}\theta_1+\cdots+b_{n,j}\theta_n)|^{2k-\#J_1}\dee \theta_n 				\right)^{\frac{1}{2k-\#J_1}} .
\]
Since $b_{n,j}\neq0$ for $j\notin J_1$, we make the change of variables $\theta_n\to b_{1,j}\theta_1+\cdots+b_{n,j}\theta_n$ and use periodicity to find that
\als{
\int_0^1 |1-e(b_{1,j}\theta_1+\cdots+b_{n,j}\theta_n)|^{2k-\#J_1}\dee \theta_n 
	&= \frac{1}{|b_{n,j|}|} \int_0^{|b_{n,j}|} |1-e(\theta)|^{2k-\#J_1} d\theta  \\
	&=\int_0^1 |1-e(\theta)|^{2k-\#J_1} d\theta .
}
We set 
\[
M(\lambda) = \int_0^1 |1-e(\theta)|^\lambda d\theta =2^\lambda\int_0^1|\sin(\pi\theta)|^\lambda\dee\theta.
\]
Note that $M(2k)=\binom{2k}{k}$. Thus we find that
\[
\SA(V)+1 \le M(2k-\#J_1)
	\int_{[0,1]^{n-1}} \prod_{j\in J_1} |1-e(b_{1,j}\theta_1+\cdots+b_{n-1,j}\theta_{n-1}))|
		\dee \theta_1\cdots \dee \theta_{n-1}.
\]
We repeat the same process to obtain
\als{
\SA(V)+1
	&\le M(2k-\#J_1)M(\#J_1-\#J_2)  \\
	&\quad \times \int_{[0,1]^{n-1}} \prod_{j\in J_2} |1-e(b_{1,j}\theta_1+\cdots+b_{n-1,j}\theta_{n-1}))|
		\dee \theta_1\cdots \dee \theta_{n-2} \\
&\le M(2k-\#J_1)M(\#J_1-\#J_2) \cdots M(\#J_{n-2}-\#J_{n-1}) M(\#J_{n-1}) .
}
Thus
\eq{S_N-maj}{
\SA(V)+1
	\le \sup\{ M(\lambda_1)\cdots M(\lambda_n) :
	\lambda_1+\cdots+\lambda_n=2k,\, \lambda_j\ge1\ (1\le j\le n) \} .
}
By Cauchy-Schwarz, for any positive reals $x,y$ we have
\[
2(xy)^{(A+B)/2}=2(xy)^B(xy)^{(A-B)/2}\le (xy)^B(x^{A-B}+y^{A-B})= x^Ay^B+x^By^A.
\]
Thus, applying this with $x=|\sin{\theta_1}|$, $y=|\sin{\theta_2}|$ we find
\als{
M(\lambda_1)M(\lambda_2) &= \frac{2^{\lambda_1+\lambda_2}}{2}\int_0^1\int_0^1 \Bigl(|\sin(\pi\theta_1)^{\lambda_1}\sin(\pi\theta_2)^{\lambda_2}|+|\sin(\pi\theta_1)^{\lambda_1}\sin(\pi\theta_2)^{\lambda_2}|\Bigr)\dee\theta_1\dee\theta_2 \\
&	\ge 2^{\lambda_1+\lambda_2} \Bigl(\int_0^{1} |\sin (\pi \theta)|^{(\lambda_1+\lambda_2)/2} \dee \theta\Bigr)^2=M\Bigl(\frac{\lambda_1+\lambda_2}{2}\Bigr)^2.
}
In particular, $\log{M(\lambda)}$ is a convex function. It is then easy to see that supremum in \eqref{S_N-maj} is attained when $\lambda_j=1$ for $n-1$ of the indices $j\in[n]$, and with the remaining $\lambda_j$ being equal to $2k-n+1$. Indeed, without loss of generality $\lambda_1,\dots,\lambda_{n-1} \le \lambda_{n}$, and if $\lambda_j\ne 1$ for some $j<n$, then we can increase the size of $M(\lambda_1)\dots M(\lambda_n)$ by replacing $\lambda_j$ with $\lambda_j-1$ and $\lambda_n$ with $\lambda_n+1$.  So
\[
\SA(V) \le M(1)^{n-1} M(2k-n+1)  -1  .
\]
Thus, it suffices to show that
\eq{comb-goal}{
M(1)^{n-1} M(2k-n+1) < \binom{2k}{k} -n  = M(2k) -n 
}
for $2\le n\le 2k-1$ and $k\ge2$. 

Firstly, consider $n=2$ and $k\ge 3$. The function $k\mapsto M(1)M(2k-1)/M(2k)$ is decreasing in $k$ by the convexity of $\log{M(\lambda)}$. Thus
\[
M(1) M(2k-1) \le \frac{M(1)M(3)}{M(4)} M(2k)
	= \frac{64}{9\pi^2} \binom{2k}{k}<\binom{2k}{k}-2.
\]
Here we have used the fact $M(1)=4/\pi$, $M(3)=32/3\pi$ and performed a quick computation to verify $64\binom{2k}{k}/(9\pi^2)<\binom{2k}{k}-2$ for all $k\ge3$.

Now consider $3\le n\le 2k-1$. The function $n\mapsto M(1)^{n-1}M(2k-n+1)$ is decreasing in $m$ since 
\[
\frac{M(1)^{n-1}M(2k-n+1)}{M(1)^{n-2}M(2k-n+2)}=\frac{M(1)M(2k-n+1)}{M(2k-n+2)}\le\frac{M(1)^2}{M(2)}< 1.
\]
Similarly $k\mapsto M(2k-2)/M(2k)$ is decreasing in $k$ respectively by the convexity of $\log{M(\lambda)}$. Thus we have
\als{
M(1)^{n-1}M(2k-n+1)&\le M(1)^2 M(2k-2) \\
&\le \frac{M(1)^2M(2)}{M(4)}M(2k)\\
&=\frac{16}{3\pi^2}\binom{2k}{k}\\
&< \binom{2k}{k} - 2k+1=M(2k)-2k+1.
}
Here we have performed a short computation to verify the final inequality. This completes the proof of the proposition.
\end{proof}


\section{Contour integration}\label{contour}

In this section we begin our attack on Theorem \ref{mainthm}. All implied constants might depend on $k$ and on $A$. We will actually prove a result that is a little weaker than Theorem \ref{mainthm}: we will show that there exist constants $c_{k,A}$ and $c_k'$ for which
\eq{kthmoments}{
\CM_{f_A,2k}(R)= 
		c_{k,A}  (\log R)^{\CE_{k,A}} +O((\log R)^{\CE_{k,A}-1}),
}
and
\eq{kthmoments-dyadic}{
\CM_{\tilde{f_0},2k}(R)= 
		c_k'  (\log R)^{\binom{2k}{k} - 2k } +O((\log R)^{\binom{2k}{k} - 2k-1}).
}
Here we recall that $\CE_{k,A}=\max(\binom{2k}{k}-2k(A+1),-1)$, and we have defined $\tilde{f_0}(x)=f_0(x)-f_0(x+\frac{\log2}{\log R})$, so that 
\[
M_{\tilde{f_0}}(n;R) =  \sum_{\substack{d|n \\ R/2<d \le R}} \mu(d)  .
\]
Notice that we do not claim here that $c_{k,A}\neq 0$ and $c_k\neq0$, as is required in order to prove Theorem \ref{mainthm}. We do obtain a very complicated expression for these constants, but we are unable to prove they are non-zero (or evaluate them at all). Showing that $c_{k,A}>0$ and $c_k'>0$ is the objective of Section \ref{lb}. 

\subsection{Initial preparations}
We will first prove relation \eqref{kthmoments}. The proof of relation \eqref{kthmoments-dyadic} is very similar, and we indicate the necessary changes in the end of Section \ref{contour}. 

We note that
\eq{f_A-mellin}{
\hat{(f_A)}_R(s) = \frac{A!R^s}{(\log R)^A s^{A+1} } .
}
This function is absolutely integrable on vertical lines $\Re(s)=c\neq0$ when $A\ge1$, but this is not the case when $A=0$. However, recall from relation \eqref{f_0-g} that 
\[
\CM_{f_0,2k}(R)  =  \CM_{h,2k}(R) + O\left(\frac{1}{(\log R)^2}\right),
\]
where $h$ is a smooth function such that $h(x)=1$ for $x\le 1-1/(\log R)^C$ and $h(x)=0$ for $x\ge1$, for some constant $C\ge (2k-1)2^{2k+1}+2k+2$ to be chosen later. Therefore relation \eqref{kthmoments} is reduced to showing that
\eq{kthmoments-goal}{
\CM_{g,2k}(R) = c_{k,A} (\log R)^{\CE_{k,A}} + O\left((\log R)^{\CE_{k,A}-1}\right) ,
}
where $g=h$ when $A=0$, and $g=f_A$ when $A\ge1$. 

For any $\lambda>1$, which will be chosen to be sufficiently large in terms of $k$, relation \eqref{perron} implies that
\als{
\CM_{g,2k}(R)	
	&=    \sum_{\substack{m_j \in \Z_{\ge 1} \\ 1\le j\le 2k}} 
		  \frac{\prod_{j=1}^{2k} \mu(m_j)}{ [m_1,\ldots ,m_{2k}]} \cdot \frac 1 {(2i\pi)^{2k}} \idotsint\limits_{\substack{\Re(s_j)=\lambda^j/\log R \\ 1\le j\le 2k}}
    \prod_{j=1}^{2k}m_j^{-s_j} 
	\left( \prod_{j=1}^{2k} \hat{g}_R(s_j) \right) ds_{2k}\cdots ds_1  .
}
To this end, we introduce the multiple Dirichlet series 
\[
D(\bs s):=  \sum_{\substack{m_j \in \Z_{\ge 1} \\ 1\le j\le 2k}} 
		  \frac{\prod_{j=1}^{2k}m_j^{-s_j} \mu(m_j)}{ [m_1,\ldots ,m_{2k}]} ,
\]
which converges absolutely when $\Re(s_j)>0$ for all $j$ as can be seen, for example, by the Euler product expansion
\al{
D(\bs s) 
	&= \prod_p \Bigg( \sum_{\substack{\nu_1,\dots,\nu_k\in\{0,1\}}} \frac{(-1)^{\nu_1+\cdots+\nu_k}}{p^{\nu_1s_1+\cdots+\nu_ks_k}} \cdot \frac{1}{[p^{\nu_1},\dots,p^{\nu_k}]} \Bigg) \nn
	&= \prod_p \left(1+\frac{1}{p} \sum_{\emptyset \neq I\subset[2k]} \frac{(-1)^{\#I}}{p^{s_I}} \right) \label{D(s)} \\
	&=\prod_p \left(1-\frac{1}{p} + \frac{1}{p} \prod_{j=1}^{2k} \left(1-\frac{1}{p^{s_j}}\right)\right), \label{D(s)2}
}
where we have used the notation $s_I=\sum_{i\in I}s_i$. (Similar computations are performed in \cite{BNP}.) We thus see that
\als{
\CM_{g,2k}(R)	
	&=    \frac 1 {(2i\pi)^{2k}} \idotsint\limits_{\substack{\Re(s_j)=\lambda^j/\log R \\ 1\le j\le 2k}}
    D(\bs s)	\left( \prod_{j=1}^{2k} \hat{g}_R(s_j) \right) ds_{2k}\cdots ds_1  .
}
for any $\lambda>1$. 

We shall truncate all variables of integration at height
\[
T: = \exp\{(\log\log R)^2\}  .
\]
To do so, we notice that $\hat{g}_R(s) \ll (\log R)^{O(1)}/|s|^2$ for $\Re(s) =\lambda^j/\log R$, a consequence of \eqref{IBP} when $A=0$ and of \eqref{f_A-mellin} when $A\ge1$, as well as that $D(\bs s) \ll (\log R)^{O(1)}$, an estimate that follows by formula \eqref{D(s)} and the Prime Number Theorem. We conclude that
\[
\CM_{g,2k}(R)= I_{g,2k}(R) + O\left(\frac{1}{(\log R)^2}\right) ,
\]
where
\[
 I_{g,2k}(R):= \frac 1 {(2i\pi)^{2k}} \idotsint\limits_{\substack{\Re(s_j)=\lambda^j/\log R
		\\ |\Im(s_j)|\le T \\ 1\le j\le 2k}}
    D(\bs s) 
   \left( \prod_{j=1}^{2k} \hat{g}_R(s_j) \right) ds_{2k}\cdots ds_1  .
\]
Motivated by \eqref{D(s)} and \eqref{D(s)2}, we set 
\als{
P(\bs s)
	& :=  D(\bs s) \prod_{I \in\CS^*(2k) } \zeta(1+s_I)^{(-1)^{1+\#I}}  \\
	 &= \prod_{p} 
		\left\{ 
		 \left( 1 -\frac 1p + \frac 1p \prod_{j=1}^{2k} \left( 1 -\frac{1}{p^{s_j}}\right) \right) 
		 	\prod_{I\in\CS^*(2k) }
		 	\left( 1 -\frac{1}{p^{1+s_I}}\right)^{(-1)^{\#I}}
		 	\right\} ,
}
which is analytic when $\Re(s_j)>-1/(4k)$ for all $j$, as well as
\[
F(\bs s) : =  P(\bs s) \prod_{j=1}^{2k} \frac{(\log R)^A\hat{g}_R(s_j)}{R^{s_j} \zeta(1+s_j)^{A+1}} 
\]
and
\[
e_I:= \begin{cases}
		A &\text{if}\ \#I=1,\\
		(-1)^{\#I} &\text{if}\ \#I\ge2 ,
	\end{cases} 
\]
so that
\[
I_{g,2k}(R)
	= \frac 1 {(2i\pi)^{2k}}  \idotsint\limits_{\substack{\Re(s_j)=\lambda^j/\log R\\ |\Im(s_j)|\le T \\ 1\le j\le 2k}}
		 \frac{F(\bs s)R^{s_1+\cdots+s_{2k}} }{(\log R)^{2kA}} 
		 \prod_{I\in \CS^*(2k)} \zeta(1+s_I)^{e_I} 
		ds_{2k}\cdots ds_1 .
\]

Given $\ell\in\N$, we now let 
\[
\Omega_\ell:=\{\bs s\in\C^\ell : |\Re(s_j)|< 2/(\log T)^{4/3},\ |\Im(s_j)|<T+1\ (1\le j\le \ell)\}
\]
and define $\CC_\ell$ to be the class of complex-valued functions $f$ such that: (a) $f$ is defined over a complex domain containing $\Omega_\ell$; (b) $f$ is analytic in $\Omega_\ell$; (c) the derivatives of $f$ satisfy the bound 
\eq{F-bound}{
\frac{\partial^{j_1+\cdots+j_\ell}f}{\partial s_1^{j_1}\cdots \partial s_\ell^{j_\ell}}(\bs s)
	\ll_{j_1,\dots,j_\ell}  \frac{(\log\log R)^{O(j_1+\cdots+j_\ell)}}{(|s_1|+1)\cdots(|s_\ell|+1)}  
}
for all $j_1,\dots,j_\ell\ge0$ and all $\bs s=(s_1,\dots,s_\ell)\in\Omega_\ell$. 

We claim that $F\in\CC_{2k}$. Indeed, there are absolute constants $\delta,c_0>0$ such that $\zeta(s)(s-1)$ is analytic and non-vanishing for $|s-1|\le\delta$ and
\eq{zeta-bound}{
\zeta^{(j)}(s), \left(\frac{1}{\zeta}\right)^{(j)}(s) 
	\ll_j  \log^{j+1}(|t|+2) \qquad \text{when}\quad
	\sigma \ge 1- \frac{c_0}{\log(|t|+2)} ,\  |s-1|\ge \delta,
}
with \eqref{zeta-bound} being a consequence\footnote{The claimed bound follows by \cite[Theorems 3.8 and 3.11]{Tit} and the fact that if $f$ is analytic in a neighbourhood of the circle $|z|\le r$, then $f^{(j)}(z_0)=\frac{1}{2\pi i}\oint_{|z|=r}f(z) dz/z$ for any $z_0$ with $|z_0|<r$.}  of the classical zero-free region for $\zeta$. Moreover,
\eq{g-bound}{
\frac{\dee^j}{\dee s^j} \left( \frac{s^{A+1}\hat{g}_R(s)}{R^s} \right)
	 \ll \frac{1}{(\log R)^A}	\quad(\Re(s)\ge-1,\ j\in\Z_{\ge0}) ,
}
an estimate that follows from \eqref{h-bound-prelim} when $A=0$ and from the formula \eqref{f_A-mellin} for $\hat{(f_A)}_R$ otherwise. Our claim that $F\in\CC_{2k}$ then follows.

\subsection{Contour shifting}\label{contour-shifting}
We will simplify $I_{g,2k}(R)$ and prove \eqref{kthmoments} by a $2k$-dimensional contour shifting argument that we will demonstrate in an iterative fashion. The general idea is to move the variables $s_j$ to the left in a certain order. When we move the contour corresponding to the variable $s_j$, we will pick up contributions from poles of the integrand (coming from solutions to linear equations of the form $s_I=0$, $I\in\CS^*(2k)$ with $e_I>0$), and be left with a residual contour (which will be negligible in size). Thus we only need to consider the contributions from the poles, and these contributions will all be multi-integrals similar to $I_{g,2k}(R)$ but involving one fewer variable. By iterating this, we show that $I_{g,2k}(R)$ is (up to a small error term) given by the total contribution of all the successive poles we have encountered having shifted all $2k$ variables. We will show that provided one moves the contours in a suitable order, all the contributions from all the multi-poles and all the residual integrals give a contribution $c_{k,A} (\log R)^{\CE_{k,A}} + O\left((\log R)^{\CE_{k,A}-1}\right)$.

When we consider poles we encounter equations of the form $s_I=0$, where we think of $s_I=\sum_{i\in I}s_i$ as a linear form in the variables $s_1,\dots,s_{2k}$. To avoid any ambiguity when we consider multiple such equations, we will let $L_{0,I}\in \mathbb{Q}[x_1,\dots,x_{2k}]$ be the linear form corresponding to $s_I$, that is to say
\[
L_{0,I}(\bs x) := \sum_{i\in I} x_i .
\]
Before we setup the necessary notation to keep track of all the terms we encounter when performing the multiple contour shifting, we first describe the first two contour shifting steps to help motivate the basic idea.

\medskip

The first variable we move to the left is $s_{2k}$. When doing so, we pick up the contribution from some poles in the integrand. Such a pole must occur when $L_{0,I_1}(\bs{s})=0$ for some $I_1\subset[2k]$ with $e_{I_1}>0$ and $2k\in I_1$ (a possible pole from $\prod_{I\in \CS^*(2k)}\zeta^{e_I}(1+L_{0,I}(\bs{s}))$.)  Having fixed such a pole and the corresponding set $I_1$, we use this equation $L_{0,I_1}(\bs{s})=0$ to rewrite $s_{2k}$ in terms of $s_j$, $j\in[2k]\setminus\{2k\}$. Imposing the same condition on the $x_j$'s, we find that for each $I\subset[2k]$, the linear form $L_{0,I}(x_1,\dots,x_{2k})$ becomes a linear form $L_{1,I}$ in the variables $x_j$ for $j\in[2k]\setminus\{2k\}$. Trivially, $L_{1,I}=0$ if, and only if, $I\in\CI_1:=\{\emptyset,I_1\}$. This pole contribution can be written as an integral over $s_1,\dots,s_{2k-1}$, with an integrand that has poles only when $L_{1,I}(\bs{s})=0$.

Next, for this integral over $s_1,\dots,s_{2k-1}$, we choose some other variable $s_{j_2}$ (precisely how we choose $s_{j_2}$ will be specified later), and move the $s_{j_2}$ contour. This produces a residual contour (which will be negligible) and contributions from further poles in the integrand which occur only when $s_{j_2}$ satisfies a linear equation $L_{1,I_2}(\bs s)=0$ for some $I_2\in\CS^*(2k)\setminus \CI_1$ with $e_{I_2}>0$ and with $L_{1,I_2}(\bs x)$ having a non-zero $x_{j_2}$ coefficient. We use this to write $s_{j_2}$ in terms of $s_j$, $j\in[2k]\setminus\{2k,j_2\}$. Imposing the corresponding condition on the variables $x_j$ makes $L_{1,I}$ a linear form $L_{2,I}$ in the variables $x_j$, $j\in[2k]\setminus\{2k,j_2\}$. Some of these new linear forms will vanish identically, and the total number will determine the order of this pole.

Continuing in this manner, we eventually write our original integral $I_{g,2k}(R)$ in terms of $O(1)$ contributions from repeatedly encountered poles (all of which will be of the form $c(\log{R})^m$ for some $c,m$) or from terms which correspond to encountering a residual integral (which will always be small).  In order to control this process, we need to keep track of which poles we encounter, the order of the poles, and the integrands of the new multi-integrals corresponding to these poles. To do this we introduce some notation and terminology. 
\begin{itemize}
	\item Let us be given an integer $N\in\{0,1,\dots,2k\}$, which we shall often refer to as the {\bf level}. It describes how many iterations we have performed (i.e. how many variables $s_j$ we have shifted). The case $N=0$ corresponds to the initial integral $I_{g,2k}(R)$. 
	
	\item Let us be given sets $I_1,\dots,I_N\subset[2k]$ and indices $j_1,\dots,j_N$ such that:
	\begin{enumerate}[label=(\roman*)]
		\item $j_n\in I_n$ for each $1\le n \le N$.
		\item $j_1,\dots,j_N$ are distinct.  
		\item $e_{I_n}>0$ for each $1\le n\le N$.
	\end{enumerate}
	Here the sets $I_1,\dots I_N$ correspond to the sequence of poles which we have encountered from performing $N$ contour shifts, and the indicex $j_n$ corresponds to the variable we have chosen to use to shift the $n^{th}$ contour.
\end{itemize}

Since the $j_i$ are distinct, the linear forms $L_{0,I_1},\dots,L_{0,I_N}$ are linearly independent over $\Q$. We let $V_N$ be their $\Q$-span and $\CI_N$ to be those forms that vanish identically subject to the conditions $L_{0,I_1}=\cdots=L_{0,I_N}=0$. More generally, for $0\le n\le N$ let
\eq{V_n-I_n}{
V_n = \Span_\Q (L_{0,I_1}(\bs x),\dots,L_{0,I_n}(\bs x))
\quad\text{and}\quad
\CI_n = \{I\in\CS(2k): L_{0,I}(\bs x)\in V_n\} ,
}
with the conventions that $V_0=\{0\}$ and $\CI_0=\{\emptyset\}$. Since $j_r\in I_r$ for all $r$, and $j_1,\dots,j_n$ are distinct integers,  if we impose the conditions $\sum_{i\in I_n}x_i=0$ on the variables $x_i$, then we may write $x_{j_1},\dots,x_{j_n}$ as $\Q$-linear combinations of the other variables. Hence the linear form $L_{0,I}$ becomes a linear form $L_{n,I}$ in the variables $x_j$, $j\in[2k]\setminus\{j_1,\dots,j_n\}$. Clearly, $L_{0,I}\in V_n$ if and only if $L_{0,I}=0$ after we have ``quotiented'' the space of linear forms in the variables $x_j$ with the relations $L_{0,I_1}=\cdots=L_{0,I_n}=0$, if and only if $L_{n,I}=0$.

\begin{rem*}
We will show later on that the variables $s_1,\dots,s_{2k}$ can be permuted in a way that allows us to assume that $j_n = 2k-n+1$ for all $n$. 
\end{rem*}

\begin{dfn}\label{defn of h} Let $N$ be a level. The triplet $(\bs I,\bs h,d)$ is called a {\it type of level $N$} if:
	\begin{enumerate}
		\item[(a)] $\bs I=(I_1,\dots,I_N)$ is an $N$-tuple of sets such that $2k-n+1\in I_n$ and $e_{I_n}>0$ for all $n=1,2,\dots,N$.
		\item[(b)] $\bs h= (h_{n,I})_{0\le n\le N,\, I\in\CS^*(2k)}$ is a tuple of non-negative integers such that:
			\begin{enumerate}[label=(\roman*)]
			\item $h_{n,I}=0$ for $0\le n\le N$ if $e_I=0$ (i.e. if $A=0$ and $\#I=1$);
			\item $0=h_{0,I}\le h_{1,I}\le \cdots \le h_{N,I}$ for $I\in\CS^*(2k)$;
			\item If $I\in \CI_n\setminus\CI_{n-1}$ for some $n\in[N]$, then $h_{m,I}=h_{n,I}$ for all $m\ge n$.
			\end{enumerate}
The integers $h_{n,I}$ will describe the different terms coming up in poles of high order, corresponding to taking many derivatives of different parts of the integrand. (The $h_{N,I}^{th}$ derivative of $\zeta^{e_I}(1+L_{N,I}(\bs{s}))$ will occur in the integrand of the term we are considering.) 
	\item[(c)] $d$ is a non-negative integer.
\end{enumerate}
We will further say that the triplet $(\bs I,\bs h,d)$ is an {\it admissible type of level $N$} if
\[
H_N\ge N+ d,
\]
where
\[
H_N = H_N(\bs{h},\CI_N,A) :=\sum_{I\in\CI_N \setminus\{\emptyset\} } (-1)^{\#I} 
		- \sum_{I\in\CS(2k)\setminus \CI_N} h_{N,I}
		+ (A+1) \sum_{j\in[2k],\, \{j\}\in\CI_N} 1   .
\]
\end{dfn}

\begin{rem} We must have that $H_N\ge N$ if $(\bs I,\bs h,d)$ is an admissible type of level $N$. We will see that the quantity $H_N$ is related to the total order of a the poles we have picked up from the first $N$ contour shiftings. We further note that, in the notation of Section \ref{combinatorial}, it can be written as
\[
H_N= \SA(V_N) - \sum_{I\in\CS(2k)\setminus \CI_N} h_{N,I}
		+ (A+1) \sum_{j\in[2k],\, \{j\}\in\CI_N} 1   .
\]
\end{rem}

The data in a type of level $N$ will keep track of all the relevant information on terms we encounter from poles having shifted $N$ contours. Given a type, we can now define the key objects we wish to consider:
\begin{dfn} \label{defn of J} Let $(\bs I,\bs h,d)$ be a type of level $N$. 
A function $J:\R_{\ge2}\to\C$ is called a {\it fundamental component of level $N$ and of type $(\bs I,\bs h,d)$} if:
\begin{itemize}
\item the type  $(\bs I,\bs h,d)$ is admissible, that is to say, we have $H_N\ge N+d$;
\item when $N=2k$, we have $J(R) = (\log R)^{H_N-N-d-2kA}$;
\item when $N<2k$, we have
\als{
J(R)
	&= \frac{(\log R)^{H_N - N - d - 2kA}}{(2i\pi)^{2k-N}}  
		 \idotsint \limits_{\substack{\Re(s_j)=\lambda_j/\log R\\ |\Im(s_j)|\le T \\ 1\le j\le 2k-N}}
		 G(\bs s)  R^{E_N(\bs s)} \\
	 &\qquad \times \prod_{I\in\CS(2k)\setminus\CI_N }  
		 	\left(\zeta^{e_I}\right)^{(h_{N,I})}(1+L_{N,I}(\bs s)) 
				 	\dee s_{2k-N}\cdots \dee s_1  
}
where $\lambda_j/\lambda_{j-1}\ge\lambda$, 
\[
E_N(s_1,\dots,s_{2k-N}):=L_{N,[2k]}(s_1,\dots,s_{2k-N}),
\]
and $G$ is a function in the variables $s_1,\dots,s_{2k-N}$ that belongs to the class $\CC_{2k-N}$. 
Moreover, if we have additionally that $d=0$, then $G$ is given by
\[
G(\bs s) =  F(L_{N,\{1\}}(\bs s),\dots,L_{N,\{2k\}}(\bs s)).
\]
\end{itemize}
\end{dfn}
We note that when $d=0$, we have that $G$ is non-vanishing in $\Omega_{2k-N}$ by \eqref{zeta-bound} and the preceding discussion.
\begin{dfn}
A fundamental component of level $N$ and type $(\bs I,\bs h,d)$ is called {\it irreducible} if either $N=2k$ or $E_N=0$. Otherwise, it is called {\it reducible}.
\end{dfn}
 With the above notation, the integral $I_{g,2k}(R)$ is a reducible fundamental component of level 0 and of type $(\emptyset,\emptyset,0)$. 

If we say that $J(R)$ is a fundamental component of level $N$, we mean that there exists an admissible type $(\bs{I},\bs{h},d)$ of level $N$ such that $J(R)$ is a fundamental component of level $N$ and type $(\bs{I},\bs{h},d)$.

We begin with a lemma that justifies the terms {\it irreducible vs. reducible}, showing how reducible components are a linear combination of irreducible ones (up to a very small error term). First, we need to introduce a last piece of notation. Notice that if $E_N\neq0$, then we may uniquely write
\[
E_N(\bs x) = \gamma_1 x_1+\gamma_2x_2+\cdots + \gamma_{j_{N+1}} x_{j_{N+1}}
\]
for some $\gamma_j\in \Q$ with $\gamma_{j_{N+1}}\neq0$. If $\lambda$ is big enough, then the sign of $\Re(E_N(\bs s))$ throughout the region of integration is constant and equal to the sign of $\gamma_{j_{N+1}}$. The behaviour of reducible fundamental components differs according to this sign:

\begin{lem}\label{contour-lem1} Assume the above setup. Let $J(R)$ be a reducible fundamental component of level $N<2k$ and type $(\bs{I},\bs{h},d)$, and let $\gamma_1,\dots,\gamma_{j_{N+1}}$ be as above. Assume that $\lambda$ is large enough in terms of  $(\bs{I},\bs{h},d)$.
\begin{enumerate}
\item If $\gamma_{j_{N+1}}>0$, then $J(R)$ is a linear combination of $O(1)$ fundamental components of level $N+1$ with coefficients of size $O(1)$, up to an error term of size $\ll T^{-1+o(1)}$. Each of these fundamental components has (admissible) type $(\bs{I}',\bs{h}',d')$ that depends only on $N,\bs{I},\bs{h},d$.
\item If $\gamma_{j_{N+1}}<0$, then $J(R)\ll T^{-1+o(1)}$.
\end{enumerate}
The implied constants depend at most on $(\bs{I},\bs{h},d)$, $A$ and the function $G$ in the definition of $J$, and are independent of $R$.
\end{lem}

We iterate the above lemma until all the fundamental components we are dealing with are irreducible. For such components, we have the following asymptotic formula.

\begin{lem}\label{contour-lem2} Assume the above setup. If $J(R)$ is an irreducible fundamental component, then there is some $c\in\C$ such that
\[
J(R) = c(\log R)^{\CE_{k,A}}+ O((\log R)^{\CE_{k,A}-1}) ,
\]
where we recall that $\CE_{k,A}=\max(\binom{2k}{k}-2k(A+1),-1)$.  The implied constant and the constant $c$ are independent of $R$.
\end{lem}

Since $I_{g,2k}(R)$ is a reducible fundamental component of level 0, we apply Lemma \ref{contour-lem1} repeatedly to write it as a linear combination of $O(1)$ irreducible fundamental components, and then estimate these components by Lemma \ref{contour-lem2}. This establishes \eqref{kthmoments-goal}. We now prove the above two key lemmas.

\subsection{Proof of the auxiliary Lemmas \ref{contour-lem1} and \ref{contour-lem2}}

\begin{proof}[Proof of Lemma \ref{contour-lem1}]
Note that if $E_N\neq0$, then we must have that either $N=0$, either $k\ge2$ or $A\ge1$: when $N=k=1$ and $A=0$, the only $I\subset\{1,2\}$ with $e_I>0$ is $I=\{1,2\}$. But if $x_{\{1,2\}}=0$, we must have that $E_1=0$, a contradiction.

\medskip

(a) Here $\gamma_{j_{N+1}}>0$. For notational simplicity, we make the change of variables
\[
s_j' = s_j\ (1\le j<j_{N+1}),\ s_j'=s_{j+1}\ (j_{N+1} \le j<2k-N),\ s_{2k-N}'=s_{j_{N+1}},
\]
which corresponds to a cyclic permutation of the variables $s_{j_{N+1}},\dots,s_{2k-N}$. We similarly define the linear forms $x_j'$, using the corresponding permutation of the forms $x_j$, as well as the parameters $\lambda_j'$. We shift the $s_{2k-N}'$ contour to the line $\Re(s_{2k-N}')=-1/(\log T)^{3/2}$. The contribution of the horizontal integrals is $\ll(\log R)^{O(1)}/T$. Moreover, when $\Re(s_{2k-N}')=-1/(\log T)^{3/2}$ and $\Re(s_j')=O(1/\log R)$ for $j<2k-N$, we have that
\[
\Re(E_N(\bs s'))  =   - \frac{\gamma_{j_{N+1}}}{(\log T)^{3/2}} + O\left(\frac{1}{\log R}\right) .
\]
It thus follows that the contribution of the integral with $\Re(s_{2k-N}')=-1/(\log T)^{3/2}$ is $\ll e^{-\sqrt{\log R}}$, say, which is of negligible size. So we need only worry about the poles that the contour shifting introduces.

The poles occur when $L_{N,I_{N+1}}(\bs s')=0$ for some $I_{N+1}\in\CS(2k)\setminus\CI_N$ with $e_{I_{N+1}}>0$ such that the coefficient of $s_{2k-N}'$ in $L_{N,I_{N+1}}$ is non-zero. As we discussed in Section \ref{contour-shifting}, imposing the relation $L_{N,I_{N+1}}(\bs x')=0$ allows us to write $x_{2k-N}'$ as a linear combination of the forms $x_1',\dots,x_{2k-N-1}'$, say $x_{2k-N}'= C(x_1',\dots,x_{2k-N-1}')$. We then define the sets $V_{N+1}$ and $\CI_{N+1}$ as in \eqref{V_n-I_n}, and similarly let $E_{N+1}=L_{N+1,[2k]}$. 

We need to understand the order of the pole at $s_{2k-N}'=C(s_1',\dots,s_{2k-N-1}')$. We only look at {\it generic points} $(s_1',\dots,s_{2k-N-1}')$: it could be the case that for some measure-zero subset of points, we get a different pole order. For example, for fixed $s_1\in\C$, the function $s_2\mapsto s_1/s_2$ has generically a pole of order 1 at $s_2=0$, unless $s_1=0$, when there is no pole. This reduced pole order however would not affect an integral over $s_1$, because it only occurs for a measure-zero set of $s_1$ values. 

With the above discussion in mind, we note that the generic order of the zero of the analytic function
\[
\prod_{I\in\CS(2k)\setminus\CI_{N+1} }(\zeta^{e_I})^{(h_{N,I})}(1+L_{N,I}(\bs s')) 
\]
at $s_{2k-N}'=C(s_1',\dots,s_{2k-N-1}')$ is 0. Indeed, for this product to vanish we must have that $L_{N+1,I}(\bs s')=0$, which happens non-generically when $I\in\CS(2k)\setminus \CI_{N+1}$.

Next, let $\nu$ be the generic order of the zero of the analytic function 
\eq{nu-dfn}{
 G(\bs s')
	 \prod_{\substack{I\in \CI_{N+1}\setminus\CI_N \\ e_I=-1,\, h_{N,I}\ge2}} 
		\left(\frac{1}{\zeta}\right)^{(h_{N,I})}(1+L_{N,I}(\bs s')) 
}
at $s_{2k-N}'=C(s_1',\dots,s_{2k-N-1}')$. If $d=0$ and $h_{N,I}=0$ for all $I\in\CS^-(2k)\cap(\CI_{N+1}\setminus\CI_N)$, then the function in \eqref{nu-dfn} equals $F(L_{N,\{1\}}(\bs s'),\dots,L_{N,\{2k\}}(\bs s'))$, which does not vanish in $\Omega_{2k}$, so that $\nu=0$.

From the above discussion, we conclude that the generic order of the pole of the integrand of $J(R)$ at $s_{2k-N}'=C(s_1',\dots,s_{2k-N-1}')$ is
\eq{pole-order}{
m &= \sum_{\substack{ I\in\CI_{N+1}\setminus\CI_N \\ \#I=\text{even} }} (h_{N,I}+1)
	- \sum_{\substack{ I\in \CI_{N+1}\setminus\CI_N \\ e_I=-1,\, h_{N,I}=0}} 1
	+ \sum_{\substack{ 1\le j\le 2k \\ \{j\}\in\CI_{N+1}\setminus\CI_N}} (h_{N,\{j\}}+A) - \nu \\
	&=\sum_{I\in\CI_{N+1}\setminus \CI_N} (h_{N,I}  + (-1)^{\#I}) 
		+ (A+1)\sum_{j\in[2k], \, \{j\}\in \CI_{N+1}\setminus\CI_N} 1 \\
	&\qquad - \nu -  
	\sum_{\substack{ I\in\CS^-(2k)\cap(\CI_{N+1}\setminus\CI_N) \\ h_{N,I}\ge2,\,\#I\ge3 }} (h_{N,I}-1) .
}
Note that it could be the case that $m\le 0$, in which case there is no pole contribution to $J(R)$ from the pole with $L_{N,I_{N+1}}(\bs s')=0$. 

Assume, now, that $m\ge1$. Then, $m+H_N\ge 1+N+d$ by our assumption that $H_N\ge N+d$. Moreover,
\eq{poweroflog}{
m+H_N
	&= \sum_{I\in\CI_{N+1} \setminus\{\emptyset\} } (-1)^{\#I} 
		- \sum_{I\in\CS(2k)\setminus \CI_{N+1}} h_{N,I}
			+  (A+1)\sum_{\substack{j\in[2k], \, \{j\}\in\CI_{N+1}}} 1 \\
			&\qquad 	- \nu
	-  \sum_{\substack{ I\in\CS^-(2k)\cap(\CI_{N+1}\setminus\CI_N) \\ h_{N,I}\ge1,\,\#I\ge3 }} (h_{N,I}-1) .
}
In order to continue, we separate two subcases depending on whether $N=2k-1$ or $N\le 2k-2$.

\medskip

\noindent 
{\bf Case 1 of the proof of Lemma \ref{contour-lem1}: $N=2k-1$.} 
In this case, we have that $s_j' = L_{2k-1,\{j\}}(s_1') = a_j s_1'$ for all $j$, where $a_j\in\Q$. Thus, the only potential pole is at $s_1'=0$. If $m\ge1$, so that there a genuine pole at $s_1'=0$, then we obtain an evaluation of $J(R)$ as a finite linear combination of powers of $\log R$ (up to an error term of size $O((\log R)^{O(1)}/T)$), the highest of which has exponent
\als{
H_{2k-1}+m-2k-2kA -d 
	&= \sum_{I\in\CS^*(2k)} (-1)^{\#I} 
		- \nu -  \sum_{\substack{ I\in\CS^-(2k) \setminus\CI_{2k-1}\\ h_{2k-1,I}\ge2,\,\#I\ge3 }} (h_{2k-1,I}-1)- d \le - 1 ,
}
since $\CI_{2k}=\CS(2k)$ in this case. We have thus written $J(R)$ as a linear combination of irreducible fundamental components of level $2k$ and suitable type (taking $h_{2k,I}=h_{2k-1,I}$ and $I_{2k}=\{1\}$), up to a small error term. This proves Lemma \ref{contour-lem1} in this case. 

As an amusing remark, we note that the above exponent equals $\CE_{k,A}$ only when $A>\frac{1}{2k}\binom{2k}{k}-1$, $d=0$, $\nu=0$, $h_{2k-1,I}\in\{0,1\}$ for $I\in\CS^-(2k) \setminus\CI_{2k-1}$, and  $G(s)=F(a_1s,\dots,a_{2k}s)$, in which case the residue is
\[
G(0) = A!^{2k} F(0,\dots,0) = A!^{2k} .
\]
Otherwise, these poles contribute towards the error term of $I_{g,2k}(R)$. 

\medskip

\noindent
{\bf Case 2 of the proof of Lemma \ref{contour-lem1}: $N\le 2k-2$.} 
Then the contribution of the pole $s_{2k-N}'=C(s_1',\dots,s_{2k-N-1}')$ to $J(R)$ equals
\als{
\frac{(\log R)^{H_N - N - d - 2kA}}{(2i\pi)^{2k-N-1}m!}  
		 \idotsint\limits_{\substack{ \Re(s_j')= \lambda_j'/\log R,
		 \\  |\Im(s_j')|\le T \\ 1\le j\le 2k-N-1}}
	 \left.	\frac{\dee^{m-1}}{\dee (s_{2k-N}')^{m-1}} \right|_{s_{2k-N}'=C(s_1',\dots,s_{2k-N-1}')}  
	 	(Z(\bs s'))
		 	\dee s_{2k-N-1}'\cdots \dee s_1' ,
}
where
\[
Z(\bs s'):=G(\bs s')R^{E_N(\bs s')} \\
		  (s_{2k-N}'-C(s_1',\dots,s_{2k-N-1}'))^m
		  \prod_{I\in\CS(2k)\setminus\CI_N}  
		 	\left(\zeta^{e_I}\right)^{(h_{N,I})}(1+L_{N,I}(\bs s')) .
\]
Applying the generalized product rule and writing $s_j$ in place of $s_j'$, we claim that the above integral can be expressed as a finite sum of terms of the form
\als{
c\cdot \frac{(\log R)^{H_N +m - h - N-1 - d - 2kA}}{(2i\pi)^{2k-N-1}}  
	&  \idotsint\limits_{\substack{ \Re(s_j')= \lambda_j'/\log R,\,  |\Im(s_j')|\le T \\ 1\le j\le 2k-N-1}}
		\tilde{G}(\bs s) R^{E_{N+1}(\bs s)}  \\
		&\times \prod_{I\in\CS(2k)\setminus\CI_{N+1}}
		 	\left(\zeta^{e_I}\right)^{(h_{N+1,I})}(1+L_{N+1,I}(\bs s)) 
		 \dee s_{2k-N-1}\cdots \dee s_1  ,
}
where:
\begin{itemize}
\item $c\ll1$; 
\item $h\in\{0,\dots,m-1\}$; 
\item $h_{N+1,I}\ge h_{N,I}$ with equality if $I\in \CI_{N+1}\setminus\{\emptyset\}$;
\item $\sum_{I\in \CS(2k)\setminus\CI_{N+1}} (h_{N+1,I}-h_{N,I}) \le h $;
\item $\tilde{G}$ is in the class $\CC_{2k-N-1}$. 
\end{itemize}

The first four claims are easy to verify, but our claim about $\tilde{G}$ requires some justification. To simplify the notation, we make the change of variables $s_{2k-N}'=\tau+C(s_1',\dots,s_{2k-N-1}')$. Let $\delta_{N,I}$ denote the coefficient of $x_{2k-N}'$ in the linear form $L_{N,I}$. If $I\in \CI_{N+1}$, so that $L_{N+1,I}=0$, we find that $L_{N,I}(s_1',\dots,s_{2k-N}')=\delta_{N,I}\tau$. So, if $I\in\CI_{N+1}\setminus\CI_N$, then $\delta_{N,I}\neq0$. Finally, we let $G_1(s_1',\dots,s_{2k-N-1}',\tau)$ to be the function $G(s_1',\dots,s_{2k-N}')$ after our change of variables. If $\nu_1$ denotes the generic order of the zero of $G_1(s_1',\dots,s_{2k-N-1}',\tau)$ at $\tau=0$, then the function $\tilde{G}$ will simply be a linear combination the functions
\[
\left.\frac{\partial^j}{\partial \tau^j}\right|_{\tau=0}(\tau^{-\nu_1}G_1(s_1',\dots,s_{2k-N-1}',\tau))
	= \frac{j!}{(j+\nu_1)!} \cdot \frac{\partial^{j+\nu_1}G_1}{\partial \tau^{j+\nu_1}}(s_1',\dots,s_{2k-N-1}',0). 
\]
Since $G$ is in the class $\CC_{2k-N}$, the above functions are in the class $\CC_{2k-N-1}$, which proves our claim about $\tilde{G}$.

It is straightforward to verify that $\bs h'$ and $\bs I'$ satisfy the required properties. Thus it remains to show that there is a suitable $d'$.

Now, relation \eqref{poweroflog} implies that the exponent of $\log R$ is $H_{N+1}-(N+1)-d'-2kA$, with 
\[
d' = d+\nu + \sum_{\substack{ I\in\CS^-(2k)\cap(\CI_{N+1} \setminus\CI_N) \\ h_{N,I}\ge2,\,\#I\ge3 }} (h_{N,I}-1)
	+ h-\sum_{I\in \CS(2k)\setminus\CI_{N+1}} (h_{N+1,I}-h_{N,I}) \ge 0 .
\]
Moreover, we have that
\[
H_{N+1}-d' = H_N-d+m-h\ge N+1
\quad\implies\quad H_{N+1}\ge d'+N+1,
\]
as needed. Finally, if $d'=0$, it is easy to check that $G(\bs s)=F(L_{N+1,\{1\}}(\bs s),\dots,L_{N+1,\{2k\}}(\bs s))$.

\medskip

(b) Here $\gamma_{j_{N+1}}<0$. We then shift the contours of $s_{2k-N},s_{2k-N-1},\dots,s_{j_{N+1}}$ in this order to the lines $\Re(s_j) = \lambda_j/(\log T)^{3/2}$, $j_{N+1}\le j\le 2k-N$. If $\lambda$ is large enough, then we do not encounter any poles and the horizontal lines contribute $\ll (\log R)^{O(1)}/T$ when we make this shift. Finally, when $\Re(s_j) = \lambda_j/(\log T)^{3/2}$ for $j_{N+1}\le j\le 2k-N$, and $\Re(s_j) = \lambda^j/\log R$ for $1\le j<j_{N+1}$, then we have that
\[
\Re(E(\bs s)) = - \frac{|\gamma_{j_{N+1}}|\lambda_{j_{N+1}}(1+O(1/\lambda))}{(\log T)^{3/2}} ,
\]
so that our integrand is $\ll e^{-\sqrt{\log R}}$ if $\lambda$ is large enough. We thus find that in this case
\[
J(R) \ll T^{-1+o(1)} ,
\]
as needed.
\end{proof}

\begin{rem}
Case 1 is feasible for some choice of $I_1,\dots,I_{2k-1}$. Indeed, if $k=1$ and $A\ge1$, so that $N=1$, then we note that at least one of the zeta factors must have survived in the numerator after shifting the $s_2$ contour, and there are none in the denominator, so there is a pole at $s_1'=0$. On the other hand, if $k\ge2$, then if $I_1=\{2k-1,2k\}$, $I_2=\{2k-1,2k-2\}$ and $I_3=\{2k-2,2k\}$, then $a_{2k}=a_{2k-1}=a_{2k-2}=0$, and $a_1\neq0$. Taking $h_{2k-1,I}=0$ for all $I$, we see that
\als{
m = \sum_{I\in\CS(2k)\setminus \CI_{2k-1}}(-1)^{\#I} +(A+1) \sum_{j\in[2k],\, a_j\neq0}1 
	=  (A+1) \sum_{j\in[2k],\, a_j\neq0}1\ge 1,
}
where we used Proposition \ref{CombProp}(a). Thus we see indeed that there is a genuine pole at $s_1'=0$. 
\end{rem}

It remains to prove the second intermediate step in the proof of \eqref{kthmoments-goal}:

\begin{proof}[Proof of Lemma \ref{contour-lem2}] We separate into three cases depending on whether $N=2k$, $N=2k-1$ or $N\le 2k-2$.

\medskip

\noindent {\bf Case 1 of the proof of Lemma \ref{contour-lem2}: $N=2k$.} Here there is no integral and we have that $J(R)=c(\log{R})^{H_{2k}-2k(A+1)-d}$.
Since $\CI_{2k}=\{I\subset[2k]\}$, we find that $H_{2k}=-1+2k(A+1)$, whence $J(R) = (\log R)^{-d-1}$. If $d=0$ and $\CE_{k,A}=-1$, the lemma follows with $c=1$; otherwise, we take $c=0$.

\medskip

\noindent {\bf Case 2 of the proof of Lemma \ref{contour-lem2}: $N=2k-1$.} 
In this case $J(R)$ is given by a one-dimensional integral over $s_1$. Moreover, there exist coefficients $a_j\in\Q$ such that $x_j= L_{2k-1,\{j\}}(x_1) = a_j x_1$ for all $j$. Therefore the only possible pole of the integrand in $J(R)$ is when $s_1=0$. If there is such a pole, then it means that $\{1\}\notin \CI_{2k-1}$. In this case, the order of this potential pole, say $m$, would be given by \eqref{pole-order} with $N=2k-1$, $\CI_{2k}=\CS(2k)$ and $\nu$ defined analogously, so that \eqref{poweroflog} implies that
\eq{pole-calc}{
m+H_{2k-1}-(2k-1)-2kA  
	& = 1+ \sum_{I\in\CS^*(2k)} (-1)^{\#I} 
		-  \sum_{\substack{ I\in\CS^-(2k) \setminus\CI_{2k-1} \\ \#I\ge3 ,\, h_{2k-1,I}\ge2 }} (h_{2k-1,I}-1) 
		- \nu \\
	&\le 0  ,
}
First, let us assume $m\ge1$ (i.e. there is a genuine pole at $s_1=0$). We find that $H_{2k-1}+1-2k(A+1) \le -1$. We then move the line of integration of $s_1=\sigma_1+it_1$ to the contour $\sigma_1=1/(\log(2+|t_1|))^{3/2}$, $|t_1|\le T$. No poles are encountered and the horizontal integrals contribute $\ll(\log R)^{O(1)}/T$. Moreover, the integral converges fast enough now (even when $A=0$) that we may remove the condition $|t_1|\le T$ at the cost of an error term of size $\ll (\log R)^{O(1)}/T$. We thus conclude that
\als{
J(R)
	&= c\cdot (\log R)^{H_{2k-1}-2k(A+1)+1-d } + O(T^{-1+o(1)})  ,
}
where
\[
c=  \int\limits_{\sigma_1=1/(\log(2+|t_1|))^{3/2}}
	G(s_1) \prod_{I\in\CS(2k)\setminus\CI_{sk-1}}
		 	\left(\zeta^{e_I}\right)^{(h_{2k-1,I})}(1+a_Is_1) 
		 \dee s_1  
\]
is some constant. This contributes towards the error term if $m\ge2$, $d\ge1$ or $A\le\frac{1}{2k}\binom{2k}{k}-1$, and towards the main term if $m=1$, $d=0$, $A>\frac{1}{2k}\binom{2k}{k}-1$ and $h_{2k-1,I}\in\{0,1\}$ for each $I\in\CS^-(2k)$, in which case $H_{2k-1}-2k(A+1)+1=-1$ by \eqref{pole-calc}. 

Alternatively, assume that $m\le 0$, so that there is no pole at $s_1=0$. We move $s_1$ to the line $\Re(s_1) = 0$. The horizontal lines contribute $\ll (\log R)^{O(1)}/T$. Furthermore, we note that we may extend the range of integration to all $s_1\in\C$ with $\Re(s_1)=0$ at the cost of an error term of size $\ll(\log R)^{O(1)}/T$. Consequently,
\[
J(R) = 
	c\cdot (\log R)^{H_{2k-1}+1-2k(A+1)-d} 
		+ O((\log R)^{O(1)}/T) ,
\]
where
\[
c = \frac{1}{2\pi} 
	\int_{-\infty}^\infty G(it) 
		\prod_{I\in\CS(2k)\setminus\CI_{2k-1}}  
		 	\left(\zeta^{e_I}\right)^{(h_{2k-1,I})}(1+ i a_I t) 
				 	\dee t 
\]
with $a_I:=\sum_{j\in I} a_j$. The power of $\log R$ is
\als{
H_{2k-1}+1-2k(A+1) -d 
	&= 1+ \sum_{I\in \CS_{2k-1}\setminus\{0\}} (-1)^{\#I} 
		- \sum_{I\in\CS(2k)\setminus \CI_{2k-1}} h_{2k-1,I} \\
	&\qquad	- (A+1)\cdot \#\{1\le j\le 2k:  a_j\neq0 \}  - d\\
	&=  1+ \SA(V_{2k-1}) 
	- \sum_{I\in\CS(2k)\setminus \CI_{2k-1}} h_{2k-1,I}  \\
	&\qquad - (A+1)\cdot \#\{1\le j\le 2k:  a_j\neq0 \}  -d 
}
in the notation of Proposition \ref{CombProp}. Clearly, this is maximized when $h_{2k-1,I}=0$ for all $I\in\CS(2k)\setminus \CI_{2k-1}$ and $d=0$, in which case 
\eq{G-formula}{
G(s)=F(a_1s,\dots,a_{2k}s)
	= \frac{P(a_1s,\dots,a_{2k}s)}{\prod_{j=1}^{2k} (a_js\zeta(1+a_js))^{A+1}} 
		+ O\left( \frac{(\log(2+|s|))^{O(1)}}{(1+|s|)^{(2k-1)(A+1)}}\right) 
}
when $\Re(s)=0$, by \eqref{h-mellin}.

Now, if $a_j=0$ for some $j\in[2k]$, then $\SA(V_{2k-1}) =-1$, by Proposition \ref{CombProp}(a). Note that there is at least one $j$ such that $a_j\neq0$; otherwise, the dimension of $V_{2k-1}$ would be $2k$, as it would contain the independent forms $s_1,\dots,s_{2k}$, which is a contradiction. We conclude that the power of $\log R$ is $\le -(A+1)\cdot \#\{1\le j\le 2k:  a_j\neq0 \}\le -A-1$. Consequently,
\[
J(R)  \ll (\log R)^{-A-1}
\]
in this case, which contributes towards the error term (i.e. $c=0$ in this case).

Finally, assume that $a_j\neq0$ for all $j\in[2k]$. Since $\dim(V_{2k-1})=2k-1$, Proposition \ref{CombProp}(b) implies that the power of $\log R$ is 
\[
H_{2k-1}+1-2k(A+1) -d \le \SA(V_{2k-1})-\dim(V_{2k-1})- 2kA \le \binom{2k}{k} - 2kA,
\]
with the second inequality being an equality when half of the $a_j$'s equal $+1$ and the other half $-1$. 

Even though this not needed for the proof, we remark that when $a_j\neq0$ for all $j$, we can give an asymptotic formula for $J(R)$. For simplicity, let as assume that $a_j=1$ for $j\le k$ and $a_j=-1$ for $j>k$. Then $a_I=\#(I\cap[1,k])-\#(I\cap(k,2k])$, which has the same parity as $\#I$. In particular, $\CI^-(2k)\cap\CI_{2k-1}=\emptyset$, so that $h_{2k-1,I}=0$ for all $I\in\CS^-(2k)$. Moreover, given $\ell\in\Z$, we have that $a_I=\ell$ for exactly $\binom{2k}{k+|\ell|}$ sets $I\subset[2k]$. Therefore
\als{
J(R)
	&=  \frac{(\log R)^{\binom{2k}{k}-2kA}}{2\pi} 
	\int_{-\infty}^\infty \frac{P(it,-it,\dots,it,-it)}{t^{2k(A+1)}}
		 \cdot \frac{\prod_{\substack{  \ell \ \text{even}, \ge 2}} |\zeta(1+i\ell t)|^{2\binom{2k}{k+\ell}}}
 {\prod_{\substack{  \ell \ \text{odd}, \ge 1 }}  |\zeta(1+i\ell t)|^{2\binom{2k}{k+\ell}}}	\dee t  \\
 	&\qquad
					+ O(T^{-1+o(1)}). 
}
This completes the study of Case 2.

\medskip

\noindent {\bf Case 3 of the proof of Lemma \ref{contour-lem2}: $N\le 2k-2$.} 
We shift the contours of $s_{2k-N},s_{2k-N-1},\dots,s_1$ in this order to the lines $\Re(s_j) = \lambda^j/(\log T)^{3/2}$, $1\le j\le 2k-N$. If $\lambda$ is large enough in terms of $k$ (but independently of $R$), then the functions $\Re(L_{N,I}(\bs s))$ with $I\in\CS(2k)\setminus \CI_N$ have constant sign in the entire domain where the contour shifting is performed, so that no poles are encountered. The horizontal lines contribute $\ll (\log R)^{O(1)}/T$. Finally, we note that the integrand on the new lines of integration is $\ll(\log\log R)^{O(1)}/[(1+|s_1|)\cdots(1+|s_{2k-N}|)]$, by \eqref{zeta-bound} and our assumption that $G$ is in the class $\CC_{2k-N}$. We thus find that
\[
J(R) \ll (\log R)^{H_N-N-2kA-d} (\log\log R)^{O(1)} 
\]
in this case. We need to understand the power of $\log R$. Firstly, note that
\als{
H_N-N-2kA-d
	&\le 2k-N  + \sum_{I\in\CI_N\setminus\{\emptyset\}}(-1)^{\#I} 	
		- (A+1) \cdot \#\{1\le j\le 2k: \{j\}\notin \CI_N\}  .
}
To continue, we separate two cases.

If there is $\{j\}\in \CI_N$, then $\sum_{I\in\CI_N\setminus\{\emptyset\}}(-1)^{\#I} =-1$ by Proposition \ref{CombProp}(a). Since $\dim(V_N)=N$ by construction, there are $\le N$ integers $j$ with $\{j\}\in\CI_N$. The power of $\log R$ is thus
\[
\le 2k-N -1 - (A+1)(2k-N) =- 1 - A(2k-N) \le -1 - 2\cdot {\bf 1}_{A\ge1} ,
\]
and we can see that $J(R)$ satisfies the conclusion of the lemma with no main term (i.e. $c=0$). So assume that there is no $\{j\}\in \CI_N$. Then we find that
\[
H_N-N-2kA
	\le 2k-N  + \sum_{I\in\CI_N\setminus\{0\}}(-1)^{\#I} - 2k(A+1) 	
	\le \binom{2k}{k} - 2k(A+1)-2,
\]
by Proposition \ref{CombProp}(c), which again means that the lemma holds with $c=0$.
\end{proof}


\subsection{Dyadic intervals}\label{dyadic-section} We conclude this section with a brief explanation of the proof of \eqref{kthmoments-dyadic}. We have that
\[
\CM_{\tilde{f_0},2k}(R) = \CM_{\tilde{h},2k}(R) +O\left(\frac{1}{\log R}\right) ,
\]
where $\tilde{h}(x) = h(x) - h(x+\frac{\log 2}{\log R})$, by the argument leading to \eqref{f_0-g}. We then note that 
\[
\hat{(\tilde{h})}_R(s) = \hat{h}_R(s)(1-2^{-s}),
\]
so that Perron's inversion formula and relation \eqref{IBP} imply that 
\als{
\CM_{\tilde{f_0},2k}(R)	
	&=    \frac{1}{(2i\pi)^{2k}} \idotsint\limits_{\substack{\Re(s_j)=\lambda^j/\log R \\ 1\le j\le 2k}}
   D(\bs s) 
	\left( \prod_{j=1}^{2k} \hat{h}_R(s_j)(1-2^{-s_j}) \right) ds_{2k}\cdots ds_1  
	+  O\left(\frac{1}{\log R}\right) ,
}
where $\lambda$ and $T$ are as before. We thus find that
\eq{dyadic-start}{
\CM_{\tilde{f_0},2k}(R)	
	&=    \frac{1}{(2i\pi)^{2k}} \idotsint\limits_{\substack{\Re(s_j)=\lambda^j/\log R \\ 1\le j\le 2k}}
			 \tilde{F}(\bs s)R^{s_1+\cdots+s_{2k}}
		 \prod_{I\in \CS^*(2k)}  \zeta(1+s_I)^{(-1)^{\#I}} \dee s_1\cdots \dee s_{2k} \\
	&\qquad\qquad	+  O\left(\frac{1}{\log R}\right) ,
}
where
\[
\tilde{F}(\bs s) : =  P(\bs s) \prod_{j=1}^{2k} \frac{\hat{h}_R(s_j)(1-2^{-s_j})}{R^{s_j}} 
\]
and $P$ is defined as above. The function $\tilde{F}$ is in the class $\CC_{2k}$, since the factor $1-2^{-s_j}$ annihilates the pole of $\hat{h}_R(s_j)$ at $s_j=0$. We thus see that the above integral has the same shape as the integral $I_{g,2k}(R)$ with $A=0$, with the difference that $e_I=-1$ when $\#I=1$. We thus follows the argument leading to \eqref{kthmoments} when $A=0$ with the obvious modifications. The only difference is that in the analogue of \eqref{pole-order} we have instead
\als{
m &= -\nu + \sum_{\substack{ I\in\CI_{N+1}\setminus\CI_N \\ \#I=\text{even} }} (h_{N,I}+1)
	- \sum_{\substack{ I\in \CI_{N+1}\setminus\CI_N \\ \#I=\text{odd},\, h_{N,I}=0}} 1  \\
	&=-\nu+ \sum_{I\in\CI_{N+1}\setminus \CI_N} (h_{N,I}  + (-1)^{\#I}) 
	-  \sum_{\substack{ I\in\CS^-(2k)\cap(\CI_{N+1}\setminus\CI_N) \\ h_{N,I}\ge2,\,\#I\ge3 }} (h_{N,I}-1) ,
}
with $\nu$ defined as in the proof of Lemma \ref{contour-lem1}. We thus find that $m$ has the same expression as when $A=-1$, and relation \eqref{kthmoments-dyadic} follows by the proof of \eqref{kthmoments-goal} when $A=-1$. An important remark is that when $A=-1$ there is a power of $(\log R)^{-2k}$ in the denominator of the integrand of $I_{g,2k}(R)$ that is not present in the denominator of the right hand side of \eqref{dyadic-start}.

 \section{Lower bounds} \label{lb}

In this section we complete our proof of  Theorem \ref{mainthm} by showing that, for fixed $k\in\Z_{\ge 1}$, $A\in\Z_{\ge0}$ and $\epsilon>0$, there are positive constants $c_{k,A}'>0$ and $c_k''>0$ such that
\eq{kthmoments-lb}{
\CM_{f_A,2k}(R) \ge 
	c_{k,A}' (\log R)^{\CE_{k,A}-\epsilon}  + O_\epsilon((\log R)^{\CE_{k,A}-1}),
}
and
\eq{kthmoments-lb-dyadic}{
\CM_{\tilde{f}_0,2k}(R) \ge 
	c_k'' (\log R)^{D_{k,0}-\epsilon}  + O_\epsilon((\log R)^{D_{k,0}-1}) ,
}
where $\tilde{f_0}(x)=f_0(x)-f_0(x+\frac{\log 2}{\log R})$, as before. Evidently, this completes the proof of Theorem \ref{mainthm}, since if the constants in the leading terms in \eqref{kthmoments} or \eqref{kthmoments-dyadic} were 0, we would obtain a contradiction to the above lower bounds by letting $R\to\infty$.

As in the previous section, there is some smooth smooth function $h$ such that $h(x)=1$ for $x\le 1-1/(\log R)^C$ and $h(x)=0$ for $x\ge1$, with $C=(2k-1)2^{2k+1}+2k+2$, so that
\[
\CM_{f_0,2k}(R) = \CM_{h,2k}(R) +O\left(\frac{1}{\log R}\right) 
\quad\text{and}\quad 
\CM_{\tilde{f_0},2k}(R) = \CM_{\tilde{h},2k}(R) +O\left(\frac{1}{\log R}\right) ,
\]
where $\tilde{h}(x)=h(x)-h(x+\frac{\log 2}{\log R})$. So, it suffices to prove that
\eq{kthmoments-lb-unified-1}{
\CM_{g,2k}(R) \ge 
	c_{k,A}' (\log R)^{\CE_{k,A}-\epsilon}  + O((\log R)^{\CE_{k,A}-1}),
}
where $g \in \{h,\tilde{h}\}$ when $A=0$, and $g=f_A$ when $A\ge1$.

Positivity is a key to this proof: we will consider the sum restricted to those integers with a convenient prime factorization, which clearly provides a lower bound. The fact that these integers have a convenient prime factorization means the corresponding sum is technically easier to analyze.

\subsection{First manipulations}
To ease notation, let
\[
\Pi_R:=\prod_{p\le R} \left(1-\frac{1}{p}\right)
\]

We start by observing that
\[
\CM_{g,2k}(R) = \Pi_R
		\sum_{P^+(n)\le R}  \frac{1}{n} 
			\left( \sum_{ d| n }
			\mu(d)g\left(\frac{\log d}{\log R}\right) \right)^{2k} ,
\]
since $\text{supp}(g)\subset(-\infty,1]$. So \eqref{kthmoments-lb} follows immediately when $A>-1+\frac{1}{2k}\binom{2k}{k}$ by noticing that $\CE_{k,A}=-1$ in this case and that we always have that $\CM_{g,2k}(R)\ge\Pi_R \gg1/\log R$. 

For the rest of the proof, we assume that $A\le-1+\frac{1}{2k}\binom{2k}{k}$, so that
\[
\CE_{k,A} = \binom{2k}{k} - 2k(A+1) \ge 0  .
\]
Let $q_1<q_2<\cdots$ be the sequence of all prime numbers and set 
\[
Q = \prod_{j=1}^{A+1} q_j .
\]
If $A\ge1$, or if $A=0$ and $g=h$, then we restrict our attention to integers of the form $Qn$ with $P^-(n)>Q$, so that
\[
\CM_{g,2k}(R) \ge \frac{\Pi_R}{Q}
		\sum_{p|n\,\Rightarrow\, q_{A+1}<p\le R} \frac{1}{n}\left( \sum_{J\subset[A+1]} (-1)^{\#J}
			\sum_{d|n} \mu(d)  g\left(\frac{\log(q_Jd)}{\log R} \right) \right)^{2k} ,
\]
where $q_J:=\prod_{j\in J} q_j$. We define
\eq{w-def}{
w(x) := \sum_{J\subset[A+1]} (-1)^{\#J} g\left(x+ \frac{\log q_J}{\log R}\right) ,
}
so that
\eq{lb-e1}{
\CM_{g,2k}(R) \ge \frac{\Pi_R}{Q}
		\sum_{p|n\,\Rightarrow\, q_{A+1}<p\le R} 
			\frac{1}{n}\left( \sum_{d|n} \mu(d)w\left(\frac{\log d}{\log R} \right) \right)^{2k} .
}
We further note that $w=\tilde{h}$ when $A=0$, so that the right hand side of \eqref{lb-e1} is a trivial lower bound for $\CM_{\tilde{h},2k}(R)$. Therefore, relation \eqref{kthmoments-lb-unified-1} will follow in all cases if we can show that
\eq{kthmoments-lb-unified}{
W&:=  \Pi_R \sum_{p|n\ \Rightarrow\ q_{A+1}<p\le R}
	\frac{1}{n}\left( \sum_{d|n} \mu(d)w\left(\frac{\log d}{\log R} \right) \right)^{2k}  \\
&\ge c_{k,A}'' (\log R)^{\CE_{k,A}-\epsilon}  + O_\epsilon((\log R)^{\CE_{k,A}-1}) 
}
for some $c_{k,A}''>0$, where $w$ is defined by \eqref{w-def} with $g=h$ if $A=0$ and $g=f_A$ if $A\ge1$. 

\medskip

Next, observe that
\eq{w-mellin}{
\hat{w}_R(s) = (\log R)\sum_{J\subset[A+1]} (-1)^{\#J} 
	\int_{-\infty}^{\infty} g\left(u + \frac{\log q_J}{\log R}\right) R^{su} \dee u
	= \hat{g}_R(s) \prod_{j=1}^{A+1}(1-q_j^{-s}) .
}
In particular, we see that $\hat{w}_R$ has an analytic continuation to $\C$ and it satisfies the bound
\eq{w-bound-1}{
\hat{w}_R(s) \ll \frac{R^{\Re(s)}}{(|s|+1)(\log R)^A} \quad(\Re(s)\ge -1 ) ,
}
which follows by the definition of $f_A$ when $A\ge1$ and by \eqref{h-bound-prelim} otherwise. Finally, we note that we also have the bound
\eq{w-bound-2}{
\hat{w}_R(s) \ll \frac{R^{\Re(s)}(\log R)^{O(1)}}{1+|s|^2} \quad(\Re(s)\ge 1/\log R),
}
which follows from \eqref{IBP} when $A=0$. (This bound can be shown to hold in a larger range, but the above range is good enough for our purposes.)

Before we apply Perron's inversion formula to write the right hand side of \eqref{lb-e1} in terms of $\hat{w}_R$, we use positivity again to focus on integers $n$ of a certain convenient form. We set
\[
y=\exp\{(\log R)^{1-\epsilon'}\}
\quad\text{and}\quad
Y=\exp\{(\log R)^{1-\epsilon'/2}\},
\]
where $\epsilon'>0$ will be taken to be small enough in terms of $\epsilon$, and write 
\[
\CN = \{n\in\Z_{\ge 1} : p|n\,\Rightarrow\, q_{A+1}<p\le y\}.
\]
We then focus on integers of the form $n=mp_1\cdots p_k$, where $m\in\CN$ with $m\le Y$, and $p_1,\dots,p_k$ are distinct primes from the interval $(R^{1/2},R]$. For such an integer $n$, if $d|n$, then either $d=d'$ or $d=d'p_\ell$, for some $d'|m$ and some $\ell\in\{1,\dots,k\}$. So, we conclude that
\als{
W     & \ge
			\frac{\Pi_R}{k!} 
			\sum_{\substack{m\in\CN \\ m\le Y}}
			\sum_{\substack{\sqrt{R}<p_1,\dots,p_k \le R \\ \text{distinct primes}}}
				\frac{1}{mp_1\cdots p_k} 	\\
	&\qquad\times	\left( \sum_{\ell=1}^k \sum_{d|m} \mu(d) w\left( \frac{\log(p_\ell d)}{\log R} \right) 
			-  \sum_{d|m} \mu(d) w\left( \frac{\log d}{\log R} \right) \right)^{2k} .
}
Note that the condition that $m\le Y=R^{o(1)}$ certainly implies that $d<R/(2q_{A+1})$ for all divisors $d$ of $m$ (since $q_{A+1}\ll_A 1$), so that 
\als{
\sum_{d|m} \mu(d) w\left( \frac{\log d}{\log R} \right) 
	&= \sum_{J\subset[A+1]} (-1)^{\#J} 
		 \sum_{d|m} \mu(d) \left(1-\frac{\log(q_Jd)}{\log R}\right)^A  \\
	&= \sum_{d|Qm} \mu(d) \left(1-\frac{\log d}{\log R}\right)^A  =0 
}
by observing that $\Delta^{(A+1)}(1-x)^A=0$ by \eqref{differencing3}, and by applying \eqref{differencing2} with $r=A+1$, since $\omega(Qm)\ge \omega(Q)\ge A+1$ here. We thus conclude that
\als{	W &\ge
		\frac{\Pi_R}{k!}
			\sum_{\substack{m\in\CN \\ m\le Y}}
			\sum_{\substack{\sqrt{R}<p_1,\dots,p_k \le R \\ \text{distinct primes}}}
				\frac{1}{mp_1\cdots p_k}
		\left( \sum_{\ell=1}^k \sum_{d|m} \mu(d) w\left( \frac{\log(p_\ell d)}{\log R} \right)\right)^{2k} \\
	&\ge \frac{\Pi_R}{k! (\log R)^k } 
			\sum_{\substack{m\in\CN \\ m\le Y}}
			\sum_{\substack{\sqrt{R}<p_1,\dots,p_k\le R \\ \text{distinct primes}}}
				\frac{\prod_{j=1}^k \log p_j}{mp_1\cdots p_k}
		\left( \sum_{\ell=1}^k \sum_{d|m} \mu(d) w\left( \frac{\log(p_\ell d)}{\log R} \right)\right)^{2k} .
}
We note that
\[
\sum_{n\in\CN,\, n>Y} \frac{\tau_r(n)}{n} \le Y^{-1/\log y}
	\sum_{n\in\CN} \frac{\tau_r(n)}{n^{1-1/\log y}}
	 \ll \frac{1}{(\log R)^B},
\]
for any fixed $B\ge1$ and $r\in\Z_{\ge 1}$. Therefore, we may drop the conditions that $m\le Y$ and that the $p_\ell$'s are distinct at the cost of an admissible error term, finding that
\als{
W
	&\ge \frac{\Pi_R}{k!(\log R)^k } 
			\sum_{m\in\CN}
			\sum_{\substack{\sqrt{R}<p_1,\dots,p_k\le R}} 
				\frac{\prod_{j=1}^k \log p_j}{mp_1\cdots p_k}
		\left( \sum_{\ell=1}^k \sum_{d|m} \mu(d) w\left( \frac{\log(p_\ell d)}{\log R} \right)\right)^{2k} \\
	&\qquad + O\left( \frac{1}{(\log R)^{100}} \right) .   
}

Next, we expand the $2k$-th power as follows:
\[
\left( \sum_{\ell=1}^k \sum_{d|m} \mu(d) w\left( \frac{\log(p_\ell d)}{\log R} \right)\right)^{2k} 
	= \sum_{\substack{J_1\cup\cdots\cup J_k=[2k]\\ J_i\cap J_j=\emptyset\text{ for }i\ne j}}
		\prod_{\ell=1}^k \prod_{j\in J_\ell} 
			\sum_{ d_j|m} \mu(d_j)  
				w\left( \frac{\log(p_\ell d_j)}{\log R} \right),
\]
with the convention that if $J_\ell=\emptyset$, then the corresponding factor equals 1. Clearly, the summands corresponding to those $k$-tuples $\bs J=(J_1,\dots,J_k)$ all of whose components have an even number of terms are non-negative. We write $\mathcal{J}$ for the set of $k$-tuples $\bs J$ such that $(J_1,\dots,J_k)$ is a partition of $[2k]$ and either $\#J_\ell=2$ for all $\ell$, or there is some $\ell$ such that $\#J_\ell$ is an odd number. Then
\[
\left( \sum_{\ell=1}^k \sum_{d|m} \mu(d) w\left( \frac{\log(p_\ell d)}{\log R} \right) \right)^{2k} 
	\ge \sum_{\bs J\in\mathcal{J}} 
		\prod_{\ell=1}^k 
			\prod_{j \in J_\ell}
				\sum_{ d_j|m} \mu(d_j)  w\left( \frac{\log(p_\ell d_j)}{\log R} \right).
\]
Moreover, if $D\in\CN$, then
\[
\sum_{\substack{m\in \CN \\ D|m}} \frac{1}{m} 
	= \frac{1}{D} \prod_{q_{A+1}<p\le y} \left(1-\frac{1}{p}\right)^{-1} 
		= \frac{c_1\prod_{p\le y}(1-1/p)^{-1}}{D} 
\]
with $c_1=\prod_{p\le q_{A+1}}(1-1/p)\asymp1$. So, we conclude that
\eq{lb-e2}{
W  &\ge \frac{c_1\prod_{y<p\le R}(1-1/p)}{k!(\log R)^k} 
		 \sum_{\bs J\in\mathcal{J}} 
		\sum_{d_1,\dots,d_{2k}\in \CN} \frac{\mu(d_1)\cdots \mu(d_{2k})}{[d_1,\dots,d_{2k}]} \\
		&\qquad\qquad\times \prod_{\ell=1}^k \sum_{\sqrt{R}<p_\ell\le R} \frac{\log p_\ell}{p_\ell}
				\prod_{j\in J_\ell} w\left( \frac{\log(p_\ell d_j)}{\log R} \right) 
		- O\left( \frac{1}{(\log R)^{100}} \right) .
}
For the convenience of notation, set
\[
W(\bs J)
	= \frac{1}{(\log R)^k} 	
		\sum_{d_1,\dots,d_{2k}\in \CN} \frac{\mu(d_1)\cdots \mu(d_{2k})}{[d_1,\dots,d_{2k}]} \\
		 \prod_{\ell=1}^k \sum_{\sqrt{R}<p_\ell\le R} \frac{\log p_\ell}{p_\ell}
				\prod_{j\in J_\ell} w\left( \frac{\log(p_\ell d_j)}{\log R} \right),
\] 
so that
\[
W \ge \frac{c_1\prod_{y<p\le R}(1-1/p)}{k!} 
		 \sum_{\bs J\in\mathcal{J}} W(\bs J) - O\left( \frac{1}{(\log R)^{100}} \right) .
\]
We will show that the dominant contribution comes from the terms $\bs J$ with $\#J_\ell=2$ for all $\ell$.

\subsection{Mellin transformation}
Fix a choice of sets $\bs J=(J_1,\dots,J_k)\in\mathcal{J}$ and let
\[
\CL = \{1\le \ell\le k: J_\ell\neq\emptyset\} .
\]

For $\ell\notin\CL$, the sum over $p_\ell$ is
\[
\sum_{\sqrt{R}<p_\ell\le R} \frac{\log p_\ell}{p_\ell} 
	= \frac{\log R}{2} + O\left(e^{-c\sqrt{\log R}}\right) .
\]
So 
\als{
W(\bs J) = \frac{2^{\#\CL-k}}{(\log R)^{\#\CL}}
		\sum_{d_1,\dots,d_{2k}\in \CN} \frac{\mu(d_1)\cdots \mu(d_{2k})}{[d_1,\dots,d_{2k}]}
			\prod_{\ell \in \CL} 
				\sum_{p_\ell>\sqrt{R}} \frac{\log p_\ell}{p_\ell}
					\prod_{j\in J_\ell} w\left( \frac{\log(p_\ell d_j)}{\log R} \right) 
\\
					+O\left(\frac{1}{(\log R)^{100}}\right) ,
}
where the condition that $p_\ell\le R$ was dropped because it is implied by the fact that $\text{supp}(w)\subset(-\infty,1]$. 

Next, we use Perron's formula $2k$ times to write each appearance of $w$ as an integral of $\hat{w}_R$. We thus find that
\al{
W(\bs J) = \frac{2^{\#\CL-k}(\log R)^{-\#\CL} }{(2\pi i)^{2k}}
		\idotsint\limits_{\substack{\Re(s_j)=1/\log R \\ 1\le j\le 2k}}\ &
		\sum_{d_1,\dots,d_{2k}\in \CN} \frac{\prod_{j=1}^{2k}\mu(d_j)d_j^{-s_j}}{[d_1,\dots,d_{2k}]}
			\left(\prod_{\ell\in\CL} \sum_{p_\ell>\sqrt{R}} \frac{\log p_\ell}{p_\ell^{1+s_{J_\ell}}} \right)\nonumber\\
		&\quad\times	
			\left( \prod_{j=1}^{2k} \hat{w}_R(s_j) \right) 
					\dee s_1\cdots \dee s_{2k} +O\left(\frac{1}{\log R}\right),
\label{eq:WJ}
}
with the notational convention that $s_J:=\sum_{j\in J}s_j$. By possibly re-indexing the variables $s_1,\dots,s_{2k}$, we may assume that $\CL=\{1,\dots,L\}$, where $L=\#\CL$, and that $\max J_\ell= 2k-L+\ell$ for all $\ell\in\{1,\dots,L\}$. We want to move the variables $s_{2k-L+1},\dots,s_{2k}$ to the left. First, we need some bounds on the sum over $p_\ell$. We note that
\[
 \sum_{p>\sqrt{R}} \frac{\log p}{p^{1+s}}
 	= -\frac{\zeta'}{\zeta}(1+s) + O(1) - \sum_{p\le R^{1/2}} \frac{\log p}{p^{1+s}} 
\]
for $\Re(s)\ge-1/3$. Using standard bounds on the Riemann zeta function (see, for example, Titchmarsch \cite[Theorem 3.11]{Tit}), we find that
\als{
 \sum_{p>\sqrt{R}} \frac{\log p}{p^{1+s}}
 	& \ll \frac{1}{|s|} + \log(2+|t|) + R^{\max\{0,-\sigma\}/2} \sum_{p\le \sqrt{R}} \frac{\log p}{p} \\
	&\ll R^{\max\{0,-\sigma\}/2} \log(R+|t|)  ,
}
where $s=\sigma+it$, as usual. Moreover, note that if $\sigma_j\ge -1/\log y$ for all $j\in\{1,\dots,2k\}$, then
\als{
\abs{\sum_{d_1,\dots,d_{2k}\in \CN} \frac{\prod_{j=1}^{2k}\mu(d_j)d_j^{-s_j}}{[d_1,\dots,d_{2k}]}}
	&\le \sum_{d_1,\dots,d_{2k}\in \CN} 
		\frac{\prod_{j=1}^{2k}\mu^2(d_j)d_j^{1/\log y}}{[d_1,\dots,d_{2k}]}	\\
	&\le \prod_{p\le y} \left( 1+ \frac{p^{1/\log y}(4^k-1)}{p}\right) \ll  (\log y)^{4^k-1},
}
using the estimate $p^{1/\log y}=1+O(\log p/\log y)$ for $p\le y$. 

We are now ready to move the variables $s_{2k-L+1},\dots,s_{2k}$ in \eqref{eq:WJ} to the left. First, we move the variable $s_{2k}$ to the line $\Re(s_{2k}) = -1/\log y$. (Here we can use \eqref{w-bound-2} to justify the convergence required for this manoeuvre.) We pick up a simple pole from the sum over $p_{2k}$ when $s_{J_L}=0$. The integrand when $\Re(s_{2k})=-1/\log y$ is
\[
\ll (\log y)^{4^k-1} (\log R)^{L-1} \log(R+|t|) R^{1/(2\log y)} \prod_{j=1}^{2k} \abs{\hat{w}_R(s_j)}
	\ll \frac{(\log R)^{O(1)}\log(2+|t|) R^{-1/(2\log y)}}{(|s_1|^2+1)\cdots(|s_{2k}|^2+1)} 
\]
by \eqref{w-bound-2}. So we find that
\als{
W(\bs J) = \frac{2^{L-k}(\log R)^{-L} }{(2\pi i)^{2k-1}}
		\idotsint\limits_{\substack{\Re(s_j)=1/\log R \\ 1\le j\le 2k-1 \\ s_{J_L}=0}}\ 
		\sum_{d_1,\dots,d_{2k}\in \CN} \frac{\prod_{j=1}^{2k}\mu(d_j)d_j^{-s_j}}{[d_1,\dots,d_{2k}]}
			\left(\prod_{\ell=1}^{L-1} \sum_{p_\ell>R^{1/2}} \frac{\log p_\ell}{p_\ell^{1+s_{J_\ell}}} \right)\\
		\times	
			\left( \prod_{j=1}^{2k} \hat{w}_R(s_j) \right) 
					\dee s_1\cdots \dee s_{2k-1} 
				+O\left(\frac{1}{(\log R)^{100}}\right) .
}
Now the product $\prod_{\ell=1}^{L-1} \sum_{p_\ell>R^{1/2}} (\log p_\ell)p_\ell^{-1-s_{J_\ell}}$ doesn't depend on the variables $s_j\in J_L\setminus \{2k\}$, and so we encounter no poles if we move all of these variables to the lines $\Re(s_j)=0$. Having done this, by \eqref{w-bound-1} the growth of $\hat{w}_R(s_j) $ only depends weakly on $R$.

Then we repeat the same argument by moving $s_{2k-1}$ to the left, then $s_{2k-2}$, and so on and so forth, until all the sums over primes have been removed and replaced by contributions coming from poles. Writing $s_j=it_j$, we conclude that
\al{
W(\bs J) =  \frac{2^{L-k}(\log R)^{-L} }{(2\pi)^{2k-L}}
		\idotsint\limits_{\substack{t_1,\dots,t_{2k}\in\R \\ t_{J_\ell} = 0\,(1\le \ell\le L)}} 
		\sum_{d_1,\dots,d_{2k}\in \CN} \frac{\prod_{j=1}^{2k}\mu(d_j)d_j^{-it_j}}{[d_1,\dots,d_{2k}]}
			\left( \prod_{j=1}^{2k} \hat{w}_R(it_j) \right)
					\dee t_1\cdots \dee t_{2k-L} \nn
		 + O\left(\frac{1}{(\log R)^{100}}\right) .
		\label{W(J)}
}

 We note that, for any $t_1,\dots,t_{2k}$, we have
\[
\sum_{d_1,\dots,d_{2k}\in \CN} \frac{\prod_{j=1}^{2k}\mu(d_j)d_j^{-it_j}}{[d_1,\dots,d_{2k}]}
	= \prod_{q_{A+1}<p\le y} \left(1+ \frac{\lambda_p(\bs t)}{p}\right) ,
\]
where
\eq{lambda-def}{
\lambda_p(\bs t) 
	:= \sum_{I\in \CS^*(2k)}  (-1)^{\#I} p^{-it_I}  
	= -1 + \prod_{j=1}^{2k} (1-p^{-it_j}).
}

\subsection{An auxiliary result}
Before we continue with the estimation of $W(\bs J)$, we establish a preliminary (and fairly standard) result.

\begin{lem}\label{zeta} For $z\ge y\ge3$ and $t\in\R$, we have that
\[
\sum_{y<p\le z} \frac{1}{p^{1+it}} 
	= \begin{cases}
		\ds \log\left(\frac{\log z}{\log y}\right) +O(1) &\text{if}\ |t|\le 1/\log z ,\\
		\\
		\ds \log\left(\frac{1}{|t|\log y}\right) +O(1) &\text{if}\ 1/\log z<|t|\le 1/\log y ,\\
		\\
		O(1)  &\text{if}\ y\ge |t|\ge 1/\log y .
	\end{cases}
\]
Finally, if $|t|\ge y$, then
\[
\abs{\sum_{y<p\le z} \frac{1}{p^{1+it}} }
	\le \log\left(\frac{\log( \min\{|t|,z\}) }{\log y}\right) +O(1) .
\]
\end{lem}

\begin{proof}
If $|t|\le 1/\log z$, then we note that $p^{-it}=1+O(\log p/\log z)$ for $p\le z$, so that
\[
\sum_{y<p\le z} \frac{1}{p^{1+it}} = \sum_{y<p\le z} \frac{1}{p} + O(1) 
	=  \log\left(\frac{\log z}{\log y}\right) + O(1),
\]
as claimed. For $|t|\ge1/\log y$ and $y\ge|t|$, then we note that
\[
\sum_{y<p\le z} \frac{1}{p^{1+it}} \ll 1
\]
by relation (4.4) in \cite{Kou13}. Next, if $1/\log z\le |t| \le 1/\log y$, then we apply the results we just proved to deduce that 
\[
\sum_{y<p\le e^{1/|t|}} \frac{1}{p^{1+it}}  
	= \log\left(\frac{\log e^{1/|t|}}{\log y}\right) + O(1) 
	= \log\left(\frac{1}{|t|\log y}\right) + O(1)
\]
and that
\[
\sum_{e^{1/|t|}<p\le z} \frac{1}{p^{1+it}}  \ll 1 .
\]
Finally, if $|t|\ge y$, then we note that
\[
\sum_{|t|<p\le z} \frac{1}{p^{1+it}} \ll 1,
\]
so that
\als{
\abs{\sum_{y<p\le z} \frac{1}{p^{1+it}} }
	= \abs{\sum_{y<p\le \min\{|t|,z\}} \frac{1}{p^{1+it}} }+O(1) 
		&\le\sum_{y<p\le \min\{|t|,z\}} \frac{1}{p} +O(1) \\
		&\le \log\left(\frac{\log( \min\{|t|,z\}) }{\log y}\right) +O(1) . \qedhere
}
\end{proof}

\subsection{Lower bound for the main term}\label{lb-even}
We now return to the study of the quantity $W(\bs J)$, defined by \eqref{W(J)}.
First, we show that the term when $\#J_\ell=2$ for all $\ell$ contributes what we claim to be our main term. In this case $L=\#\CL=k$ and, by possibly permuting the $t_j$'s, we may assume that $J_\ell=\{\ell,k+\ell\}$ for all $\ell\in\{1,\dots,k\}$. Thus for these terms we have $t_{k+\ell} = - t_\ell$. We want to show that the integrand in \eqref{W(J)} is non-negative for all choices of $t_1,\dots,t_k$.  We have that
\[
\prod_{j=1}^{2k} \hat{w}_R(it_j) 
	= \prod_{j=1}^{k} \hat{w}_R(it_j)  \hat{w}_R(-it_j) 
	= \prod_{j=1}^{k} \abs{\hat{w}_R(it_j)}^2
	= \prod_{j=1}^{k} \abs{t_j^{A+1} \hat{g}_R(it_j)}^2 \prod_{a=1}^{A+1} \frac{|1-q_a^{-it_j}|^2}{t_j^2} ,
\]
which is clearly non-negative for all $t_1,\dots,t_k$. Moreover, if $t_j\ll1$ for all $j$, then relations \eqref{w-mellin} and \eqref{h-mellin} imply that
\[
\prod_{j=1}^{2k} \hat{w}_R(it_j) \ge \frac{c_2}{(\log R)^{2kA}}
\]
for some $c_2>0$. Furthermore, the definition \eqref{lambda-def} implies that in this case we have 
\[
\lambda_p(\bs t) =- 1+ \prod_{j=1}^k|1-p^{-it_j}|^2 \ge -1,
\]
 so 
\[
\sum_{d_1,\dots,d_{2k}\in \CN} \frac{\prod_{j=1}^{2k}\mu(d_j)d_j^{-it_j}}{[d_1,\dots,d_{2k}]}
	= \prod_{q_{A+1}<p\le y} \left(1+ \frac{\lambda_p(\bs t)}{p}\right) \ge 0
\]
for such $\bs t$. 

Since the integrand in \eqref{W(J)} is non-negative, we may obtain a lower bound by restricting the range of integration to any region we wish. We restrict to $t_1\in[1,2]$ and $t_j \in[t_1,t_1+1/\log y]$ for $2\le j\le k$. The volume of this region is $\asymp 1/(\log y)^{k-1}$. Moreover, in this region we find that
\[
\lambda_p(\bs t)   = -1 + |1-p^{it_1}|^{2k} + O\left(\frac{\log p}{\log y}\right) .
\]
Therefore
\begin{align*}
\sum_{d_1,\dots,d_{2k}\in \CN} \frac{\prod_{j=1}^{2k}\mu(d_j)d_j^{-it_j}}{[d_1,\dots,d_{2k}]}
	&= \exp\left\{\sum_{p\le y} \frac{-1+|1-p^{it_1}|^{2k}}{p}+O\Bigl(\frac{1}{p^2}\Bigr) +O\Bigl(\frac{\log{p}}{p\log{y}}\Bigr)\right\}\\
	&\gg \frac{1}{\log{y}}\exp\left\{ \sum_{p\le y} \frac{|1-p^{it_1}|^{2k}}{p} \right\}.
\end{align*}
By binomial expansion
\[
|1-p^{i t_1}|^{2k}  =  \sum_{j+j'=2k}\binom{2k}{k} (-1)^{j'} p^{i(j-j')t_1}.
\]
The terms with $j=j'$ contribute a factor
\[
\exp\Bigl(\sum_{p\le y}\binom{2k}{k}\frac{1}{p}\Bigr)\asymp (\log y)^{\binom{2k}{k}}.
\]
By the final part of Lemma \ref{zeta}, since $t_1\in[1,2]$ the terms with $j\neq j'$ contribute a factor
\[
\exp\Bigl(O\Bigl(\sup_{|j|\le 2k}\Bigl|\sum_{p\le y}\frac{p^{i j t_1}}{p}\Bigr|\Bigr)\Bigr)\asymp 1.
\]
Thus we obtain a lower bound of $\gg (\log{y})^{\binom{2k}{k}-1}$ for our sum over $d_1,\dots d_{2k}$. Since the region of integration has volume $\gg (\log{y})^{-(k-1)}$, this gives
\[
W(\bs J) \gg \frac{(\log y)^{\binom{2k}{k} - k}}{(\log R)^{k+2kA}} 
\]
in this case. So we find that the total contribution to the right hand side of \eqref{lb-e2} from such $\bs J$ is $\ge c_2(\log y)^{\binom{2k}{k}-k+1}/(\log R)^{(2A+1)k}$ for some $c_2>0$, which is greater than the claimed main term in \eqref{kthmoments-lb-unified} if $\epsilon'$ is small enough in terms of $\epsilon$ and $k$. 

\medskip

\subsection{Upper bound for the error term}\label{lb-odd}
It remains to show that the contribution of the $\bs J$'s for which at least one of the $\#J_a$'s is odd, is smaller than what we have above. Before we get started, we note that
\als{
\lambda_p(\bs t) 
	= -1 + \prod_{j=1}^{2k} (1-p^{-it_j}) 
	= - 1+ \prod_{j=1}^{2k} (p^{it_j/2}-p^{-it_j/2}) 
	= - 1 +(-4)^k \prod_{j=1}^{2k} \sin\left(\frac{t_j\log p}{2}\right) 
}
whenever $t_1+\cdots+t_{2k}=0$, which is the case here. In particular, $-4^k\le\lambda_p(\bs t)+1\le 4^k$, whence
\[
\abs{\sum_{d_1,\dots,d_{2k}\in \CN} \frac{\prod_{j=1}^{2k}\mu(d_j)d_j^{-it_j}}{[d_1,\dots,d_{2k}]}}
	=\abs{ \prod_{q_{A+1}<p\le y} \left(1+ \frac{\lambda_p(\bs t)}{p}\right) }
	\ll  \prod_{4^k<p\le y} \left(1+ \frac{\lambda_p(\bs t)}{p}\right) .
\]

Next, we split the region of integration in \eqref{W(J)} into various subsets. First, we note that, by a dyadic decomposition argument and \eqref{w-bound-2}, we have that
\[
W(\bs J) \ll \frac{(\log T)^{2k}}{(\log R)^{L+2kA}} \max_{1\le T_1,\dots,T_{2k}\le T}
		\frac{ \Lambda(\bs T)}{T_1\cdots T_{2k-L}}
			+ \frac{(\log R)^{O(1)}}{T} ,
\]
where
\[
\Lambda(\bs T) := \idotsint\limits_{\substack{T_j\le |t_j|+1\le 2T_j \\ 1\le j\le 2k-L \\ 
			 t_{J_\ell} = 0\ (\ell\in\CL)}} 
		\prod_{4^k<p\le y}\left(1+\frac{\lambda_p(\bs t)}{p}\right) 
					dt_1\cdots dt_{2k-L} .
\]
We take
\[
T= \exp\{(\log\log R)^2\} ,
\]
and fix a choice of $T_1,\dots,T_{2k-L}$ as above. We will further break the region of integration according to which sums $t_J=\sum_{j\in J}t_j$ are small. Indeed, by Lemma \ref{zeta} the product $\prod_{q_{A+1}<p\le y} (1+\lambda_p(\bs t)/p)$ can become large only if there are such configurations. Note that if $t_{J_1}$ and $t_{J_2}$ are both small, so is any linear combination of $t_{J_1}$ and $t_{J_2}$. Thus, we are naturally led to the following definition: given free variables $x_1,\dots,x_{2k}$ and $J\subset[2k]$, we define the linear form
\[
L_J(x_1,\dots,x_{2k}) := \sum_{j\in J} x_j,
\]
(thinking of the linear form as acting on the space $\Q^{2k}$) and, given $\CI\subset \CS(2k)$, we set
\[
V_B(\CI) := \left\{ \sum_{I\in\CI} \frac{a_I}{q_I} \cdot L_I \, : \,a_I,q_I\in\Z\cap[-B,B]  \right\}
\]
(thought of as a subspace in the dual of $\Q^{2k}$.) For the simplicity of notation, we also set
\[
V(\CI) := V_\infty(\CI) =\Span_{\Q}(\{ L_I: I\in\CI \}) .
\]
Since there are only finitely many linear forms $L_I,\, I\subset[2k]$, there is some finite $B_0=B_0(k)$ such that, for any $\CI$, if $L_J\in V(\CI)$, then $L_J\in V_B(\CI)$ for all $B\ge B_0$. We assume that $B\ge B_0$ from now on (we will eventually choose $B$ to be large enough in terms of $k$).

We set $m=\fl{\log\log y}$ and, to each $\bs t$, we associate the sets
\[
\CI_j  = \CI_j(\bs t) :=\{ I\subset[2k] : e^m |t_I|+1 \le e^{j+1} \} \quad(0\le j\le m).
\]
Note that if $L_I\notin V_B(\CI_m)$, then $|t_I|>1$. Since $t_I\ll T$ for all $I$, Lemma \ref{zeta} implies that
\begin{align*}
\prod_{4^k<p\le y} \left(1+\frac{\lambda_p(\bs t)}{p}\right) 
&=\exp\Bigl(\sum_{I\in \CS^*(2k)}(-1)^{\#I}\sum_{4^k<p\le y}\Bigl(p^{-1-it_I}+O\Bigl(\frac{1}{p^2}\Bigr)\Bigr)\Bigr)\\
&=\exp\Bigl(\sum_{\substack{I\in \CS^*(2k)\\ |t_I|\le 1}}(-1)^{\#I}\min\Bigl\{\log\log{y},\log\Bigl(\frac{1}{|t_I|}\Bigr)\Bigr\}\Bigr)(\log\log{R})^{O(1)}\\
&\ll (\log\log R)^{O(1)} \prod_{\substack{I\in \CS^*(2k) \\ L_I \in V_B(\CI_m) }}
		\min\left\{ \log y, \frac{1}{|t_I|} \right\}^{(-1)^{\#I}} .
\end{align*}
If $L_I\in V_B(\CI_j)\setminus V_B(\CI_{j-1})$, where $\CI_{-1}=\{\emptyset\}$ so that $V_B(\CI_{-1})=\{0\}$, then $I\notin \CI_{j-1}$, which means that $e^m|t_I|+1>e^j$. On the other hand, since $L_I\in V_B(\CI_j)$, then we find that $e^m|t_I|\le B\cdot \#\CI_j\cdot e^{j-m}$. We thus conclude that
\[
\min\left\{ \log y, \frac{1}{|t_I|} \right\} \asymp e^{m-j} 
	\quad\text{for}\quad L_I\in V_B(\CI_j)\setminus V_B(\CI_{j-1}),\ 0\le j\le m.
\]
Therefore
\[
\prod_{4^k<p\le y} \left(1+\frac{\lambda_p(\bs t)}{p}\right) 
	\ll (\log\log R)^{4^k}  \prod_{j=0}^m  e^{(m-j)F_j},
\]
where
\[
F_j := \sum_{\substack{I\subset[2k] \\ L_I\in V_B(\CI_j)\setminus V_B(\CI_{j-1})}} (-1)^{\#I} .
\]
(Here we use the fact that there are only finitely many indices $j$ for which $\CI_j\neq \CI_{j-1}$.) 
Summing over the $\ll(\log\log R)^{O(1)}$ choices for $\CI_0,\CI_1,\dots,\CI_m$, we conclude that
\[
\Lambda(\bs T) \ll (\log\log R)^{O(1)}
	 \max_{\substack{ \CI_m\supset\cdots\supset\CI_0\supset\{J_1,\dots,J_L\} }} 
		\left(\nu(\bs T,\bs \CI) \cdot   \prod_{j=0}^m e^{(m-j)F_j } \right) ,
\]
where $\nu(\bs T,\bs \CI)$ denotes the $(2k-L)$-dimensional Lebesgue measure of $(t_1,\dots,t_{2k-L})\in\R$ such that $T_j\le|t_j|+1\le 2T_j$ for $j\le 2k-L$ and 
\[
\{ I\subset [2k]   :  e^m|t_I|+1 < e^{j+1} \} = \CI_j 
	\quad(0\le j\le m) ,
\]
where the variables $t_{2k-L+1},\dots,t_{2k}$ are defined via the equations $t_{J_\ell}=0$ for $\ell\in\{1,\dots,2k\}$. (In particular, $\CI_0\supset \{J_1,\dots,J_L\}$.)

Next, we use linear algebra to understand $\nu(\bs T,\bs \CI)$. If
\[
D_j = \dim V(\CI_j) ,
\]
then we may find sets $I_1,\dots,I_{D_m}$ such that, for each $j\le m$, $\{L_{I_1},\dots,L_{I_{D_j}}\}$ is a basis of $V(\CI_j)$. We may also assume that $\{J_\ell: 1\le \ell\le L\}$ is contained in $\{I_1,\dots,I_{D_1}\}$. 

Eliminating variables from linear combinations of the asymptotic formulas
\[
t_I = O(e^{j-m})   \quad(I \in  \{I_1,\dots,I_{D_j} \}) 
\]
(for example, as in Gaussian elimination), yields  $D_j$ of the variables $t_i$, say the variables $\{t_i: i\in \CD_j\}$ (where $\#\CD_j=D_j$), in terms of linear combinations of the other variables, up to an error of  $O(e^{j-m})$. 
We can also arrange the sets $\CD_0,\dots,\CD_L$ to satisfy 
$\CD_0\subset\cdots\subset \CD_L$. (Recall that $\CI_j\neq \CI_{j-1}$ for only finitely many $j$'s.) We may therefore prove that
\[
\nu(\bs T,\bs I) \ll (\log y)^L \left(\prod_{j=0}^m e^{(D_j-D_{j-1})(j-m)} \right)
	\Biggl(\prod_{\substack{1\le j\le 2k \\ j \notin \CD_0\cup\cdots \cup \CD_L}} T_j\Biggr),
\]
where the extra factor $(\log y)^L$ is included because we are not integrating over the variables $t_{2k-L+1},\dots,t_{2k}$, which are fixed via the conditions $t_{J_\ell}=0$, $1\le \ell\le L$. By the above discussion, we find that
\[
\Lambda(\bs T) \ll T_1\cdots T_{2k-L} (\log y)^L (\log\log R)^{O(1)}
	 \max_{\substack{ \CI_m\supset\cdots\supset\CI_0\supset\{J_1,\dots,J_L\} }} 
		  \prod_{j=0}^m e^{(m-j)( F_j - (D_j-D_{j-1}) ) }  .
\]
We note that
\als{
\sum_{j=0}^m (m-j)( F_j - (D_j-D_{j-1}) )
	&= \sum_{j=0}^{m-1} \sum_{i=1}^{m-j}( F_j - (D_j-D_{j-1}) )   \\
	&= \sum_{i=1}^m \sum_{j=0}^{m-i} ( F_j - (D_j-D_{j-1}))  \\
	&= \sum_{i=1}^m ( \SA(V(\CI_{m-i}))  -   \dim( V(\CI_{m-i})) )
}
in the notation of Section \ref{combinatorial}, provided that $B$ is large enough. Since $\exp$ is a convex function, we have that
\[
e^{\sum_{j=0}^m (m-j)( F_j -(D_j-D_{j-1}) ) }
\le \frac{1}{m} \sum_{j=0}^{m-1} e^{ m(\SA(V(\CI_j)) - \dim(V(\CI_j))} 
	\asymp \frac{1}{m} \sum_{j=0}^{m-1} (\log y)^{\SA(V(\CI_j) ) - \dim(V(\CI_j))}
\]
We then conclude that
\[
W(\bs J) 
	\ll \frac{(\log\log R)^{O(1)}(\log y)^L}{(\log R)^{L+2kA} }
		\max_{\CI \supset\{J_1,\dots,J_L\} } (\log y)^{\SA(V(\CI)) - \dim(V(\CI))}.
\]
The above discussion reduces \eqref{kthmoments-lb-unified} to proving that
\eq{lb-goal}{
\SA(V(\CI))- \dim(V(\CI)) \le \binom{2k}{k} -2k -1 
}
whenever $\CI$ contains the sets $J_1,\dots,J_L$ and at least one of the $J_i$'s has an odd number of elements. This follows directly by Proposition \ref{CombProp}, thus completing the proof of \eqref{kthmoments-lb} and, hence, of Theorem \ref{mainthm}.



\section{On the anatomy of integers contributing to $\CM_{f,2k}(R)$}\label{anatomy}

This section is devoted to establishing Theorems \ref{thm-factors} and \ref{thm-smooth}. Throughout this section, given $n\in\Z_{\ge 1}$ and $R\ge1$, we adopt the notations
\[
\Omega(n;R):= \sum_{\substack{p^\alpha\|n,\, p\le R}} \alpha
\quad\text{and}\quad
\Omega(n;r,R) :=\sum_{\substack{p^\alpha\|n,\, r<p\le R}} \alpha  .
\]

A key observation, that we will use several times, is that that if $(a,b)=1$, then
\eq{submult-1}{
M_{f_A}(ab;R) 
	&= \sum_{d|a}\sum_{d'|b} \mu(d) \mu(d') f_A\left(\frac{\log d}{\log R}+ \frac{\log d'}{\log R} \right) 	 \\
	&= \sum_{\substack{ d'|b \\ d\le R}} \mu(d') \left( \frac{\log(R/d')}{\log R}\right)^A
		\sum_{d|a} \mu(d) f_A\left(\frac{\log d}{\log(R/d')}\right) 
}
by the definition of $f_A$.

\subsection{An estimate for the logarithmic means}

We start by proving a preliminary result, where the integer $n$ is weighted by $1/n$. The transition to the uniform weights will be accomplished in the subsequent section.

\begin{thm} \label{thm-factors-log}
Let $R\ge2$, $k\in\Z_{\ge 1}$ and $A\in\Z_{\ge0}$.
\begin{enumerate} 
\item If $A>\frac{1}{2k}\binom{2k}{k}-1$ and $\eta\in[\log2/\log R,1]$, then
\als{
\prod_{p\le R}\left(1-\frac{1}{p}\right)\sum_{P^+(n)\le R}
	\frac{\Omega(n;R^\eta) M_{f_A}(n;R)^{2k}}{n} 
	&\ll_{k,A} \frac{\eta^{2k}}{\log R} + (\log R)^{\binom{2k}{k}-2k(A+1)} \log\log R .
}
\item If $A\le \frac{1}{2k}\binom{2k}{k}-1$ and $1-1/\binom{2k}{k}\le v\le 2-\epsilon$, then
\[
\prod_{p\le R}\left(1-\frac{1}{p}\right)
	\sum_{P^+(n)\le R} \frac{v^{\Omega(n)} M_{f_A}(n;R)^{2k}}{n} 
	\ll_{\epsilon,k,A} (\log R)^{v\binom{2k}{k}-2k(A+1)} (\log\log R)^{O_k(1)}.
\]
\end{enumerate}
\end{thm}

\begin{proof}To ease notation, for this proof we let all implied constants depend on $k,A$ and $\epsilon$ without explicitly stating this. 

\medskip

(a) First of all, we claim we may restrict our attention to small enough $\eta$. To see this, it suffices to show that
\eq{largeprimesbound}{
\sum_{P^+(n)\le R}
	\frac{\Omega(n;R^\delta,R) M_{f_A}(n;R)^{2k}}{n} 
	\ll_\delta 1,
}
for any fixed $\delta>0$.  Now, H\"older's inequality applied to \eqref{submult-1}  yields
\eq{submult-2}{
M_{f_A}(ab;R)^{2k} \le \tau(b)^{2k-1} \sum_{\substack{d'|b \\ d'\le R}} \mu^2(d') \left( \frac{\log(R/d')}{\log R}\right)^{2kA} M_{f_A}(a; R/d')^{2k} .
}
Therefore, writing $n=ab$, where $b=\prod_{p^v\|n,\,R^\delta<p\le R}p^v$, we find that
\als{
&\sum_{P^+(n)\le R}
	\frac{\Omega(n;R^\delta,R) M_{f_A}(n;R)^{2k}}{n} \\
	&\quad\le \sum_{\substack{p|b\,\Rightarrow R^\delta<p\le R}} 
		\frac{\tau(b)^{2k-1}  \Omega(b)}{b} \sum_{\substack{d'|b \\ d'\le R}} \mu^2(d') 
		\left( \frac{\log(R/d')}{\log R}\right)^{2kA} 
		\sum_{P^+(a)\le R} \frac{M_{f_A}(a;R/d')^{2k}}{a}  \\
	&\quad\ll  \sum_{\substack{p|b\,\Rightarrow R^\delta<p\le R}} 
		\frac{\tau(b)^{2k-1} \Omega(b)}{b} \sum_{\substack{d'|b \\ d'\le R}} \mu^2(d') 
		\left( \frac{\log(R/d')}{\log R}\right)^{2kA}  \frac{\log R}{\log(R/d')} 
		\ll_\delta 1
}
by Theorem \ref{mainthm}, since $A>\frac{1}{2k}\binom{2k}{k}-1$ and so $\CE_{k,A}=-1$ and $A\ge 1$. This proves \eqref{largeprimesbound}, so for the rest of the proof, we may assume that $\eta$ is sufficiently small. 

Expanding $\Omega(n,R^\eta)$, we have that
\als{
S&:=\prod_{p\le R}\left(1-\frac{1}{p}\right)\sum_{P^+(n)\le R}
	\frac{\Omega(n;R^\eta) M_{f_A}(n;R)^{2k}}{n} \\
 	&= \prod_{p\le R}\left(1-\frac{1}{p}\right)\sum_{\substack{q\le R^\eta \\ j\ge1}} \frac{j}{q^j} 
	\sum_{\substack{P^+(m)\le R \\ q\nmid m}} \frac{M_{f_A}(mq^j;R)^{2k}}{m} \\
	&\le \prod_{p\le R}\left(1-\frac{1}{p}\right)\sum_{q\le R^\eta} \frac{q}{(q-1)^2} 
	\sum_{P^+(m)\le R} 
		\frac{\left( \sum_{d|m} \mu(d) \left\{f_A\left( \frac{\log d}{\log R}\right) 
		- f_A\left( \frac{\log(qd)}{\log R}\right)
		\right\}\right)^{2k}}{m} .
}
where we used \eqref{submult-1} with $a=q^j$ and $b=m$. Expanding the $2k^{th}$ power, we find that
\[
S \le \sum_{q\le R^\eta} \frac{q}{(q-1)^2}
	\sum_{\substack{P^+(d_j)\le R \\ 1\le j\le 2k}} 
		\frac{\prod_{j=1}^{2k}\mu(d_j) \left\{f_A\left( \frac{\log d_j}{\log R}\right) 
			- f_A\left( \frac{\log(qd_j)}{\log R}\right)\right\}}
			{[d_1,\dots,d_{2k}]} ,
\]
Here $A\ge1$, so that $\hat{(f_A)}(s)=A!R^s/(s^{A+1}(\log R)^A)$ decays fast, and Perron inversion implies that
\eq{S-e1}{
S\le  \sum_{q\le R^\eta} \frac{q}{(q-1)^2}
			\idotsint\limits_{\substack{\Re(s_j)=1/\log R \\ 1\le j\le 2k}}
			\sum_{d_1,\dots,d_{2k}\ge1}
		\frac{\prod_{j=1}^{2k}\mu(d_j)d_j^{-s_j}}{[d_1,\dots,d_{2k}]} 
			\prod_{j=1}^{2k} \frac{A!R^{s_j}(1-q^{-s_j})}{(\log R)^A s_j^{A+1}} \dee s_1\cdots \dee s_{2k} .
}
The above integral is amenable to the methods of Section \ref{contour}. Precisely, we note that the integrand is of the form
\[
A!^{2k}  \left(\frac{\log q}{\log R}\right)^{2k} 
	\frac{	P(\bs s)R^{s_{[2k]}} }{(\log R)^{2k(A-1)}}
 \prod_{I\in\CS^*(2k)} \zeta^{e_I}(1+s_I)
		 \prod_{j=1}^{2k} \frac{(1-q^{-s_j})/(s_j\log q)}{(s_j\zeta(1+s_j))^A},
\]
where $P(\bs s)$ is as in Section \ref{contour}, $e_I=+1$ for $I\in\CS^+(2k)$, $e_I=-1$ for $I\in\CS^-(2k)$ with $\#I>1$, and $e_I=A-1\ge0$ for $\#I=1$. If $q\le R^\delta$ with $\delta$ small enough, then the argument leading to \eqref{kthmoments} yields that
\eq{kthmoments-p}{
&\idotsint\limits_{\substack{\Re(s_j)=1/\log R \\ 1\le j\le 2k}}
			\sum_{d_1,\dots,d_{2k}\ge1}
		\frac{\prod_{j=1}^{2k}\mu(d_j)d_j^{-s_j}}{[d_1,\dots,d_{2k}]} 
			\prod_{j=1}^{2k} \frac{A!R^{s_j}(1-q^{-s_j})}{(\log R)^A s_j^{A+1}} \dee s_1\cdots \dee s_{2k}  \\
	&\qquad\ll \frac{(\log q)^{2k}}{(\log R)^{2k+1}} + (\log R)^{\binom{2k}{k}-2k(A+1)} ,
}
with the first term coming from Case 1a and the second one from Case 2. Inserting the above inequality into \eqref{S-e1} completes the proof of part (a).

\medskip

(b) We will use a variation of the argument of Section \ref{lb}. The fact that Proposition \ref{CombProp} requires $s_{[2k]}=0$ complicates the proof; otherwise, a direct application of the method of Section \ref{lb-odd} would be possible. 

Call $S'$ the sum in question. As usual, we may replace replace $f_0$ by a sufficiently smooth function $h$. So write $g=h$ if $A=0$, and $g=f_A$ otherwise. Note that, since $M_g(n;R)$ depends only on the square-free kernel of $n$, we have that
\[
S' = \prod_{p\le R}\left(1-\frac{1}{p}\right)
	\sum_{P^+(n)\le R} \frac{\mu^2(n)v^{\omega(n)}}{\phi_v(n)} M_{f_A}(n;R)^{2k}
\]
with $\phi_v(n)=\prod_{p|n}(p-v)$. Set $y=R^{c/\log\log R}$ for a small enough but fixed $c$. Given an integer $n$, we decompose it as $n=ab$ with $P^+(a)\le y<P^-(b)$. If $\omega(a)\le A+1$, then $M_{f_A}(n;R)\ll 4^{k\omega(b)}$, and we immediately find that such $n$'s contribute at most $\ll(\log\log R)^{O(1)}/\log R$. Otherwise, we sum over all possible choices $a=qa'$ with $\omega(q)=A+1$ to deduce that 
\eq{S'(q)}{
S'\le v^{A+1}\sum_{\substack{P^+(q)\le y \\ \omega(q)=A+1}} \frac{\mu^2(q)}{\phi_v(q)} S'(q)
	+ O\left(\frac{(\log\log R)^{O(1)}}{\log R}\right) ,
}
where
\[
S'(q) := \prod_{p\le R}\left(1-\frac{1}{p}\right)
	 \sum_{\substack{P^+(a')\le y<P^-(b)\le R \\ (a',q)=1}} \frac{\mu^2(a'b)v^{\omega(a'b)}}{\phi_v(a'b)} M_{f_A}(qa'b;R)^{2k}.
\]
Next, we apply \eqref{submult-2} with $a=a'q$ to find that
\[
S'(q) \le  \prod_{p\le R}\left(1-\frac{1}{p}\right)
	 \sum_{\substack{P^+(a')\le y \\ y<P^-(b)\le R \\ (a',q)=1}} \frac{\mu^2(a'b)v^{\omega(a'b)}\tau(b)^{2k-1}}{\phi_v(a'b)} 
	 \sum_{\substack{d|b \\ d\le R}} \left( \frac{\log(R/d)}{\log R}\right)^{2kA} M_{f_A}(a'q; R/d)^{2k} .
\]
We write $b=cd$ and note that the sum over $c$ is $\le(\log\log R)^{O(1)}$ by our choice of $y$. Moreover, $\phi_v(n)\gg n/(\log\log n)^v$. Thus,
\[
S'(q) \ll (\log\log R)^{O(1)}  
	\sum_{\substack{d\le R\\ P^-(d)>y}} \frac{4^{k\Omega(d;y,R)}}{d} 
		S''(d,q)
		\le (\log\log R)^{O(1)}  
	\sum_{\substack{d\le R \\ P^-(d)>y}} \frac{S''(d,q)}{d^{1-2k/\log y}} ,
\]
where
\als{
S''(d,q) 
	&:= \prod_{p\le R}\left(1-\frac{1}{p}\right)
	 \sum_{\substack{P^+(a')\le y \\ (a',q)=1}} \frac{\mu^2(a')v^{\omega(a')}}{\phi_v(a')} 
			\left( \frac{\log(R/d)}{\log R}\right)^{2kA} M_{f_A}(a'q; R/d)^{2k} \\
	&\le \prod_{p\le R}\left(1-\frac{1}{p}\right)
	 \sum_{\substack{P^+(a)\le y \\ (a,q)=1}} \frac{v^{\Omega(a)}}{a}
			\left( \frac{\log(R/d)}{\log R}\right)^{2kA} M_{f_A}(aq; R/d)^{2k}.
}
Before we proceed to the estimation of $S''(d,q)$, we note that 
\eq{S'(q)-1}{
S'(q)\ll (\log\log R)^{O(1)}  
	\sum_{d\le R} \frac{(\lambda^+*1)(d)}{d^{1-2k/\log y}} 
		S''(d,q),
}
where $(\lambda^+(m))_{m\le R^\delta}$ is an upper bound sieve with $\delta$ small enough, constructed using the fundamental lemma of sieve methods, taking  $\kappa=1$ in \cite[Lemma 6.3, p. 159]{IK}. Then the Dirichlet series $\sum_d (1*\lambda^+)(d)/d^s$ has a simple pole at $s=1$ of residue $\sum_m\lambda^+(m)/m \asymp (\log\log R)^{O(1)}/\log R$. 

Next, if $q=p_1^{r_1}\cdots p_{A+1}^{r_{A+1}}$ is the prime factorisation of $q$, then \eqref{submult-1} implies that
\[
 \left( \frac{\log(R/d)}{\log R}\right)^{A} M_{f_A}(aq; R/d)
 	= \sum_{d'|aq} \mu(d') f_A\left( \frac{\log(dd')}{\log R}\right)
	= \sum_{d''|a} \mu(d'') w_q\left(\frac{\log(dd'')}{\log R}\right)
\]
with
\[
 w_q(x):= \sum_{J\subset[A+1]} (-1)^{\#J} f_A\left(x+\frac{\sum_{j\in J}\log p_j}{\log R}\right) .
\]
As usual, if $A=0$, we may replace replace $f_0$ by a sufficiently smooth function $h$ at the cost of a small error. Letting $g=h$ when $A=0$ and $g=f_A$ otherwise, and letting $W_q$ have the same definition as $w_q$ with $g$ in place of $f_A$, we find that
\als{
S''(d,q) 
	&\le \prod_{p\le R}\left(1-\frac{1}{p}\right)
	 \sum_{\substack{P^+(a)\le y \\ (a,q)=1}} \frac{v^{\Omega(a)}}{a} 
	 	\left( \sum_{f|a} \mu(f) W_q\left(\frac{\log(df)}{\log R}\right)\right) ^{2k} 
		+O\left(\frac{1}{\log R}\right) \\
	&= \frac{\prod_{p\le R}(1-1/p)}{\prod_{p\le y,\, p\nmid q}(1-v/p)} 
		\sum_{\substack{(f_j,q)=1 \\ 1\le j\le 2k}} 
			 \frac{v^{\Omega([f_1,\dots,f_{2k}])}}{[f_1,\dots,f_{2k}]} 
			 	\prod_{j=1}^{2k} \mu(f_j) W_q\left(\frac{\log(df_j)}{\log R}\right)
			+O\left(\frac{1}{\log R}\right) .
}
We apply Mellin inversion $2k$ times to find that
\als{
S''(d,q) \le  \frac{\prod_{p\le R}(1-1/p)}{(2\pi i)^{2k} \prod_{p\le y,\, p\nmid q}(1-v/p)} 
	\idotsint\limits_{\substack{\Re(s_j)=4k/\log y \\ |\Im(s_j)|\le T \\ 1\le j\le 2k}}
		\sum_{\substack{P^+(f_j)\le y,\, (f_j,q)=1 \\ 1\le j\le 2k}}  \frac{v^{\Omega([f_1,\dots,f_{2k}])}
		\prod_{j=1}^{2k}\mu(f_j)f_j^{-s_j}}{[f_1,\dots,f_{2k}]} \\
	 \times	\prod_{j=1}^{2k} \frac{\hat{g}_R(s_j)}{d^{s_j}} \prod_{a=1}^{A+1} (1-p_a^{-s_j}) \dee s_j 
					+ O\left(\frac{1}{\log R}\right)  .
}
Together with \eqref{S'(q)-1}, this implies that
\als{
S'(q) \le  \frac{(\log\log R)^{O(1)}\prod_{p\le R}(1-1/p)}{(2\pi i)^{2k} \prod_{p\le y,\, p\nmid q}(1-v/p)} 
	\idotsint\limits_{\substack{\Re(s_j)=4k/\log y \\ |\Im(s_j)|\le T \\ 1\le j\le 2k}}
		\sum_{\substack{P^+(f_j)\le y,\, (f_j,q)=1 \\ 1\le j\le 2k}}  \frac{v^{\Omega([f_1,\dots,f_{2k}])}
		\prod_{j=1}^{2k}\mu(f_j)f_j^{-s_j}}{[f_1,\dots,f_{2k}]} \\
	 \times	P(1+s_{[2k]}-2k/\log y) \prod_{j=1}^{2k} \hat{g}_R(s_j)  \prod_{a=1}^{A+1} (1-p_a^{-s_j}) \dee s_j 
					+ O\left(\frac{1}{\log R}\right)  ,
}
where
\[
P(s) := \sum_{d=1}^\infty \frac{(\lambda^+*1)(d)}{d^s}
	=\zeta(s) \sum_{m\le R^\delta} \frac{\lambda^+(m)}{m^s} .
\]
We fix $s_1,\dots,s_{2k-1}$ and move $s_{2k}$ to the line $\Re(s_{2k})=-8k/\log y$ to pick up the pole at $s_{[2k]}=2k/\log y$. If $c$ is small enough in the definition of $y$, and the same is true for $\delta$, the complementary contours contribute $\ll 1/\log R$. There are no other poles, since the factors $1-p_a^{-s_j}$ are annihilating the poles of $\hat{g}_R(s_j)$. We conclude that
\als{
S'(q) \ll \frac{(\log\log R)^{O(1)}}{(\log R)^{2kA+2-v}} 
	\idotsint\limits_{\substack{\Re(s_j)=4k/\log y \\ |\Im(s_j)|\le T \\ 1\le j\le 2k \\ s_{[2k]}=2k/\log y}}
	\Biggl| \sum_{\substack{P^+(f_j)\le y \\ 1\le j\le 2k}}  \frac{v^{\Omega([f_1,\dots,f_{2k}])}
		\prod_{j=1}^{2k}\mu(f_j)f_j^{-s_j}}{[f_1,\dots,f_{2k}]}\Biggr| \\
		\left(	\prod_{j=1}^{2k} \frac{1}{|s_j|^{A+1}}  \right) |\dee s_1\cdots\dee s_{2k-1}| 
					+ \frac{1}{\log R} .
}
Recall the notation $\lambda_p(\bs t)$ defined in \eqref{lambda-def}, and combine the above with \eqref{S'(q)} to find that
\als{
S' \ll \frac{(\log\log R)^{O(1)}}{(\log R)^{2kA+2-v}} 
	\idotsint\limits_{\substack{-T\le t_j\le T\ (1\le j<2k) \\ t_{[2k]}=0}}
		\left( 1+ \frac{v \cdot \lambda_p(\bs t) +(A+1)\sum_{j=1}^{2k}\cos(t_j\log p)}{p}\right)  \\
	\times \left(\prod_{j=1}^{2k} \frac{(\log(2+|t_j|))^{2k(A+1)}}{(1+|t_j|)^{A+1}}\right) \dee t_1\cdots \dee t_{2k-1} 
	+  \frac{(\log\log R)^{O(1)}}{\log R} ,
}
where we used the fact that $s\zeta(1+s)\gg 1/\log(2+|t|)$ for $\Re(s)>1$. Then, following the arguments of Section \ref{lb-odd} (with $L=1$), and recalling the notations $V(\CI)$ and $\SA(V(\CI))$, we find that
\[
S'   \ll  \max_{\CI\subset\CI(2k)} (\log R)^{v(1+\SA(V(\CI)))-\dim(V(\CI))+(A+1) U(\CI)-2kA-1} (\log\log R)^{O(1)} ,
\]
where 
\[
U(\CI):=\#\{1\le j\le 2k: L_{\{j\}}\in V(\CI)\} .
\]
If $U(\CI)\ge1$, then Proposition \ref{CombProp}(a) implies that $\SA(V(\CI))=-1$, and the exponent of $\log R$ is then
\[
(A+1)\cdot U(\CI) - \dim(V(\CI))-2kA-1\le A\cdot U(\CI) -2kA-1\le -1,
\]
since we clearly have that $2k\ge \dim(V(\CI))\ge U(\CI)$. Assume now that $U(\CI)=0$. If $\SA(V(\CI))=\binom{2k}{k}-1$ and $\dim(V(\CI))=2k-1$, then the exponent of $\log R$ is $v\binom{2k}{k}-2k(A+1)$. Finally, if this is not the case, then Proposition \ref{CombProp} (together with the argument in the end of Section \ref{lb-odd}) implies that 
\als{
v(1+\SA(V(\CI)))-\dim(V(\CI))-1
&\le \max\{0,v-1\} \binom{2k}{k} + \SA(V(\CI)) -  \dim(V(\CI)) \\
&\le \max\{0,v-1\} \binom{2k}{k} + \binom{2k}{k}-2k -1  \\
&\le v \binom{2k}{k} - 2k,
}
by our assumption that $v\ge 1-1/\binom{2k}{k}$. This completes the proof of the theorem.
 \end{proof}

\subsection{From logarithmic weights to uniform weights}

In this section, we show how to go from Theorem \ref{thm-factors-log} to the analogous result for the regular mean value and then prove Theorem \ref{thm-factors}.

\begin{thm} \label{thm-factors-prel}
Let $x\ge R\ge2$, $k\in\Z_{\ge 1}$, $A\in\Z_{\ge0}$ and $\epsilon\in(0,1/2)$.
\begin{enumerate} 
\item Assume that $A>\frac{1}{2k}\binom{2k}{k}-1$. Then uniformly for $\eta\in[\log2/\log R,1]$, we have
\[
\frac{1}{x}\sum_{n\le x}
			\Omega(n;R^\eta) M_{f_A}(n;R)^{2k}
		\ll_{k,A} \frac{\eta}{\log R} .
\]
\item If $A\le\frac{1}{2k}\binom{2k}{k}-1$ and $1-1/\binom{2k}{k}+\epsilon \cdot {\bf 1}_{k=1} \le v\le 2-\epsilon$, then
\[
\frac{1}{x}	\sum_{n\le x} v^{\Omega(n;R)} M_{f_A}(n;R)^{2k} 
	\ll_{k,A,\epsilon} (\log R)^{v\binom{2k}{k}-2k(A+1)} (\log\log R)^{O(1)} .
\]
\end{enumerate}
\end{thm}

\begin{proof}
We start by proving a preparatory estimate. Our goal is to show that there is some constant $C=C(k)$ such that
\eq{ind-hyp}{
\sum_{n\le x} M_{f_A}(n;R)^{2k} \le \frac{Cx}{\log R} 
	\quad(x\ge R>1) .
}
When $x\le 2^H$ for some fixed $H\in\Z_{\ge 1}$ that will be taken large enough in terms of $k$ and $A$, relation \eqref{ind-hyp} trivially holds by taking $C$ to be large enough in terms of $H$. Assume now that \eqref{ind-hyp} holds for all $x\le 2^h$, where $h\ge H$. We wish to prove that \eqref{ind-hyp} is also true when $x\in(2^h,2^{h+1}]$.

We follow a variation of the argument leading to Theorem III.3.5 in \cite[p.308]{Ten}: note that
\eq{diff-delay}{
\sum_{n\le x} M_{f_A}(n;R)^{2k} = \sum_{n\le x} M_{f_A}(n;R)^{2k} \frac{\log(x/n)}{\log x}
	+ \sum_{n\le x} M_{f_A}(n;R)^{2k} \frac{\log n}{\log x}.
}
For the first sum, we bound $\log(x/n)$ by $x/n$ to give
\eq{log-bound}{
\sum_{n\le x} M_{f_A}(n;R)^{2k} \frac{\log(x/n)}{\log x}
	&\le \sum_{n\le x} M_{f_A}(n;R)^{2k}\frac{x}{n\log x}  \\
	&\ll \frac{x}{\log R}  \sum_{P^+(n)\le R} \frac{M_{f_A}(n;R)^{2k}}{n}  \\
	&\ll \frac{x}{\log R},
}
by Theorem \ref{mainthm}. In the second sum, we write $\log n=\sum_{p^j\|n}j\log p$ to find that
\[
\sum_{n\le x} M_{f_A}(n;R)^{2k} \frac{\log n}{\log x}
	= \sum_{p^j\le x} \frac{j\log p}{\log x} \sum_{\substack{m\le x/p^j \\ p\nmid m}} M_{f_A}(mp^j;R)^{2k} .
\]
Since
\[
M_{f_A}(mp^j;R)
	= M_{f_A}(m;R)
		- {\bf 1}_{p<R} \left(\frac{\log(R/p)}{\log R}\right)^A  M_{f_A}(m;R/p) 
\]
by \eqref{submult-1}, Minkowski's inequality implies that
\[
\sum_{n\le x} M_{f_A}(n;R)^{2k} \frac{\log n}{\log x} 
	\le \left( S_1^{\frac{1}{2k}}+ S_2^{\frac{1}{2k}} \right)^{2k},
\]
where
\[
S_1 := \sum_{p^j\le x} \frac{j\log p}{\log x}
		\sum_{m\le x/p^j} M_{f_A}(m;R)^{2k}  .
\]
and
\[
S_2 := \sum_{\substack{p^j\le x \\ p<R}} \frac{j\log p}{\log x}
	\left(\frac{\log(R/p)}{\log R}\right)^{2kA} 
		\sum_{m\le x/p^j} M_{f_A}(m;R/p)^{2k}  .
\]
For $S_1$, we note that
\als{
S_1 = \sum_{m\le x } M_{f_A}(m;R)^{2k} 
	\sum_{p^j\le x/m} \frac{j\log p}{\log x}
	\ll \frac{x}{\log x}\sum_{m\le x} \frac{M_{f_A}(m;R)^{2k} }{m}
	\ll \frac{x}{\log R} ,
}
by \eqref{log-bound}. Finally, we need to bound $S_2$. First, we bound its subsum with $j\ge2$. We have that
\als{
&\sum_{\substack{p^j\le x \\ j\ge 2 \\  p<R}} \frac{j\log p}{\log x}
	\left(\frac{\log(R/p)}{\log R}\right)^{2kA} 
		\sum_{m\le x/p^j}  M_{f_A}(m;R/p)^{2k}  \\
	&\quad\le \sum_{\substack{p^j\le x \\ j\ge 2 \\  p<R }} \frac{j\log p}{\log x}
	\left(\frac{\log(R/p)}{\log R}\right)^{2kA} 
		\sum_{m\le x/p^j } M_{f_A}(m;R/p)^{2k}  \frac{x}{p^jm} \\
	&\quad\ll x  \sum_{\substack{p^j\le x \\ j\ge 2 \\  p<R}} \frac{j\log p}{p^j\log x}
	\left(\frac{\log(R/p)}{\log R}\right)^{2kA} \cdot \frac{\log x}{\log(R/p)} ,
}
by \eqref{log-bound} with $R$ replaced by $R/p$. Since $A\ge1$ here, we find that the above is 
\[
\ll \sum_{\substack{p^j\le x \\ j\ge 2}} \frac{j\log p}{p^j}
		 \cdot \frac{x}{\log R} 
		 \ll \frac{x}{\log R},
\]
where the implied constant is independent of $C$.

Finally, we need to bound the part of $S_2$ with $j=1$. We note that $x/p\le 2^h$ and that $R/p\le x/p$, so we may apply the induction hypothesis.  This gives a bound
\als{
\le \sum_{p<R}\frac{\log p}{\log x}
	\left(\frac{\log(R/p)}{\log R}\right)^{2kA} \cdot 
		\frac{Cx}{p\log(R/p)}  
	&\le \frac{Cx}{(\log{x})(\log R)^2} \sum_{p<R} \frac{(\log p)(\log(R/p))}{p} \\
	&\le \frac{2C}{3} \cdot \frac{x}{\log R} ,
}
provided that $H$ (and hence $x$) is big enough, where we used our assumptions that $A\ge1$ and $x\ge R$. Combining the above, and assuming that $C$ is big enough in terms of $k$ and $A$ completes the inductive step and hence the proof of \eqref{ind-hyp}.

\medskip

We now turn to proving part (a), that is bounding the sum
\[
S(x,R,Q):= \sum_{n\le x}  \Omega(n;Q) M_{f_A}(n;R)^{2k}.
\]
The proof is similar to the proof of \eqref{ind-hyp}, so we only give the main technical twists: we use induction on the dyadic interval in which $x$ lies to prove that there is some constant $C'=C'(k,\epsilon)$ such that
\eq{ind-hyp2}{
	S(x,R,Q) \le \frac{C'x\log Q}{(\log R)^{2}}  
	\quad(x\ge R\ge Q\ge 2) .
}
When $x\le 2^H$, for some fixed $H\in\Z_{\ge 1}$ that will be taken large enough in terms of $k$, $A$ and $\epsilon$, this trivially holds by taking $C'$ to be large enough in terms of $H$. Assume now that \eqref{ind-hyp} holds for all $x\le 2^h$, where $h\ge H$. We will prove that it also holds for $x\in(2^h,2^{h+1}]$. Note that the analogues of \eqref{diff-delay} and \eqref{log-bound} hold here as well, so let us focus on understanding the sum
\[
T:= \sum_{n\le x} \Omega(n;Q) M_{f_A}(n;R)^{2k} \frac{\log n}{\log x} .
\]
Fix a large integer $N$ and call $T_1$ the portion of this sum with $\Omega(n;Q)>2N$ and $T_2$ the remaining sum. 
Writing $\log n=\sum_{p^j\|n}j\log p$, we find that
\[
T_1 = \sum_{p^j\le x} \frac{j\log p}{\log x} \sum_{\substack{m\le  x/p^j,\, p\nmid m \\ \Omega(mp^j;Q)>2N}}
		\Omega(mp^j;Q)M_{f_A}(mp^j;R)^{2k}.
\]
If $T_1'$ is the part with $\Omega(mp^j;Q)\le 2j$ and $T_1''$ is the rest, then
\als{
T_1' &\le \sum_{\substack{p^j\le x \\ j> N}} 
	\frac{2j^2\log p}{\log x} \sum_{\substack{m\le  x/p^j \\  p\nmid m }} 
 			M_{f_A}(mp^j;R)^{2k} \frac{x}{p^jm}  \\
	&\ll x\sum_{\substack{p^j\le x,\, p\le e^{\sqrt{\log R}} \\ j>N}} \frac{j^2\log p}{p^j\log R} 
		\sum_{\substack{P^+(m)\le R \\ p\nmid m }} \frac{M_{f_A}(mp;R)^{2k}}{m} 
			+ \frac{x (\log R)^{O(1)}}{e^{\sqrt{\log R}}}  \\
	&\ll_N x\sum_{\substack{p^j\le x,\,p\le  e^{\sqrt{\log R}} \\ j>N}} \frac{j^2\log p}{p^j\log R} 
		\cdot \left\{ \left(\frac{\log p}{\log R}\right)^{2k} + (\log R)^{1+\binom{2k}{k}-2k(A+1)} \right\} \\
	&\ll \frac{x}{(\log R)^{2k+1}} + x(\log R)^{\binom{2k}{k}-2k(A+1)} ,
}
where the second to last bound follows from \eqref{kthmoments-p}. Finally, in the range of $T_1''$, we note that 
$\Omega(m;Q) \ge (\Omega(m;Q)+j)/2>N$, so that $\Omega(mp^j;Q)\le j(1+\Omega(m;Q))\le (1+1/N)\cdot j\cdot \Omega(m;Q)$. Therefore,
\eq{T_1}{
T_1	 &\le   \frac{N+1}{N}  \sum_{p^j\le x} \frac{j^2\log p}{\log x} 
		\sum_{\substack{m\le  x/p^j,\, p\nmid m \\ \Omega(mp^j;Q)>2N}}
		\Omega(m;Q)M_{f_A}(mp^j;R)^{2k}   \\
&\qquad	+O\left( \frac{x}{(\log R)^{2k+1}} +x (\log R)^{\binom{2k}{k}-2k(A+1)} \right) .
}

Next, we need to bound 
\[
T_2 =  \sum_{\substack{n\le x \\ \Omega(n;Q)\le 2N}} \Omega(n;Q) M_{f_A}(mp^j;R)^{2k} \frac{\log n}{\log x} .
\]
If $Q>R^{1/(2N^2)}$, then we simply note that
\[
T_2\le 2N \sum_{n\le x} M_{f_A}(n;R)^{2k}  \ll_N \frac{x}{\log R} 
	\ll_N \frac{x\log Q}{(\log R)^2} .
\]
by \eqref{ind-hyp}. Here the implied constant depends on $N$ but does not depend on $C'$.

On the other hand, if $Q\le R^{1/2N^2}$, then we have that $\sum_{p^j\|n,\,p\le Q} j\log p \le (\log x)/N$, so that 
\als{
T_2	&\le \frac{S(x,R,Q)}{N} 
	+ \sum_{\substack{n\le x \\ \Omega(n;Q)\le 2N}} 
		\Omega(n;Q) M_{f_A}(n;R)^{2k} \sum_{p^j\|n,\,p>Q} \frac{j\log p}{\log x} \\
	&\le \frac{S(x,R,Q)}{N} 
		+\sum_{\substack{p^j\le x \\ p>Q}} \frac{j\log p}{\log x} 
			\sum_{\substack{m\le x/p^j ,\, p\nmid m\\ \Omega(mp^j;Q)\le 2N}} 
		\Omega(m;Q) M_{f_A}(mp^j;R)^{2k} .
}
Combining the above inequality and \eqref{T_1}, we deduce that
\als{
T& \le \frac{S(x,R,Q)}{N} 
		+\frac{N+1}{N} \sum_{p^j\le x} \frac{j^2\log p}{\log x} 
			\sum_{m\le x/p^j,\, p\nmid m}
		\Omega(m;Q) M_{f_A}(mp^j;R)^{2k}
			 + O_N\left(\frac{x\log Q}{(\log R)^{2}}  \right) .
}
We thus conclude that
\als{
S(x,R,Q) &\le \frac{N+1}{N-1} \sum_{p^j\le x  } \frac{j^2\log p}{\log x} 
			\sum_{m\le x/p^j,\, p\nmid m}
		\Omega(m;Q) M_{f_A}(mp^j;R)^{2k}    
			 + O_N\left(\frac{x\log Q}{(\log R)^{2}}  \right) .
}
We can bound the sum on the right hand side in an entirely analogous way to the proof of \eqref{ind-hyp}. Choosing $N$ sufficiently large, and then $C'$ large enough in terms of $N$ completes the inductive step and thus the proof of \eqref{ind-hyp2}. This proves part (a) of the theorem.

\medskip

(b) The proof of part (b) is, for the most part, similar to the proof of \eqref{ind-hyp}. An important detail is that, after applying the induction hypothesis, we use the fact that
\eq{ind-critical}{
&\sum_{p<R}\frac{v\log p}{\log x}
	\left(\frac{\log(R/p)}{\log R}\right)^{2kA} \cdot 
		\frac{(\log(R/p))^{v\binom{2k}{k}-2k(A+1)}(\log\log(R/p))^D}{p} \\
&\qquad\sim \frac{v}{v\binom{2k}{k}-2k+1} \cdot \frac{\log R}{\log x} \cdot (\log R)^{\binom{2k}{k}-2k(A+1)}(\log\log R)^D 
}
for any fixed $D\ge0$, as $R\to\infty$. This is sufficient when $k\ge2$, because 
\[
\frac{v}{v\binom{2k}{k}-2k+1} 
	\le \frac{1-1/\binom{2k}{k}}{\binom{2k}{k}-2k} \le \frac{1}{2} 
\]
in this case. 

However, when $k=1$, the situation is more tricky. First of all, we note that the above argument allows to establish the theorem for $x\ge R^{2v/(2v-1)}$. (Note that if $p<R$ and $x\ge R^{2v/(2v-1)}$, then we also have that $x/p\ge (R/p)^{2v/(2v-1)}$, so that the inductive hypothesis can be applied.) Finally, it remains to treat the case when $x\le R^{2v/(2v-1)}$. We then observe that $\Omega(n;R^\delta,R)\ll_{\epsilon,\delta}1$, for any fixed $\delta$. It would thus suffice to prove that
\eq{ind-hyp3}{
\sum_{n\le x} v^{\Omega(n;R^\delta)} \delta^{\Omega(n;R^\delta,R)} M_{f_0}(n;R)^2 
	\le C''x (\log R)^{2v-2} (\log\log R)^D \quad( x\ge R \ge 2) ,
}
for some appropriate constants $C'',D>0$. Then, in place of \eqref{ind-critical}, we note that
\als{
&\sum_{p<R}\frac{(v\cdot{\bf1}_{p\le R^\delta} + \delta\cdot {\bf 1}_{R^\delta<p\le R})\log p}{\log x}
	\left(\frac{\log(R/p)}{\log R}\right)^{2kA} \cdot 
		\frac{(\log(R/p))^{v\binom{2k}{k}-2k(A+1)}(\log\log(R/p))^D}{p} \\
&\qquad< \frac{1}{2} \cdot (\log R)^{\binom{2k}{k}-2k(A+1)}(\log\log R)^D
}
as $R\to\infty$, as long as $\delta$ is small enough. This allows us to complete the inductive step and establish \eqref{ind-hyp3}, thus completing the proof of the theorem.
\end{proof}

Given the above result, proving Theorem \ref{thm-factors} is quite easy:

\begin{proof}[Proof of Theorem \ref{thm-factors}]
(a) This an immediate consequence of Theorem \ref{thm-factors-prel}(a).

\medskip

(b) We use Rankin's trick: Given a small $\alpha>0$ and $1\le v\le 3/2$, we have that
\[
\sum_{\substack{n\le x \\ \Omega(n;R)>\binom{2k}{k}(1+\alpha)\log\log R}}  M_{f_A}(n;R)^{2k} 
	\le 	\sum_{n\le x}  v^{\Omega(n;R)-\binom{2k}{k}(1+\alpha)\log\log R}   M_{f_A}(n;R)^{2k} .
\]
Theorem \ref{thm-factors-prel}(b) then implies that
\[
\sum_{\substack{n\le x \\ \Omega(n;R)>\binom{2k}{k}(1+\alpha)\log\log R}}M_f(n;R)^{2k} 
	\ll (\log R)^{\binom{2k}{k}-2k(A+1) + \binom{2k}{k}(v-1-(1+\alpha)\log v)} ,
\]
provided that $v$ is close enough to 1. We optimize this by choosing $v=1+\alpha$, so that $1-v+(1+\alpha)\log v = \int_1^{1+\alpha}(\log t)\dee t>0$. 

Similarly, we have that
\als{
\sum_{\substack{n\le x   \\ \Omega(n;R)<\binom{2k}{k}(1-\alpha)\log\log R}}   M_{f_A}(n;R)^{2k}		
	\le    \sum_{n\le x} v^{\Omega(n;R)-\binom{2k}{k}(1-\alpha)\log\log R} M_{f_A}(n;R)^{2k} ,
}
for any $1-1/\binom{2k}{k}+\epsilon \cdot {\bf 1}_{k=1}v\le 1$. Applying Theorem \ref{thm-factors-prel}(b) and choosing $v=1-\alpha$ for small enough $\alpha$ completes the proof of the theorem.
\end{proof}


\subsection{Estimates for general weight functions}

It is not so hard to go from estimates for the moments of $M_{f_A}(n;R)$ to the moments of $M_f(n;R)$ for a general weight function $f$. The following lemma provides the key link.

\begin{lem}\label{holder}
Let $f:\R\to\R$ be supported in $(-\infty,1]$. Assume further that $f\in C^A(\R)$ and that all functions $f,f',\dots,f^{(A)}$ are uniformly bounded for some integer $A\ge 1$. Then
\[
M_f(n;R)^{2k} \ll_{A,k,f} 
			\int_{\frac{\log2}{\log R}}^1 u^{2k(A-1)} M_{f_{A-1}}(n;R^u )^{2k} \dee u + \frac{1}{(\log R)^{2kA}}.
\]
\end{lem}

\begin{proof}
Since $f(x)=0$ for $x>0$ and $f\in C^A(\R)$, we must have that $f^{(j)}(1)=0$ for $j\le A$. Taylor's theorem with the integral form of the remainder term implies that
\[
f(x) = \int_1^x \frac{f^{(A)}(u)}{(A-1)!} (x-u)^{A-1} \dee u
	= \frac{(-1)^A}{(A-1)!} \int_{[x,1]} f^{(A)}(u) (u-x)^{A-1} \dee u ,
\]
for all $x$, since both sides vanish if $x>1$. Therefore,
\als{
M_f(n;R) 
	&=  \frac{(-1)^A}{(A-1)!}
			\sum_{d|n} \mu(d) \int_{\frac{\log d}{\log R}\le u\le 1} f^{(A)}(u) 
				\left(u- \frac{\log d}{\log R}\right)^{A-1} \dee u   \\
	&=   \frac{(-1)^A}{(A-1)!}
			\int_0^1 f^{(A)}(u) \sum_{\substack{d|n \\ d\le R^u}} \mu(d)\left(u- \frac{\log d}{\log R}\right)^{A-1} \dee u   \\
	&=    \frac{(-1)^A}{(A-1)!}
			\int_{\frac{\log2}{\log R}}^1 f^{(A)}(u) u^{A-1} M_{f_{A-1}}(n;R^u ) \dee u  
					+ O\left(\frac{1}{(\log R)^A}\right) ,
}
by noting that $d=1$ if $d\le R^u<2$. H\"older's inequality then completes the proof.
\end{proof}

We use the above lemma to show the analogue of Theorem \ref{thm-factors-prel} for general weight functions $f$.

\begin{thm}\label{thm-factors-prel-general}
Let $k\in\Z_{\ge 1}$, $x\ge R\ge 2$ and $f:\R\to\R$ be supported in $(-\infty,1]$. Assume further that $f\in C^A(\R)$ and that all functions $f,f',\dots,f^{(A)}$ are uniformly bounded for some integer $A\ge1$, and fix some $\epsilon\in(0,1)$.
\begin{enumerate} 
\item Let $A>\frac{1}{2k} \binom{2k}{k}$. Uniformly for $\eta\in[\log2/\log R,1]$, we have
\[
\frac{1}{x}\sum_{n\le x}
			\Omega(n;R^\eta) M_{f}(n;R)^{2k}
	\ll \frac{\eta}{\log R} .
\]
\item If $A\le \frac{1}{2k}\binom{2k}{k}$ and $1-1/\binom{2k}{k}+\epsilon\cdot{\bf 1}_{k=1} \le v\le 2-\epsilon$, then
\[
\frac{1}{x}	\sum_{n\le x} v^{\Omega(n;R)} M_{f}(n;R)^{2k} 
	\ll (\log R)^{v\binom{2k}{k}-2kA} (\log\log R)^{O(1)}.
\]
\end{enumerate}
All implied constants depend at most on $k$, $f$ and $\epsilon$.
\end{thm}

\begin{proof}
(a) Lemma \ref{holder} implies that
\eq{int-e0}{
\sum_{n\le x} \Omega(n;R^\eta) M_f(n;R)^{2k}
	&\ll \int_{\frac{\log 2}{\log R}}^1 u^{2k(A-1)} \sum_{n\le x} \Omega(n;R^\eta) M_{f_{A-1}}(n;R^u)^{2k} \dee u  \\
	&\qquad 	+ O\left(\frac{x\log\log R}{(\log R)^{2k A}}\right) ,
}
When $u\ge\eta$, Theorem \ref{thm-factors-prel}(a) implies that
\eq{int-e1}{
\int_\eta^1 u^{2k(A-1)} \sum_{n\le x} \Omega(n;R^\eta) M_{f_{A-1}}(n;R^u)^{2k} \dee u 
	\ll \frac{\eta x}{\log R} .
}
Finally, we consider the integral over $u\in[\log2/\log R,\eta]$. Observe that
\[
\sum_{n\le x} \Omega(n;R^u) M_{f_{A-1}}(n;R^u)^{2k} \ll \frac{x}{u\log R} ,
\]
by Theorem \ref{thm-factors-prel}(a). If $x\le R^{100}$, then we also have that $\Omega(n;R^u,R^\eta) \le 100/u$ for each $n\le x$, so that
\[
\sum_{n\le x} \Omega(n;R^u,R^\eta) M_{f_{A-1}}(n;R^u)^{2k} \ll \frac{x}{u^2\log R} 
\]
by Theorem \ref{thm-factors-prel}. We thus find that
\[
\int_{\frac{\log2}{\log R}}^\eta u^{2k(A-1)} \sum_{n\le x} \Omega(n;R^\eta) M_{f_{A-1}}(n;R^u)^{2k} \dee u 
	\ll \int_{\frac{\log2}{\log R}}^\eta \frac{u^{2k(A-1)-2}x}{\log R} \dee u \ll \frac{\eta x}{\log R} .
\]
Together with \eqref{int-e0} and \eqref{int-e1}, this proves part (a) in the case when $x\le R^{100}$. 

Finally, let us consider the case when $x>R^{100}$. Then
\[
\sum_{n\le x} \Omega(n;R^u,R^\eta)  M_{f_{A-1}}(n;R^u)^{2k} 
	\le \sum_{\substack{p^j\le x \\ R^u<p\le R^\eta}} j \sum_{m\le  x/p^j} M_{f_{A-1}}(m;R^u)^{2k} .
\]
When $p^{j-1}\ge \sqrt{x}$, then we bound the inner sum trivially by $\ll (x/p^j)(\log R)^{O(1)}$, so that the total contribution of such summands is $\ll\sqrt{x}(\log R)^{O(1)}$. Finally, when $p^{j-1}\le \sqrt{x}$, then we see that $x/p^j\ge \sqrt{x}/R\ge x^{0.499} \ge R^u$, so that the sum over $m$ is $\ll (x/p^j)/\log(R^u)$, by Theorem \ref{thm-factors-prel}(a). We thus conclude that
\[
\sum_{n\le x} \Omega(n;R^u,R^\eta)  M_{f_{A-1}}(n;R^u)^{2k} 
	\ll \frac{x\log(\eta/u)}{u\log R} +\frac{x}{(u\log R)^2}  \ll \frac{x}{u\log R} .
\]
Therefore, 
\[
\int_{\frac{\log2}{\log R}}^\eta u^{2k(A-1)} \sum_{n\le x} \Omega(n;R^\eta) M_{f_{A-1}}(n;R^u)^{2k} \dee u 
	\ll \frac{\eta x}{\log R} .
\]
in this case as well, thus completing the proof of part (a) of the theorem.

\medskip

(b) This proof of this part is similar. We start again with Lemma \ref{holder} to find that
Lemma \ref{holder} implies that
\eq{int-e0-b}{
\sum_{n\le x} v^{\Omega(n;R)} M_f(n;R)^{2k}
	&\ll \int_{\frac{\log2}{\log R}}^1 u^{2k(A-1)} \sum_{n\le x}v^{\Omega(n;R)} M_{f_{A-1}}(n;R^u)^{2k} \dee u \\
	&\qquad 	+ O\left(x(\log R)^{v-1-2kA} \right) ,
}
where we also used Theorem III.3.5 in \cite[p.308]{Ten}. Next, if $w=\max\{v,1\}$ and $\log2/\log R\le u\le1$, then note that
\als{
\sum_{n\le x} v^{\Omega(n;R)} M_{f_{A-1}}(n;R^u)^{2k}
	&\le \sum_{\substack{a\le x \\ P^+(a)\le R^u}} v^{\Omega(a;R^u)} M_{f_{A-1}}(a;R^u)^{2k}
		 \sum_{\substack{b\le x/a \\ P^-(b)>R^u}} w^{\Omega(b;R^u,R)}  .
}
When $x/a\ge R^u$, the sum over $b$ is $\ll u^{1-w} x/(a\log(R^u))$ by  Theorem III.3.5 in \cite[p.308]{Ten}. Hence,
\als{
\sum_{n\le x} v^{\Omega(n;R)} M_{f_{A-1}}(n;R^u)^{2k}
	&\ll \frac{u^{1-w} x}{\log(R^u)} \sum_{\substack{a\le x \\ P^+(a)\le R^u}} 
		\frac{v^{\Omega(a;R^u)} M_{f_{A-1}}(a;R^u)^{2k}}{a} \\
	&\qquad	+ \sum_{\substack{a\le x \\ P^+(b)\le R^u}} v^{\Omega(a;R^u)} M_{f_{A-1}}(a;R^u)^{2k} .
}
We bound the first sum by Theorem \ref{thm-factors-log}(b) and the second one by Theorem \ref{thm-factors-prel}(b) to find that
\als{
\sum_{n\le x} v^{\Omega(n;R)} M_{f_{A-1}}(n;R^u)^{2k}
	\ll x u^{1-w+v\binom{2k}{k}-2kA} (\log R)^{v\binom{2k}{k} -  2kA} .
}
Inserting the above bound into \eqref{int-e0-b} completes the proof of the theorem.
\end{proof}

We conclude this section with the proof of Theorem \ref{thm-smooth}. We need a preliminary lemma.

\begin{lem}\label{prime-sum}
If $m\in \Z_{\ge 1}$, $g\in C^1(\R^m)$ and $z\ge y\ge3$, then there is a positive constant $c>0$ such that
\als{
\sum_{y<p_1<\cdots<p_m\le z} \frac{g(\log p_1,\dots,\log p_m)}{(p_1-1)\cdots (p_m-1)} 
	= \frac{1}{m!} \int_{[\log y,\log z]^m} \frac{g(t_1,\dots,t_m)}{t_1 \cdots t_m} \dee t_1\cdots \dee t_m \\
		+ O\left(\frac{\|g\|_\infty + \| \nabla g\|_\infty}{e^{c\sqrt{\log y}}}
			\cdot \frac{(\sum_{y<p\le z}1/p)^{m-1}+(\int_y^z\dee t/t\log t)^{m-1}}{m!/m^2} \right) .
}
\end{lem}

\begin{proof} First of all, note that
\als{
\sum_{y<p_1<\cdots<p_m\le z} \frac{g(\log p_1,\dots,\log p_m)}{(p_1-1)\cdots (p_m-1)} 
	&= \frac{1}{m!}\sum_{\substack{y<p_1,\dots,p_m \le z \\ \text{distinct}}} 
		\frac{g(\log p_1,\dots,\log p_m)}{(p_1-1)\cdots (p_m-1)} \\
	&= \frac{1}{m!}\sum_{ y<p_1,\dots,p_m \le z } 
		\frac{g(\log p_1,\dots,\log p_m)}{p_1\cdots p_m}   \\
	&\quad	+ O\left( \frac{m^2\|g\|_\infty(\sum_{y<p\le z}1/p)^{m-1}}{m! y} \right) .
}
So, if we can show that
\als{
\sum_{y<p_1,\dots, p_m\le z} \frac{g(\log p_1,\dots,\log p_m)}{p_1\cdots p_m} 
		= \int_{[\log y,\log z]^m} \frac{g(t_1,\dots,t_m)}{t_1 \cdots t_m} \dee t_1\cdots \dee t_m   \\
		\quad+ O\left( \sum_{j=1}^m  \frac{\|g\|_\infty+\| \partial g/\partial x_j\|_\infty}{e^{c\sqrt{\log y}}} 
			\left(\int_y^z \frac{\dee t}{t\log t}\right)^{j-1} \left( \sum_{y<p\le z}\frac{1}{p}\right)^{m-j} \right)  ,
}
then the lemma will follow. But this estimate can be easily proved by induction on $m$ and the Prime Number Theorem, and the proof is completed.
\end{proof}

Let us now see how we can deduce Theorem \ref{thm-smooth} from the above results:

\begin{proof}[Proof of Theorem \ref{thm-smooth}]
The first estimate of part (a) follows by \ref{thm-factors-prel-general}(a) and part (b) follows by Theorem \ref{thm-factors-prel}(b).  It remains to prove the second estimate of part (a). Note that it suffices to prove that
\[
\prod_{p\le R}\left(1-\frac{1}{p}\right) 
	\sum_{P^+(n)\le R } \frac{M_f(n;R)^{2k}}{n} 
	= \frac{c_{k,f}}{\log R} + O\left(\frac{1}{(\log R)^{2-\epsilon} }\right),
\]
since $x\ge R^{2k}\log^2R$ here. Fix $\eta\in[\log 2/\log R,1]$ to be chosen later. Then Theorem \ref{thm-factors-prel-general}(a) implies that
\[
\sum_{\substack{P^+(n)\le R \\ P^-(n)\le R^\eta}} \frac{M_f(n;R)^{2k}}{n} 
	\le \sum_{P^+(n)\le R} \frac{\Omega(n;R^\eta)M_f(n;R)^{2k}}{n} 
		\ll \eta.
\]
Moreover, for each positive integer $m$, we have that
\als{
\sum_{\substack{p|n\,\Rightarrow R^\eta<p\le R \\  \omega(n) =m }} 
		\frac{M_f(n;R)^{2k} }{n}
		&= \sum_{R^\eta<p_1<\cdots<p_m \le R} \frac{M_f(p_1\cdots p_m;R)}{(p_1-1)\cdots(p_m-1)}
}
We note that
\[
M_f(p_1\cdots p_m;R)
	= g_m\left(\frac{\log p_1}{\log R},\dots,\frac{\log p_m}{\log R}\right),
\]
where
\[
g_m(t_1,\dots,t_m) 
	:= \sum_{J\subset\{1,\dots,m\}} (-1)^{\#J}
		f\left( \sum_{j\in J} t_j  \right) 
\]
which is a smooth function satisfying the estimates $\|g_m^{2k}\|_\infty\le 4^{km} \|f\|_\infty^{2k}$ and $\|\nabla g_m^{2k}\|_\infty\le 2k 4^{km} \|f'\|_\infty\|f\|_\infty^{2k-1}$. Therefore Lemma \ref{prime-sum} implies that
\als{
\sum_{\substack{p|n\,\Rightarrow R^\eta<p\le R \\  \omega(n) =m }} 
		\frac{M_f(n;R)^{2k} }{n}
		= \frac{1}{m!} \int_{[\eta,1]^m} \frac{g_m(t_1,\dots,t_m)^{2k}}{t_1\cdots t_m} dt_1\cdots dt_m  
			+O\left( \frac{(4^k\log(1/\eta)+O(1))^m}{e^{c\sqrt{\log(R^\eta)}} m!/m^2} \right) 
}
for all $m\in\Z_{\ge 1}$. We thus conclude that
\eq{smooth-e1}{
\sum_{P^+(n)\le R } \frac{M(n;R)^{2k}}{n} 
	= F(\eta) 
		+ O\left( \eta + \frac{\log^2(1/\eta)}{\eta^{4^k} e^{c\sqrt{\log(R^\eta)}} } \right) 
		\quad(0<\eta\le 1,\ R^\eta\ge 2) ,
}
where
\[
F(\eta):= 
	1+ \sum_{m=1}^\infty \frac{1}{m!}  \int_{[\eta,1]^m} \frac{g_m(t_1,\dots,t_m)^{2k}}{t_1\cdots t_m} \dee 		t_1\cdots \dee t_m  .
\]

Completing the proof is now an exercice in real analysis. We start by proving that $\lim_{\eta\to 0^+} F(\eta)$ exists. Indeed, applying \eqref{smooth-e1} twice, we deduce that
\[
F(\eta_1) - F(\eta_2) \ll \max_{j\in\{1,2\}} \left\{ \eta_j + \frac{\log^2(1/\eta_j)}{\eta_j^{4^k} e^{c\sqrt{\log(R^{\eta_j})}} } \right\}
\]
whenever $0<\eta_1\le \eta_2\le 1$ and $R^{\eta_1}\ge 2$. Letting $R\to\infty$, we find that
\eq{smooth-e2}{
F(\eta_1) - F(\eta_2) \ll  \eta_2 \quad(0<\eta_1\le \eta_2\le 1).
}
In particularly, Cauchy's convergence criterion implies that $\lim_{\eta\to0^+}F(\eta)$ exists. Call $F$ this limit, which clearly equals $F(0)$, and note that letting $\eta_1\to0^+$ in \eqref{smooth-e2} implies that 
\[
F(\eta) = F + O(\eta ) \quad(0<\eta\le 1) .
\]
Together with \eqref{smooth-e1}, this yields the estimate 
\[
\sum_{P^+(n)\le R } \frac{M(n;R)^{2k}}{n} 
	= F + O\left( \eta + \frac{\log^2(1/\eta)}{\eta^{4^k} e^{c\sqrt{\log(R^\eta)}} } \right) 
		\quad(0<\eta\le 1,\ R^\eta\ge 2) .
\]
Selecting $\eta$ such that
\[
R^\eta = e^{(\log\log R)^3}
\]
completes the proof of Theorem \ref{thm-smooth}.
\end{proof}

We conclude this section with the proof of Theorem \ref{thm-smooth-factors}.

\begin{proof}[Proof of Theorem \ref{thm-smooth-factors}]
The first estimate of part (a) can be proven following {\it mutatis mutandis} the proof of Theorem \ref{thm-factors}(a) above, using Theorem \ref{thm-factors-prel-general} in place of Theorem \ref{thm-factors-prel}. Similarly, part (b) follows from the proof of Theorem \ref{thm-factors}(b).
\end{proof}



\section{The analogy for non-exceptional Dirichlet characters}\label{nonexcepsec}

In the section, we study the sum 
\[
\CX_{2k}(R)   = \prod_{p\le R} \left(1-\frac{1}{p}\right) 
		  \sum_{P^+(n)\le R} \frac{1}{n} \left( \sum_{\substack{d|n \\  R/2<d\leq R}} \chi(d) \right)^{2k} 
\]
when $k\ge2$ and $L(1,\chi)$ is not very small and prove Theorem \ref{thm-chars}(b). 

We have that
\[
\CX_{2k}(R)  =  \sum_{R/2<d_1,\ldots ,d_{2k}\leq R} 
		 \frac{\chi(d_1)\chi(d_2)\cdots \chi(d_{2k}) }
		 		{ [d_1,\ldots ,d_{2k}]}    .
\]
We want to introduce new variables $D_I$, $I\in\CS^*(2k)$, as in Section \ref{heuristics}, but first we perform a technical manoeuvre to simplify the situation. We write $d_i=d_i'd_i''$, where $P^+(d_i')\le y<P^-(d_i'')$, 
where
\[
y:=(\log R)^{4^k+1} .
\]
The contribution to $\CX_{2k}(R)$ of $\bs d$'s for which $d_i''$ is not square-free for some $i$ is $\ll (\log R)^{4^k-1}/y$ by a crude upper bound, and so is the contribution of those $\bs d$'s with $\max_i d_i'>B$, where
\[
B:=e^{(\log\log R)^3},
\] 
by Rankin's trick. Then we let $D_I$, $I\in\CS^*(2k)$, be the product of those primes that divide $d_i''$ when $i\in I$, and are coprime to the other $d_i''$'s. The numbers $D_I$ are pairwise coprime and square-free, and $d_i''=\prod_{I\in\CS^*(2k),\ I\ni i}D_I$, so that
\[
 \frac{\chi(d_1'')\chi(d_2'')\cdots \chi(d_{2k}'') }
		 		{ [d_1'',\ldots ,d_{2k}'']}  =  
	 \frac{\prod_{I\in\CS^+(2k)}\chi_0(D_I) \prod_{I\in\CS^-(2k)}\chi(D_I)}{\prod_{I\in\CS^*(2k)}D_I}  .
\]
We may now drop the condition that the $D_I$'s are square-free and coprime, since the contribution to $\CI_{2k}(R)$ of the $D_I$'s not satisfying these conditions is $\ll (\log R)^{4^k-1}/y$. Finally, we may drop the condition that $(D_I,q)=1$ for $I\in\CS^+(2k)$, encoded in the notation $\chi_0(D_I)$, since the contribution of $D_I$'s with $P^-(D_I)>y $ and $(D_I,q)>1$ is $\ll(\log R)^{4^k-1}\sum_{p|q,\,p>y} 1/p \ll (\log R)^{4^k}/y$. The above discussion implies that
\als{
\CX_{2k}(R)  
	&=  \sum_{\substack{P^+(d_i')\le y,\, d_i'\le B \\ 1\le i\le 2k }}
	 \frac{\chi(d_1')\cdots \chi(d_{2k}') }{ [d_1',\ldots ,d_{2k}']}    
		\sum_{\substack{P^-(D_I)>y\ (I\in\CS^*(2k)) \\ 
			R/(2d_i')<\prod_{I\in\CS^*(2k),\, I\ni i}D_I\le R/d_i' \\ 1\le i\le 2k}}
		 	 \frac{\prod_{I\in\CS^-(2k)}\chi(D_I)}{\prod_{I\in\CS^*(2k)}D_I}  \\
		&\qquad+ O\left(\frac{1}{\log R} \right) .
}

Next, we note that
\[
\sum_{\substack{ n\le x \\ P^-(n)>y}} \chi(n)
	\ll x^{1-1/(30\log y)}\quad (x\ge \max\{q^4,y\}) ,
\]
by Lemma 2.4 in \cite{Kou13}. Therefore,
\eq{chi-sieve}{
\sum_{\substack{ n>q^4B \\ P^-(n)>y}} \frac{\chi(n)}{n} \ll \frac{1}{y} .
}
This implies that the contribution to $\CX_{2k}(R)$ with $D_I>q^4B$ for some $I\in\CS^-(2k)$ is $\ll (\log R)^{4^k-1}/y$. To conclude, we have shown
\eq{before-lattice}{
\CX_{2k}(R)  
	&=  \sum_{\substack{P^+(d_i') \le y,\, d_i'\le B \\ 1\le i\le 2k}} 
		\sum_{\substack{D_I\le q^4B \\ P^-(D_I)>y  \\ I\in\CS^-(2k)}}  
			 \frac{\chi(d_1')\cdots \chi(d_{2k}') }{ [d_1',\ldots ,d_{2k}']}    
			 \cdot \frac{\prod_{I\in\CS^-(2k)}\chi(D_I)}{\prod_{I\in\CS^-(2k)}D_I}  
				\cdot T( R_1,\dots,R_{2k}) \\
		&\qquad+ O\left(\frac{1}{\log R}\right) .		
}
where $R_i=R/(d_i'\prod_{I\in\CS^-(2k),\, I\ni i} D_I)$ and
\[
T(\bs R):= 
	\sum_{\substack{P^-(D_I)>y\ (I\in\CS^+(2k)) \\ 
			R_i/2<\prod_{I\in\CS^+(2k),\, I\ni i}D_I\le R_i \\ 1\le i\le 2k }}
		 	 \frac{1}{\prod_{I\in\CS^*(2k)}D_I}  .	
\]

Our task now becomes estimating $T(\bs R)$. Let $d=2^{2k-1}-2k$ and recall the definition of $V_d(\cdot)$ from the statement of Theorem \ref{thm-chars}(b). The proof of Theorem \ref{thm-perm} shows that $V_k(m)\asymp m^d$. We claim that
\eq{lattice}{
T(\bs R)  = V_k(\log R)\left(1 + O\left(\frac{\log B}{\log R}\right)\right) 
		 \prod_{p\le y} \left(1-\frac{1}{p}\right)^{2^{2k-1}-1}   
}
whenever $R/B\le R_i\le R$ for all $i$, as is the case here. Proving \eqref{lattice} can be accomplished easily using a lattice point count and the fundamental lemma of sieve methods. First of all, note that the part of $T(\bs R)$ where $D_I\le B$ for some $I\in \CS^+(2k)$ is trivially $\ll (\phi(P)/P)^{2^{2k-1}-1} (\log R)^{d-1}\log B$ by an upper bound sieve, where we have set $P:=\prod_{p\le y} p$ for simplicity. In the rest of the range, we set $\rho=1+1/\log R$ and divide the variables $D_I$ into boxes of the form $(\rho^{m_I},\rho^{m_I+1}]$, $m_I\ge0$. Replacing $D_I$ by $\rho^{m_I}$ in the conditions $R_i/2<\prod_{J\in\CS^+(2k),\, J\ni i}D_J\le R_i$ creates a total error of size $\ll (\phi(P)/P)^{2^{2k-1}-1} (\log R)^{d-1}$. In addition, if $\rho^{m_I}\ge B = e^{(\log\log R)^3}$, then we have that
\[
\sum_{\substack{\rho^{m_I}<D_I\le \rho^{m_I+1} \\ P^-(D_I)>y}} \frac{1}{D_I}
	= (\log\rho)\frac{\phi(P)}{P}  + O_C\left(\frac{1}{(\log R)^C}\right) 
\]
for any fixed $C$, by the fundamental lemma of sieve methods (see, for example, \cite[Theorem I.4.3]{Ten}). We thus conclude that
\als{
T(\bs R) & = \left(  (\log\rho)\phi(P) / P \right)^{2^{2k-1}-1}
				\sum_{\substack{m_I\ge \log B/\log\rho \ (I\in\CS^+(2k)) \\ 
					\frac{\log(R_i/2)}{\log\rho}< \sum_{I\in\CS^+(2k),\,I\ni i} m_I\le 
						\frac{\log R_i}{\log\rho}}} 1 \\
	&\qquad			+ O\left( (\phi(P)/P)^{2^{2k-1}-1}(\log R)^{d-1}\log B\right)  .
}
A straightforward lattice point counting argument implies that the sum on the right hand side equals
\[
\frac{W_k(\log R_1,\dots,\log R_{2k})}{(\log\rho)^{2^{2k-1}-1}} \left(1+O\left(\frac{\log B}{\log R}\right)\right) ,
\]
where $W_k(\bs m)$ is the volume of the polytope $\{(x_I)_{I\in\CS^+(2k)} : x_I\ge0\ \forall I,\ m_i-\log2\le\sum_{I\ni i}x_I\le m_i\ \forall i\}$. Since $m_i=\log R_i=\log R+O(\log B)$ here, we may show using the Mean Value Theorem that $W_k(\bs m) = V_k(m) +O((\log R)^{d-1}\log B)$. Relation \eqref{lattice} then follows.

We are now ready to complete the proof of Theorem \ref{thm-chars}(a): inserting the estimate \eqref{lattice} into \eqref{before-lattice}, we conclude that
\als{
\CX_{2k}(R)  
	&=  \frac{V_k(\log R)}{(P/\phi(P))^{2^{2k-1}-1}} 
		\sum_{\substack{P^+(d_i') \le y,\, d_i\le B \\ 1\le i\le 2k}} 
		\sum_{\substack{D_I\le q^4B \\ P^-(D_I)>y  \\ I\in\CS^-(2k)}}  
			 \frac{\chi(d_1')\cdots \chi(d_{2k}') }{ [d_1',\ldots ,d_{2k}']}    
			 \cdot \frac{\prod_{I\in\CS^-(2k)}\chi(D_I)}{\prod_{I\in\CS^-(2k)}D_I}  \\
		&\qquad+ O((\log R)^{d-1}(\log(q\log R))^{O(1)} ) .
}
We now remove the conditions $D_I\le q^4B$ and $d_i'\le B$ via \eqref{chi-sieve} and Rankin's trick, respectively. We conclude that
\als{
\CX_{2k}(R)  
	&=  \frac{V_k(\log R)}{(P/\phi(P))^{2^{2k-1}-1}} 
		\sum_{\substack{P^+(d_i') \le y \\ 1\le i\le 2k}} 
		\sum_{\substack{P^-(D_I)>y  \\ I\in\CS^-(2k)}}  
			 \frac{\chi(d_1')\cdots \chi(d_{2k}') }{ [d_1',\ldots ,d_{2k}']}    
			 \cdot \frac{\prod_{I\in\CS^-(2k)}\chi(D_I)}{\prod_{I\in\CS^-(2k)}D_I}  \\
		&\qquad+ O((\log R)^{d-1}(\log(q\log R))^{O(1)} ) .
}
Finally, we note that 
\[
\sum_{\substack{P^+(d_i') \le y \\ 1\le i\le 2k}}  \frac{\chi(d_1')\cdots \chi(d_{2k}') }{ [d_1',\ldots ,d_{2k}']}    
	= \prod_{p\le y} \left( 1+ \sum_{j=1}^\infty \sum_{\substack{j_1,\dots,j_{2k}\ge0 \\ \max\{j_1,\dots,j_{2k}\}=j}}
		\frac{\chi(p)^{j_1+\cdots+j_{2k}}}{p^j} \right) .
\]
The coefficient of $1/p$ is $(2^{2k-1}-1)\chi_0(p)+2^{2k-1}\chi(p)$. We thus conclude that
\als{
&(\phi(P)/P)^{2^{2k-1}-1} 
		\sum_{\substack{P^+(d_i') \le y \\ 1\le i\le 2k}} 
		\sum_{\substack{P^-(D_I)>y  \\ I\in\CS^-(2k)}}  
			 \frac{\chi(d_1')\cdots \chi(d_{2k}') }{ [d_1',\ldots ,d_{2k}']}    
			 \cdot \frac{\prod_{I\in\CS^-(2k)}\chi(D_I)}{\prod_{I\in\CS^-(2k)}D_I} \\
	&\qquad=
		 \prod_{p\le y} \left( 1+ \sum_{j=1}^\infty \sum_{\substack{j_1,\dots,j_{2k}\ge0 \\ \max\{j_1,\dots,j_{2k}\}=j}}
		\frac{\chi(p)^{j_1+\cdots+j_{2k}}}{p^j} \right) \left(1-\frac{1}{p}\right)^{2^{2k-1}-1} 
			\prod_{p>y}\left(1- \frac{\chi(p)}{p}\right)^{-2^{2k-1}}\\
	&\qquad=\prod_p
		\left( 1+ \sum_{j=1}^\infty \sum_{\substack{j_1,\dots,j_{2k}\ge0 \\ \max\{j_1,\dots,j_{2k}\}=j}}
		\frac{\chi(p)^{j_1+\cdots+j_{2k}}}{p^j} \right) \left(1-\frac{1}{p}\right)^{2^{2k-1}-1} 
		+ O\left(\frac{1}{y}\right) .
}
An easy calculation then completes the proof of Theorem \ref{thm-chars}(b).

\section{The analogy for exceptional Dirichlet characters}\label{excepsec}

In this section, we consider the quantity $\CX_{2k}(R)$ when $k=1$, or when when $\chi(p)=-1$ for most primes $p\le R$, and complete the proof of Theorem \ref{thm-chars}. Our arguments here resemble closely the ones of Section \ref{contour},  so we only highlight the key points here. Throughout the proof, we assume that $R\ge q^{2c_1}$, the complimentary case being trivial.

\subsection{Initial preparations}
Arguing as in Section \ref{dyadic-section}, we find that
\als{
\CX_{2k}(R)
	&= \frac 1 {(2i\pi)^{2k}} \idotsint\limits_{\substack{\Re(s_j)=\lambda^j/\log R \\ |\Im(s_j)|\le T \\ 1\le j\le 2k}}  
		\sum_{d_1,\ldots ,d_{2k}\geq 1}  \frac{\prod_{j=1}^{2k} \chi(d_j)d_j^{-s_j}}{[d_1,\ldots ,d_{2k}]} 
\prod_{j=1}^{2k} \hat{h}_R(s_j) (1-2^{-s_j}) 
	\dee s_1\cdots \dee s_{2k} \\
	&\qquad + O\left( \frac{1}{(\log R)^{100}}\right) ,
}
where $h$ is a smooth function with $h(x)=1$ for $x\le 1-1/(\log R)^{(2k-1)2^{2k+1}+200k+2}$ and $h(x)=0$ for $x>1$, $T=\exp\{(\log\log R)^2\}$, and $\lambda$ is some large parameter $>1$ to be chosen later.

By expanding as an Euler product, we find that for $\Re(s_1),\dots\Re(s_{2k})\ge -1/4k$ we have
\als{
\sum_{d_1,\ldots ,d_{2k}\geq 1}  \frac{\prod_{j=1}^{2k} \chi(d_j)d_j^{-s_j}}{ [d_1,\ldots ,d_{2k}]} 
	&=\prod_{p}\left(1+\sum_{I\in \CS^*(2k)}\frac{\chi(p)^{\#I}}{p^{1+s_I}}+O\left(\frac{1}{p^{2-2k/4^k}}\right)\right)\\
	&= P(\bs s) \prod_{I\in\CS^*(2k)} L(1+s_I,\chi^{\#I}),
}
where $P(\bs s)$ is given by an Euler product which converges absolutely in the region $\Re(s_j)>-1/4k$ for all $j$. Next, we set
\[
F(\bs s) = P(\bs s) \prod_{j=1}^{2k} \frac{\hat{h}_R(s_j)(1-2^{-s_j})}{R^{s_j}} 
\]
and
\[
\zeta_q(s):= L(s,\chi_0) = \zeta(s)\prod_{p|q}\left(1-\frac{1}{p^s}\right) ,
\]
so that
\eq{X-e1}{
\CX_{2k}(R) &= 
		\frac 1 {(2i\pi)^{2k}} \idotsint\limits_{\substack{\Re(s_j)=\lambda^j/\log R \\ |\Im(s_j)|\le T \\ 1\le j\le 2k}}  
	 F(\bs s) R^{s_{[2k]}}  \prod_{I\in \CS^-(2k)}L(1+s_I,\chi) \\
	 &\qquad \times \prod_{I\in \CS^+(2k)} 
	 	\zeta_q(1+s_I)
		\dee s_1\cdots \dee s_{2k}  
	+ O\left( \frac{1}{(\log R)^{100}}\right) .
}

Similarly to Section \ref{contour}, we let $\tilde{\CC}_\ell$ denote the class of complex-valued functions $f$ defined in a domain containing
\[
\tilde{\Omega}_\ell:=\{\bs s\in\C^\ell : |\Re(s_j)|< 1/5k, |\Im(s_j)|<T+1\ (1\le j\le \ell)\},
\]
it is analytic in $\tilde{\Omega}_\ell$, and its derivatives satisfy the bound 
\eq{F'-bound}{
\frac{\partial^{j_1+\cdots+j_\ell}f}{\partial s_1^{j_1}\cdots \partial s_\ell^{j_\ell}}(\bs s)
	\ll_{j_1,\dots,j_\ell}  \prod_{m=1}^\ell \frac{[(1+(qT)^{-\Re(s_m)})(\log\log R)^{j_m}]^{O(1)}}{|s_m|+1} 
}
for all $j_1,\dots,j_\ell\ge0$ and all $\bs s=(s_1,\dots,s_\ell)\in\tilde{\Omega}_\ell$. 

Since there is an absolute constant $c_1>0$ such that
\eq{L-bound}{
L^{(j)}(s,\psi) \ll_m (1+(q+|t|)^{c_1(1-\sigma)}) \log^{j+1}(q+|t|)  +  \frac{{\bf 1}_{\psi=\chi_0}}{|s-1|^{j+1}} 
}
for $j\in\Z_{\ge0}$, $\psi\in\{\chi,\chi_0\}$ and $j\in\{0,1\}$, a standard consequence of bounds on the exponential sum $\sum_{n\le N,\ n\equiv a\mod q} n^{it}$ (see, for example, Lemma 4.1 in \cite{Kou13}), we have that $F$ is in the class $\tilde{\CC}_{2k}$.

\subsection{The case $k=1$} 
We first deal with the case $k=1$ that is easy and will help us clarify some of the technical details of the argument.  When $k=1$, we move the variable $s_2$ to the line $\Re(s_2)=-\epsilon$, for a sufficiently small $\epsilon$. The contribution of the horizontal contours is $\ll (\log R)^{O(1)}/T$, and the contribution of the contour $\Re(s_2)=-\delta$ is $\ll R^{-\delta}$, for some positive $\delta=\delta(\epsilon)$ by \eqref{L-bound}, and by our assumptions that $F\in\CC_2$ and that $R\ge q^{2c_1}$. In conclusion,
\[
\CX_2(R) = 
		\frac{1}{2i\pi} \int\limits_{\substack{\Re(s_1)=\lambda/\log R \\ |\Im(s_1)|\le T \\ 1\le j\le 2k}}  
	 F(s_1,-s_1)L(1+s_1,\chi)L(1-s_1,\chi) \dee s_1 + O\left( \frac{1}{(\log R)^{100}}\right) .
\]
Finally, we move $s_1$ to the line $\Re(s_1)=0$. No poles are encountered, and the contribution of the horizontal lines is easily seen to be $\ll (\log R)^{O(1)}/T$, so that
\als{
\CX_2(R) 
	&= 
	\frac{1}{2\pi} \int_{-T}^T 
	 F(it,-it) |L(1+it,\chi)|^2 \dee t + O\left( \frac{1}{(\log R)^{100}}\right)  \\
	& =\frac{1}{2\pi} \int_{-(\log R)^{102}}^{(\log R)^{102}}
	 P(it,-it) \abs{ L(1+it,\chi) \hat{h}_R(it) (1-2^{it}) }^2 \dee t
	+ O\left( \frac{1}{(\log R)^{100}}\right) ,
}
since $\hat{h}_R(it) (1-2^{it})\ll 1/(1+|t|)$. Finally, using \ref{h-mellin} to replace $\hat{h}_R(s)$ by $R^s/s$, choosing $C$ to be large enough, and then extending the range of integration to $\R$ yields the estimate
\als{
\CX_2(R) 
	& =\frac{1}{2\pi} \int_{-\infty}^\infty 
	  P(it,-it)\abs{ L(1+it,\chi) \cdot \frac{\sin(t(\log 2)/2)}{t}  }^2 \dee t
	+ O\left( \frac{1}{(\log R)^{100}}\right) ,
}
which proves Theorem \ref{thm-chars}(a). (Obviously, we can obtain a much stronger error term, but we have chosen to content ourselves with a more qualitative result.)

\subsection{Contour shifting}
Next, we focus on the case $k\ge2$ and prove Theorem \ref{thm-chars}(c). 
From now on, we will always be working under the assumptions and notations
\[
L(\beta,\chi)=0,\quad \beta>1-1/(100\log q),\quad Q=e^{1/(1-\beta)} .
\]
As it is well-known, we have that $\sum_{q<p\le Q}(1+\chi(p))/p\ll1$ and $\sum_{p>Q}\chi(p)/p \ll 1$ (see, for example, Theorems 2.1 and 2.4 in \cite{kou}). As a consequence, we note once and for all that
\eq{L(1,chi)}{
L(1,\chi) \asymp \frac{1}{\log Q}  \prod_{p\le q} \left(1+\frac{1+\chi(p)}{p} \right)  .
}

As in Section \ref{contour}, we shift the contours of the variables $s_1,\dots,s_{2k}$ in a certain order. As in that section, to describe the general contour shifting argument after $N$ steps, $0\le N\le 2k$, we fix sets $I_1,\dots,I_N$, and distinct integers $j_n\in I_n$ for each $n$. Recall, also, that $s_j$ denotes a variable and $x_j$ denotes a linear form. We then define 
\[
V_n= \Span_\Q (x_{I_1},\dots,x_{I_n}) 
\quad\text{and}\quad
\CI_N = \{I\in\CS(2k) : x_I\in V_n\}\quad(0\le n\le N) .
\]
Imposing the conditions $x_{I_1}=\cdots=x_{I_n}=0$, we may write $x_{j_1},\dots,x_{j_n}$ in terms of the variables $s_j$ with $j\in[2k]\setminus\{j_1,\dots,j_N\}$. Hence $x_I$ becomes a linear form $L_{N,I}$ in the variables $s_j$ with $j\in[2k]\setminus\{j_1,\dots,j_N\}$. Moreover, $x_I\in V_N$ if and only if $L_{N,I}=0$. 

As we will see, we will always be able to assume that $j_n=2k-n+1$. Let $d\in\Z_{\ge0}$ and Given the above set-up with $j_n=2k-n+1$, an integer $d\in\Z_{\ge0}$ and $\bs h= (h_{n,I})_{0\le n\le N,\, I\in\CS^*(2k)}$ be a vector of non-negative integers such that:
\begin{itemize}
\item $0=h_{0,I}\le h_{1,I}\le \cdots \le h_{N,I}$ for $I\in\CS^*(2k)$;
\item If $I\in \CI_n\setminus\CI_{n-1}$ for some $n\in\{1,\dots,N\}$, then $h_{m,I}=h_{n,I}$ for all $m\ge n$;
\item $X_N\ge N+d$, where
\[
X_N :=  \#(\CI_N\cap\CS^+(2k)) - \sum_{I\in\CS^-(2k)\cup(\CS^+(2k)\setminus \CI_N)} h_{N,I} .
\]
\end{itemize}

A function $J:\R_{\ge2}\to \C$ is a called a {\bf fundamental component of level $N$ and of type $(\bs I,\bs h,d)$} if:
\begin{itemize}
\item when $N=2k$, it equals 
\[
J(R) =  (\log R)^{X_N-N-d} \prod_{I\in\CS^-(2k)}
		 	L^{(h_{N,I})}(1,\chi)   \,; 
\]
\item when $N<2k$, it is of the form
\als{
J(R)&=\left(\frac{\phi(q)}{q}\right)^{M_N}  \frac{(\log R)^{X_N - N-d}}{(2i\pi)^{2k-N}}
	 \prod_{I\in\CS^-(2k)\cap \CI_N}  
		 	L^{(h_{N,I})}(1,\chi) 	\\
	&\qquad \times \idotsint \limits_{\substack{\Re(s_j)=\lambda_j/\log R\\ |\Im(s_j)|\le T \\ 1\le j\le 2k-N}}
		 G(\bs s)R^{E_N(\bs s)} 
		 	\prod_{I\in\CS^-(2k)\setminus\CI_N }  
		 	L^{(h_{N,I})}(1+L_{N,I}(\bs s),\chi)    \\
	 &\qquad\qquad\times \prod_{I\in\CS^+(2k)\setminus\CI_N }  
		 	\zeta_q^{(h_{N,I})}(1+L_{N,I}(\bs s)) 
				 	\dee s_{2k-N}\cdots \dee s_1  ,
}
where $\lambda_j/\lambda_{j-1}\ge \lambda$, $E_N(\bs s) := L_{N,[2k]}(\bs s)$, 
\[
M_N:=  \sum_{n=1}^N \sum_{I\in(\CI_n\setminus\CI_{n-1})\cap \CS^+(2k) } (h_{n-1,I}+1) ,
\]
and $G$ is a function in the variables $s_1,\dots,s_{2k-N}$ that belongs to the class $\CC_{2k-N}$, given by
\[
G(\bs s)=   F(L_{N,\{1\}}(\bs s),\dots,L_{N,\{2k\}}(\bs s)) 
\]
when $d=0$. In particular, $G$ is non-vanishing in $\Omega_{2k-N}$ when $d=0$ by \eqref{L-bound}. 
\end{itemize}

As in Section \ref{contour-shifting}, a fundamental component of level $N$ is called reducible when $N<2k$ and $E_N\neq0$. Otherwise, it is called irreducible. With this above terminology, the integral on the right hand side of \eqref{X-e1} is a reducible fundamental component of level 0 and of type $(\emptyset,\emptyset,0)$.

Again as in Section \ref{contour-shifting}, when $E_N\neq0$ there are some $\gamma_j\in \Q$ with $\gamma_{j_{N+1}}\neq0$ such that 
\[
E_N(\bs x) = \gamma_1 x_1+\gamma_2x_2+\cdots + \gamma_{j_{N+1}} x_{j_{N+1}}
\]
If $\lambda$ is big enough, then the sign of $\Re(E_N(\bs s))$ throughout the region of integration is constant and equal to the sign as $\gamma_{j_{N+1}}$. 

The analogies of Lemmas \ref{contour-lem1} and \ref{contour-lem2} can be proven in this setting:

\begin{lem}\label{chi-contour-lem1} Assume the above setup, let $J(R)$ be a reducible fundament component of level $N<2k$, and let $\gamma_{j_{N+1}}$ be as above. Suppose, further, that $k\ge2$ and that $e^{(\log q)^2}\le R\le Q$. All implied constants below depend at most on $k$.
\begin{enumerate}
\item If $\gamma_{j_{N+1}}>0$, then $J(R)$ is a linear combination of fundamental components of level $N+1$, up to an error term of size $\ll 1/\log R$. Moreover, the coefficients of this linear combination are $\ll (\log q)^{O(1)}$.
\item If $\gamma_{j_{N+1}}<0$, then $J(R)\ll T^{-1+o(1)}$.
\end{enumerate}
\end{lem}

We iterate the above lemma till all the fundamental components we are dealing with are irreducible. For such components, we have the following asymptotic formula.

\begin{lem}\label{chi-contour-lem2} Assume the above setup. Suppose, further, that $k\ge2$ and that $e^{(\log q)^2}\le R\le Q$. If $J(R)$ is an irreducible fundamental component, then there is some complex number $c\ll(\log q)^{O(1)}$ such that
\[
J(R) = c(\log R)^{\binom{2k}{k}-2k}+ O((\log(q\log R))^{O(1)}(\log R)^{\binom{2k}{k}-2k-1}) .
\]
All implied constants depend at most on $k$.
\end{lem}

Since the integral on the right hand side of \eqref{X-e1} is a reducible fundamental component of level 0, we apply Lemma \ref{contour-lem1} repeatedly to write it as a linear combination of irreducible fundamental components, and then estimate these components by Lemma \ref{contour-lem2}. This proves that there is a constant $c_k(\chi)\ll(\log q)^{O(1)}$ such that
\eq{chi-thm-goal}{
\CX_{2k}(R) = c_k(\chi) \cdot (\log R)^{\binom{2k}{k}-2k}
	+ O\left( (\log(q\log R))^{O(1)}  (\log R)^{\binom{2k}{k} - 2k -1 } \right)   
}
when $k\ge2$ and $e^{(\log q)^2}\le R\le Q$. We will show that $c_k(\chi) \gg (\log q)^{-O(1)}$ in Section \ref{chi-lb} and complete the proof of Theorem \ref{thm-chars}. The key intermediate Lemmas \ref{chi-contour-lem1} and \ref{chi-contour-lem2} are proven in the next section.

\subsection{Proof of the auxiliary Lemmas \ref{chi-contour-lem1} and \ref{chi-contour-lem2}}

For easy reference, we record the following bound that we will repeatedly use: for $R\le Q$, we have
\eq{L-bound-contour}{
(\log R)^{X_N - N-d} \prod_{I\in\CS^-(2k)\cap \CI_N}L^{(h_{N,I})}(1,\chi)
	\ll (\log q)^{O(1)} (\log R)^{\SA(V_N) - N - D} ,
}
where
\[
D = d + \sum_{I\in\CS(2k)\setminus\CI_N} h_{N,I} +\sum_{\substack{I\in \CS^-(2k)\cap\CI_N\\ h_{N,I}\ge1}} (h_{N,I}-1)  
\]
and $\SA(V_N)$ is defined in Section \ref{combinatorial}. Indeed, when $h_{N,I}=0$ with $I\in\CS^-(2k)\cap\CI_N$, we use \eqref{L(1,chi)} to find that have that 
\[
L(1,\chi)\ll \frac{(\log q)^2}{\log Q} \le \frac{(\log q)^2}{\log R} .
\]
Otherwise, we use the bound $L^{(h_{N,I})}(1,\chi) \ll (\log q)^{h_{N,I}+1}$. Putting these estimates together yields \eqref{L-bound-contour}.

\begin{proof}[Proof of Lemma \ref{chi-contour-lem1}]
(a) Here $\gamma_{j_{N+1}}>0$. We make the change of variables
\[
s_j' = s_j\ (1\le j<j_{N+1}),\ s_j'=s_{j+1}\ (j_{N+1} \le j\le 2k-N),\ s_{2k-N}'=s_{j_{N+1}},
\]
and similarly for the forms $x_j$ and the parameters $\lambda_j$. 

Next, we shift the $s_{2k-N}'$ contour to the line $\Re(s_{2k-N}')=-\epsilon$ for a small enough $\epsilon>0$. The integral on the new contour is $\ll (\log R)^{O(1)}/T$, which is negligible, and we are left with having to analyze the pole contributions. The poles occur when $L_{N,I_{N+1}}(\bs s')=0$ for some $I_{N+1}\in\CS^+(2k)\setminus\CI_N$ such that the coefficient of $s_{2k-N}'$ in $L_{N,I_{N+1}}$ is non-zero. As we discussed in the previous section, imposing the relation $L_{N,I_{N+1}}(\bs x')=0$ to write the form $x_{2k-N}'$ as a linear combination of $x_1',\dots,x_{2k-N-1}'$, say $x_{2k-N}'= C(x_1',\dots,x_{2k-N-1}')$. In particular, $L_{N,I}(\bs x')$ becomes a linear form $L_{N+1,I}$ in the variables $x_1',\dots,x_{2k-N-1}'$. We also set $E_{N+1}=L_{N+1,[2k]}$ and let $\CI_{N+1}$ be the set of $I\subset[2k]$ such that $L_{N+1,I}=0$. 

The generic order of the pole at $s_{2k-N}'=C(s_1',\dots,s_{2k-N-1}')$ is 
\eq{pole-order-chi}{
m = \sum_{I\in(\CI_{N+1}\setminus\CI_N)\cap \CS^+(2k) } (h_{N,I}+1)  - \nu,
}
where $\nu$ is the generic order of the zero of
\[
G(\bs s) \prod_{I\in\CS^-(2k)\cap(I_{N+1}\setminus I_N)} L^{(h_{N,I})}(1+L_{N,I}(\bs s),\chi) 
\]
at the same point. In particular, $\nu=0$ if $d=0$ (so that $G(\bs s)=F(L_{N,\{1\}}(\bs s),\dots,L_{N,\{2k\}}(\bs s))$) and $h_{N,I}=0$ for all $I\in\CS^-(2k)\cap(\CI_{N+1}\setminus \CI_N)$. By a direct computation, we then find that
\eq{poweroflog-chi}{
X_{2k-1}+m = \#(\CI_{N+1}\cap\CS^+(2k)) 
	-  \sum_{I\in\CS^-(2k)\cup(\CS^+(2k)\setminus \CI_{N+1})} h_{N,I}  - \nu .
}
We note that $m\ge1$ for all $N\le 2k-1$ when $k\ge2$ and $\nu=0$; otherwise, we would have that $\CS^+(2k)\subset\CI_N$, which is impossible because the dimension of $V_N$ is $N$, whereas the dimension of the span of the linear forms $s_I$, $I\in\CS^+(2k)$, is $2k$.

We want to understand the pole contribution when $m\ge1$. We separate two subcases:

\medskip

\noindent
{\bf Case 1 of the proof of Lemma \ref{chi-contour-lem1}: $N=2k-1$.} 
In this case, we have that $s_j' = L_{2k-1,\{j\}}(s_1') = a_j s_1'$ for all $j$, where $a_j\in\Q$. Then the pole occurs necessarily when $s_1'=0$. Thus $\CI_{2k}=\CS(2k)$, and we obtain an evaluation of $J(R)$ as finite linear combination of powers of $\log R$, the highest of which has exponent
\als{
X_{2k-1}+m-2k 
	&= 2^{2k-1} - 2k-1 - \sum_{I\in \CS^-(2k)} h_{2k-1,I} -\nu ,
}
up to an admissible error. The coefficients of the polynomial in $\log R$ are given in terms of the derivatives $L^{(j)}(1,\chi)$. Specifically, the coefficient of $(\log R)^{X_{2k-1}+m-2k-h}$, $0\le h\le m-1$, is a linear combination of products of the form
\[
\left(\frac{\phi(q)}{q}\right)^{M_{2k}} \prod_{I\in\CS^-(2k)}
		 	L^{(h_{2k,I})}(1,\chi)    ,
\]
with the coefficients of this linear combination being $\ll1$, and with the parameters $h_{2k,I}$ satisfying $h_{2k,I}\ge h_{2k-1,I}$ with equality if $I\in\CI_{2k-1}\setminus\{0\}$, and $\sum_{I\in \CS^-(2k)}(h_{2k,I} - h_{2k-1,I}) \le h$. Arguing as in the proof of \eqref{L-bound-contour}, we find that
\[
J(R) \ll \frac{(\log q)^{O(1)} (\log R)^{X_{2k-1}+m-2k -h}}{(\log Q)^{\#\{I\in \CS^-(2k) : h_{2k,I}=0\}}} 
	\le\frac{(\log q)^{O(1)} (\log R)^{2^{2k-1} - 2k-1 - \sum_{I\in \CS^-(2k)} h_{2k,I} }}{(\log R)^{\#\{I\in \CS^-(2k) : h_{2k,I}=0\}}} ,
\]
where we used that $Q\le R$, $\nu\ge0$ and $h\ge \sum_{I\in \CS^-(2k)}(h_{2k,I} - h_{2k-1,I})$. We thus conclude that
\[
J(R) \ll \frac{(\log q)^{O(1)}}{(\log R)^{2k+1}} \ll \frac{1}{\log R} .
\]
This completes the proof of the lemma in this case (the linear combination is empty).

\medskip

\noindent
{\bf Case 2 of the proof of Lemma \ref{chi-contour-lem1}: $N\le 2k-2$.} Arguing as in the proof of part (b) of Lemma \ref{contour-lem1}, the contribution of the pole $s_{2k-N}'=C(s_1',\dots,s_{2k-N-1}')$ to $J(R)$ can be seen to be a linear combination of terms of the form
\als{
	\left(\frac{\phi(q)}{q}\right)^{M_{N+1}} 
		\frac{(\log R)^{X_N +m - h - N-1-d}}{(2i\pi)^{2k-N-1}}  
	\idotsint\limits_{\substack{ \Re(s_j)= \lambda_j'/\log R,\,  |\Im(s_j)|\le T \\ 1\le j\le 2k-N-1}}
		\tilde{G}(\bs s) R^{E_{N+1}(\bs s)} \\
		\times
		\prod_{I\in\CS^-(2k)}L^{(h_{N+1,I})}(1+L_{N+1,I}(\bs s),\chi)  \\
		\times \prod_{I\in\CS^+(2k)\setminus\CI_{N+1}}
		 	\zeta_q^{(h_{N+1,I})}(1+L_{N+1,I}(\bs s)) 
		 \dee s_{2k-N-1}\cdots \dee s_1  
}
plus an error term of size $O(1/\log R)$, where $h\in\{0,\dots,m-1\}$, $h_{N+1,I}\ge h_{N,I}$ with equality if $I\in \CI_{N+1}\setminus\{0\}$, and $\sum_{I\in \CS^-(2k)\cup(\CS^+(2k)\setminus\CI_{N+1})} (h_{N+1,I}-h_{N,I}) \le h $. Relation \eqref{poweroflog-chi} then implies that the power of $\log R$ is then $X_{N+1}-N-1-d'$ with
\[
d' =  d +\nu + h - \sum_{I\in \CS^-(2k)\cup(\CS^+(2k)\setminus\CI_{N+1})} (h_{N+1,I}-h_{N,I}) \ge 0.
\]
Moreover, $X_{N+1}-d' = X_N+m-h\ge N+1$. This completes the proof of part (a).

\medskip

(b) Here $\gamma_{j_{N+1}}<0$. We treat this case using the same argument as in part (b) of Lemma \ref{contour-lem1}, with the difference that the contours of  $s_{2k-N},s_{2k-N-1},\dots,s_{j_{N+1}}$ are shifted to the lines $\Re(s_j) = \lambda^j/((\log q)+(\log T)^{3/2})$, $j_{N+1}\le j\le 2k-N$. Since $q\le e^{\sqrt{\log R}}$ by assumption, we find that
\[
J(R) \ll T^{-1+o(1)} ,
\]
as needed.
\end{proof}

\begin{proof}[Proof of Lemma \ref{chi-contour-lem2}] We separate three cases.

\medskip

\noindent
{\bf Case 1 of the proof of Lemma \ref{chi-contour-lem2}: $N=2k$.} Here $\CI_{2k}=\CS(2k)$ and thus $J(R)\ll (\log q)^{O(1)}/(\log R)^{2k+1}$ by \eqref{L-bound-contour}, which proves Lemma \ref{chi-contour-lem2} in this case.

\medskip

\noindent
{\bf Case 2 of the proof of Lemma \ref{chi-contour-lem2}: $N=2k-1$.} 
As in Case 1 of the proof of Lemma \ref{chi-contour-lem1}, we have that $s_j= L_{2k-1,\{j\}}(s_1) = a_j s_1$ for all $j$, where $a_j\in\Q$. Then
\als{
J(R)
	&= \left(\frac{\phi(q)}{q}\right)^{M_{2k-1}}  \frac{(\log R)^{X_{2k-1} - 2k+1-d}}{2\pi i} 
		\prod_{I\in\CS^-(2k)\cap\CI_{2k-1}} L^{(h_{2k-1,I})}(1,\chi)  \\
	&\quad\times	\int \limits_{\substack{\Re(s_1)=\lambda_1/\log R\\ |\Im(s_1)|\le T}} 
		 G_{2k-1}(s_1)
		 	\prod_{I\in\CS^-(2k)\setminus\CI_{2k-1}}
		 	L^{(h_{2k-1,I})}(1+a_Is_1, \chi)  \\
	&\qquad \times			\prod_{I\in\CS^+(2k)\setminus\CI_{2k-1} }  
		 	\zeta_q^{(h_{2k-1,I})}(1+a_Is_1) \dee s_1 ,
}

We first show a crude bound on $J(R)$, that will allow us to focus on a more convenient subcase. We move the line of integration to $\Re(s_1)=\lambda/\log(qT)$ and use \eqref{L-bound} to find that the integral over $s_1$ is $\ll(\log(q\log R))^{O(1)}$. Together with \eqref{L-bound-contour} and Proposition \ref{CombProp}, this implies that
\[
J(R) \ll (\log(q\log R))^{O(1)} \cdot (\log R)^{\binom{2k}{k} - 2k -1 } ,
\]
unless $d=0$, $h_{2k-1,I}\in\{0,1\}$ for all $I\in\CS^-(2k)\cap\CI_{2k-1}$, $h_{2k-1,I}=0$ for all $I\in\CS(2k)\setminus \CI_{2k-1}$, half of the $a_j$'s are $+b$ and the other half are $-b$, for some $b\neq0$.

We have thus reduced proving the lemma to the case when $d=0$, $h_{2k-1,I}\in\{0,1\}$ for all $I\in\CS^-(2k)\cap\CI_{2k-1}$, $h_{2k-1,I}=0$ for all $I\in\CS(2k)\setminus \CI_{2k-1}$, half of the $a_j$'s are $+b$ and the other half are $-b$, where $b\neq0$. In particular, we find that $\CI_{2k-1}\cap\CS^-(2k)=\emptyset$ and that $\#(\CI_{2k-1}\cap\CS^+(2k))=\binom{2k}{k}-1$, so that
\als{
J(R)
	= \left(\frac{\phi(q)}{q}\right)^{M_{2k-1}}  \frac{(\log R)^{\binom{2k}{k}-2k}}{2\pi i} 
		&	\int \limits_{\substack{\Re(s_1)=\lambda/\log R\\ |\Im(s_1)|\le T}} 
		 G(s_1)
		 	\prod_{I\in\CS^-(2k)}
		 	L(1+a_Is_1, \chi)  \\
	&\qquad \times	\prod_{I\in\CS^+(2k)\setminus\CI_{2k-1} }  
		 	\zeta_q(1+a_Is_1) \dee s_1 .
}
Since $\nu=0$ here, we saw before that the integrand has a genuine pole of order $m\ge1$ at $s_1=0$ by a dimension argument. In fact, we have that $m\ge2$: indeed, we know that $[2k]\in \CI_{2k-1}$ by our assumption that $E_{2k-1}=0$, so that $I\in \CI_{2k-1}$ if and only if $[2k]\setminus I\in \CI_{2k-1}$. In particular, since we know that there is at least one $I\in\CS^+(2k)\setminus\CI_{2k-1}$, there must be at least two such $I$'s.

The presence of this pole make the estimation of $J(R)$ tricky. In particular, we cannot shift the contour to the line $\Re(s_1)=0$ as in Case 2 of Section \ref{contour-shifting}. Instead, we write 
\[
L(1+a_Is_1,\chi) = L(\beta+a_Is_1,\chi)+ \Delta(a_Is_1) ,
\]
where
\[
\Delta(s):= L(1+s,\chi)-L(\beta+s,\chi) .
\]
We thus find that
\[
\Delta(s) = \int_\beta^1 L'(u+s,\chi)\dee u \ll \frac{\log(q+|t|)}{\log Q} \quad(\sigma\ge -1/\log(q+|t|)),
\]
by \eqref{L-bound} and the assumption that $\beta>1-1/(100\log q)$. With this notation,
\[
J(R)= M+E,
\]
where
\als{
	M = \left(\frac{\phi(q)}{q}\right)^{M_{2k-1}}  \frac{(\log R)^{\binom{2k}{k}-2k}}{2\pi i} 
		&	\int \limits_{\substack{\Re(s_1)=\lambda/\log R\\ |\Im(s_1)|\le T}} 
		 G_{2k-1}(s_1)
		 	\prod_{I\in\CS^-(2k)}
		 	L(\beta+a_Is_1, \chi)  \\
	&\qquad \times	\prod_{I\in\CS^+(2k)\setminus\CI_{2k-1} }  
		 	\zeta_q(1+a_Is_1) \dee s_1,
}
and $E$ is a sum of similar expressions where at least one of the $L(\beta+a_Is_1, \chi)$ factors is replaced by $\Delta(a_Is_1)$. 

First, we bound $E$. Moving $s_1$ to the line $\Re(s_1) = 1/\log(qT)$, we find that 
\[
E\ll \frac{(\log(q\log R))^{O(1)} (\log R)^{\binom{2k}{k}-2k}}{\log Q}
	\le  (\log(q\log R))^{O(1)} (\log R)^{\binom{2k}{k}-2k-1} .
\]
Finally, we estimate $M$ by moving the line of integration of $s_1$ to $\Re(s_1)=0$, and use the fact that $G(s_1)\ll 1/(1+|s_1|)^{2k}$ (note that the integrand is now analytic, since the pole of the $\zeta_q$'s is annihilated by the zeroes of the $L(\cdot,\chi)$'s) to find that
\als{
J(R)
	&= \left(\frac{\phi(q)}{q}\right)^{M_{2k-1}}  \frac{(\log R)^{\binom{2k}{k}-2k}}{2\pi} 
		\int_{-T}^T
		 G(it)
		 	\prod_{I\in\CS^-(2k)}
		 	L(\beta+ia_It, \chi)  \\
	&\qquad \times	\prod_{I\in\CS^+(2k)\setminus\CI_{2k-1} }  
		 	\zeta_q(1+ia_It) \dee t
			+ O((\log(q\log R))^{O(1)} (\log R)^{\binom{2k}{k}-2k-1}) .
}
Note that $G(s)=F(a_1s,\dots,a_{2k}s)$ here, because $d=0$. When $|t|\le \log R$, we use \eqref{h-mellin} to replace $\hat{h}_R$ by $R^s/s$, and when $|t|>\log R$ we use the bound $\hat{h}_R(s)\ll R^\sigma/|s|$ by \eqref{h-bound-prelim}. Taking $C$ to be large enough, we thus conclude that
\als{
\frac{J(R)}{(\log R)^{\binom{2k}{k}-2k}}
	&=  \frac{(\phi(q)/q)^{M_{2k-1}}}{2\pi} 
		\int_{-\infty}^\infty
		 F(ia_1t,\dots,ia_{2k}t)
		 	\prod_{I\in\CS^-(2k)}
		 	L(\beta+ia_It, \chi)  \\
	&\qquad \times	\prod_{I\in\CS^+(2k)\setminus\CI_{2k-1} }  
		 	\zeta_q(1+ia_It) \prod_{j=1}^{2k} \frac{1-2^{-ia_jt}}{ia_jt} \dee t
			+ O\left( \frac{(\log(q\log R))^{O(1)}}{\log R}\right ) .
}
This completes the proof of Lemma \ref{chi-contour-lem2} in this case.

\medskip

\noindent
{\bf Case 3 of the proof of Lemma \ref{chi-contour-lem2}: $N\le 2k-2$.} 
As in the corresponding case of the proof of Lemma \ref{contour-lem2}, and using \eqref{L-bound-contour}, we find that
\als{
J(R)
	&\ll (\log(q\log R))^{O(1)}  (\log R)^{X_N-N-d} 
		\prod_{I\in\CS^-(2k)\cap\CI_N}|L^{(h_{N,I})}(1,\chi) |    \\
	&\ll (\log(q\log R))^{O(1)} (\log R)^{\SA(V_N)-N} .
}
Proposition \ref{CombProp}(c) then implies that $\SA(V_N)-N\le \binom{2k}{k}-2k-2$, thus completing the proof of Lemma \ref{chi-contour-lem2}.
\end{proof}

\subsection{Lower bounds}\label{chi-lb} In order to complete the proof of Theorem \ref{thm-chars}, we show that constant $c_k(\chi)$ in \eqref{chi-thm-goal} is $\gg(\log q)^{-O(1)}$. In order to do so, we follow the argument of Section \ref{lb} and prove that, for any $\epsilon>0$, there is a constant $c_k>0$ such that
\eq{chars-lb}{
\CX_{2k}(R) \ge c_k  \frac{(\log R)^{\binom{2k}{k}-2k-\epsilon}}{(\log q)^{O(1)}} 	
	- O_\epsilon\left( (\log(q\log R))^{O(1)}(\log R)^{\binom{2k}{k}-2k-1}\right) ,
}
provided that $\sqrt{\log Q}\ge \log R\ge 2c_1\log q$, where $c_1$ is the constant appearing in \eqref{L-bound}. (For this section, all constants will be independent of $\epsilon$, unless specified by a subscript, as above.)

We set
\[
y=\exp\{(\log R)^{1-\epsilon'}\}
\quad\text{and}\quad
Y=\exp\{(\log R)^{1-\epsilon'/2}\},
\]
where $\epsilon'$ will be taken to be small enough in terms of $\epsilon$, and focus our attention on integers of the form $n=ap_1\cdots p_k$ with $P^+(a)\le y$, $a\le Y$, and $p_1,\dots,p_k$ are distinct primes such that $p_\ell>\sqrt{R}$ and $\chi(p_\ell)=-1$ for all $\ell\in\{1,\dots,k\}$. Then
\[
\CX_{2k}(R) \ge \frac{\prod_{p\le R}(1-1/p) }{k!} 
	\sum_{\substack{P^+(a)\le y \\ a\le Y}}
		\sum_{\substack{p_j>\sqrt{R} \\ \chi(p_j)=-1 \\ 1\le j\le k}}\frac{\mu^2(p_1\cdots p_k)}{ap_1\cdots p_k} 
		\left( \sum_{\ell=1}^k \sum_{\substack{R/(2p_\ell)<d\le R/p_\ell \\ d|a}} \chi(d) \right)^{2k} .
\]

The next step is to drop the condition that $a\le Y$ by an application of Rankin's trick and to remove the condition that the $p_j$'s are distinct. We further replace the sharp cut-off $R/(2p_\ell)<d\le R/p_\ell$ by the smooth cut-off $h(\frac{\log(dp_\ell)}{\log R})-h(\frac{\log(2dp_\ell)}{\log R})$, where $h(x)=1$ for $x\le 1-1/(\log R)^B$ and $h(x)=0$ for $x\ge 1$, with $B$ is sufficiently large. To conclude, we have that
\als{
\CX_{2k}(R) 
	&\ge \frac{\prod_{p\le R}(1-1/p)}{k!(\log R)^k} 
	\sum_{P^+(a)\le y }
		\sum_{\substack{\sqrt{R}< p_j \le R \\  \chi(p_j)=-1 \\ 1\le j\le k}}\frac{\prod_{j=1}^k\log p_j}{ap_1\cdots p_k} 
		\left( \sum_{\ell=1}^k \sum_{d|a} \chi(d) w\left(\frac{\log(dp_\ell)}{\log R}\right)\right)^{2k}  \\
	&\qquad 	+ O\left(\frac{1}{\log R} \right) ,
}
where $w(x) = h(x) - h(x+\log 2/\log R)$. 

The rest of the proof follows the argument of Section \ref{lb}, with a small twist, as we will explain in the end. We expand the $2k$-th power and focus on a convenient subset of summands. We then conclude that
\[
\CX_{2k}(R) \ge \frac{\prod_{y<p\le R}(1-1/p)}{k!}
		 \sum_{\bs J\in\mathcal{J}} X(\bs J)    +   O\left( \frac{1}{\log R}\right) 
\]
with
\[
X(\bs J)
	= \frac{1}{(\log R)^k} 
		\sum_{\substack{P^+(d_j)\le y \\ 1\le j\le 2k}} \frac{\chi(d_1)\cdots \chi(d_{2k})}{[d_1,\dots,d_{2k}]} \\
		 \prod_{\ell=1}^k \sum_{\sqrt{R}<p_\ell\le R} \frac{(1-\chi(p_\ell))\log p_\ell}{2p_\ell}
				\prod_{j\in J_\ell} w\left( \frac{\log(p_\ell d_j)}{\log R} \right),
\] 
where $\bs J$ is as in Section \ref{lb}. We set 
\[
S:= \sum_{\sqrt{R}<p\le R} \frac{(1-\chi(p))\log p}{2p} 
	\asymp \log R 
\] 
and write $\CL$ for the set of $\ell\in\{1,\dots,k\}$ such that $J_\ell\neq\emptyset$. Then, using Perron's formula $2k$ times to write each appearance of $w$ as an integral of $\hat{w}_R$, we find that
\als{
X(\bs J) = \frac{S^{k-\#\CL}(\log R)^{-k}}{(2\pi i)^{2k}}
		\idotsint\limits_{\substack{\Re(s_j)=1/\log R \\ 1\le j\le 2k}}\ &
		\sum_{\substack{P^+(d_j)\le y \\ 1\le j\le 2k}}
			\frac{\prod_{j=1}^{2k}\chi(d_j)d_j^{-s_j}}{[d_1,\dots,d_{2k}]}
			\left(\prod_{\ell\in\CL} \sum_{p_\ell>\sqrt{R}} \frac{(1-\chi(p_\ell)) \log p_\ell}{2p_\ell^{1+s_{J_\ell}}} \right)\\
		&\times	
			\left( \prod_{j=1}^{2k} \hat{w}_R(s_j)  \right) 
					\dee s_1\cdots \dee s_{2k} +O\left(\frac{1}{\log R}\right),
}
where $s_J=\sum_{j\in J}s_j$, as usual, and the condition that $p_\ell\le R$ was dropped because it is encoded in the support of $w$. By possibly re-indexing the variables $s_1,\dots,s_{2k}$, we may assume that $\CL=\{1,\dots,L\}$, where $L=\#\CL$, and that $\max J_\ell= 2k-L+\ell$ for all $\ell\in\{1,\dots,L\}$ with $L=\#\CL$. As in Section \ref{lb}, we will move the variables $s_{2k-L+1},\dots,s_{2k}$ to the left. We note that 
\[
 \sum_{p>\sqrt{R}} \frac{(1-\chi(p))\log p}{p^{1+s}}
 	= -\frac{\zeta'}{\zeta}(1+s) + \frac{L'}{L}(1+s,\chi) + O(1) - \sum_{p\le R^{1/2}} \frac{\log p}{p^{1+s}} 
\]
for $\Re(s)\ge-1/3$. The above has simple poles at $s=0$ and $s=\beta-1$, each of residue 1. Therefore, using the argument leading to \eqref{W(J)}, we find that
\eq{X-int}{
X(\bs J) = \sum_{\substack{\epsilon_\ell \in\{0,\beta-1\} \\ 1\le \ell\le L}}
 	\frac{S^{k-L}(\log R)^{-k}}{2^L(2\pi i)^{2k-L}}
		&	\idotsint\limits_{\substack{\Re(s_j) = 1/\log R \,(1\le j\le 2k-L) \\ s_{J_\ell} = \epsilon_\ell \,(1\le \ell\le L)}} 
		\sum_{\substack{P^+(d_j)\le y \\ 1\le j\le 2k}} \frac{\prod_{j=1}^{2k}\mu(d_j)d_j^{-s_j}}{[d_1,\dots,d_{2k}]}
			\\
		&\qquad \times \left( \prod_{j=1}^{2k} \hat{w}_R(s_j) \right)
					\dee s_1\cdots \dee s_{2k-L} 
		+ O\left(\frac{1}{\log R}\right) .
}

The above expression is sufficient for handling the terms $\bs J \in \CJ$ with at least one $J_\ell$ of odd cardinality: following the argument of Section \ref{lb-odd} with the obvious modifications implies that
\eq{X-ub1}{
X(\bs J)  \ll (\log(q\log R))^{O(1)}  \frac{(\log y)^{\binom{2k}{k}-2k-1+L}}{(\log R)^L} ,
}
where we used the fact that $\sup_{q<p\le Q} (1+\chi(p))/p\ll 1$. 

However, we need to be more careful on our lower bound for the main term, that is to say for $X(\bs J)$ with $\#J_\ell=2$ for all $\ell$. First of all, by relabelling our variables, we may assume that $J_\ell=\{\ell,\ell+k\}$ for all $\ell$. In our expression \eqref{X-int} for $X(\bs J)$, we see that $s_{J_\ell}=\epsilon_\ell\in\{0,\beta-1\}$ implies that $s_{\ell+k}=-s_{\ell}+O(1/\log Q)$. We want to replace $s_{\ell+k}$ by $-s_{\ell}$. This introduces an error that we will control by an application of the mean value theorem. In particular, we need to understand the derivative of the integrand. If 
\[
G(\bs s)=\sum_{\substack{P^+(d_j)\le y \\ 1\le j\le k}}
			\frac{\prod_{j=1}^{2k}\chi(d_j)d_j^{-s_j}}{[d_1,\dots,d_{2k}]} ,
\]
then $-G'(\bs s)/G(\bs s)$ equals $\sum_{I\in\CS^*(2k)} \sum_{p\le y}\chi^{\#I}(p)\log p/p^{1+s_I}$, plus lower order terms, so that $G'(\bs s)\ll (\log y)G(\bs s)$ for the vectors $\bs s$ we are considering. Similarly, $\hat{w}_R(s+\epsilon)/R^{s+\epsilon}=\hat{w}_{R}(s)/R^s+O(\epsilon /(|s|+1) )$ by \eqref{h-bound-prelim}. Since we also have that
\[
\idotsint\limits_{\substack{\Re(s_j) = 1/\log R \\ |s_{j+k}+s_j|\le 1-\beta \\ 1\le j \le k}} 
		\abs{G(\bs s)}  \prod_{j=1}^{k} \frac{|\dee s_j|}{1+|s_j|^2} 
			  \ll (\log(q\log R))^{O(1)} \frac{(\log y)^{\binom{2k}{k}-k}}{(\log R)^k} ,
\]
by the argument leading to \eqref{X-ub1}, we conclude that
\als{
X(\bs J) = 
 	\frac{(1+R^{\beta-1})^k}{(2\pi i\log R)^k}
		&	\idotsint\limits_{\substack{\Re(s_j) = 1/\log R \\ s_{j+k} = -s_j \\ 1\le j \le k}} 
		G(\bs s)  \left( \prod_{j=1}^{2k} \hat{w}_R(s_j) \right)
					\dee s_1\cdots \dee s_{k} \\
		&\qquad + O\left(\frac{\log y}{\log Q}\cdot  \frac{(\log(q\log{R}))^{O(1)} (\log y)^{\binom{2k}{k}-k}}{(\log R)^k} \right) .
}
The main term can now be bounded from below as in Section \ref{lb-even}. We thus arrive to the lower bound
\[
X(\bs J)  \ge c_k \frac{(\log y)^{\binom{2k}{k}-k}(\log R)^{-k}}{(\log q)^{O(1)}} 
	-  O\left(\frac{(\log(q\log{R}))^{O(1)} (\log y)^{\binom{2k}{k}-k+1}}{(\log R)^{k+2}} \right) ,
\]
using that $\sum_{q<p\le Q}(1+\chi(p))/p\ll 1$ and $y\le R\le e^{\sqrt{\log Q}}$ here. This completes the proof of \eqref{chars-lb}, and thus of Theorem \ref{thm-chars}(c).



\begin{thebibliography}{99}


\bibitem{BNP} M. Balazard, M. Naimi and Y.-F. S. P\'etermann,\,
\emph{\'Etude d'une somme arithm\'etique multiple li\'ee \`a la fonction de M\"obius},
Acta Arith. {\bf 132} (2008),        245-298. 

\bibitem{delaB}  R. de la Bret\`eche,\
 \emph{Estimation de sommes multiples de fonctions arithm\'etiques}, Compos. Math. {\bf 128}                     (2001), 261-298.


\bibitem{Dav}
Harold Davenport.
\newblock {\em Multiplicative number theory}, volume~74 of {\em Graduate Texts
  in Mathematics}.
\newblock Springer-Verlag, New York, third edition, 2000.
\newblock Revised and with a preface by Hugh L. Montgomery.


\bibitem{DIT}
F. Dress, H. Iwaniec and G. Tenenbaum,\,
\emph{Sur une somme li\'ee \`a la function de M\"obius},
J. reine angew. Math. {\bf 340} (1983), 53--58. 

\bibitem{EFG} S. Eberhard, K. Ford and B. Green, \
\emph{Permutations fixing a $k$-set}, IMRN, to appear.  arXiv:1507.04465

\bibitem{EH}
P. Erd\H os,  and R. R. Hall,\,
\emph{On the M\"obius function},
J. Reine Angew. Math. {\bf 315} (1980), 121--126. 

\bibitem{Ford} K. Ford,\,
\emph{The distribution of integers with a  divisor  in a given interval},  Ann. of Math.  168,   (2008), 367-433.

\bibitem{Ford-rough} \bysame,\,
{\it Rough integers with a divisor in a given interval.} J. Australian Math. Soc., to appear, 17 pages, arXiv:1901.02548.

\bibitem{Friedlander}    J. Friedlander,\,
\emph{Sifting short intervals. {II}}, Math. Proc. Cambridge Philos. Soc., {\bf 92} (1982), 381--384.

\bibitem{opera} J. Friedlander and H. Iwaniec,\,
\emph{Opera de cribro.} American Mathematical Society Colloquium Publications, 57. 
American Mathematical Society, Providence, RI, 2010.

\bibitem{Gra} A. Granville,\,
\emph{Cycle lengths in a permutation are typically Poisson.}
Electron. J. Combin. 13 (2006), no. 1, Research Paper 107, 23 pp.

\bibitem{divisors} R. R. Hall and G. Tenenbaum,\,
\emph{Divisors.} 
Cambridge Tracts in Mathematics, 90. Cambridge University Press, Cambridge, 1988.

\bibitem{GPY}A. Goldston, J. Pintz and C. Y. Y{\i}ld{\i}r{\i}m, \,
\emph{Primes in tuples. {I}.}
Ann. of Math. (2) 170  (2009), no.2, 819--862.

\bibitem{IK} H. Iwaniec and E. Kowalski,\,
\emph{Analytic number theory.} 
American Mathematical Society Colloquium Publications, 53. American Mathematical Society, Providence, RI, 2004.

\bibitem{Kou10} D. Koukoulopoulos,\,
\emph{Localized factorizations of integers.} Proc. Lond. Math. Soc. (3) 101 (2010), no. 2, 392--426.

\bibitem{Kou13} \bysame,\,
\emph{Pretentious multiplicative functions and the prime number theorem for arithmetic progressions.}
Compos. Math., 149 (2013), no. 7, 1129--1149.

\bibitem{kou}  \bysame,\,
\emph{On multiplicative functions which are small on average.} 
Geom. Funct. Anal. 23 (2013), no. 5, 1569--1630. 

\bibitem{maynard} J. Maynard, \emph{Small gaps between primes},  
Ann. of Math.  181 (2015),  383--413. 

\bibitem{Moto}  Y. Motohashi,\,
\emph{A multiple sum involving the M\"obius function}, 
Publ. Inst. Math. (Beograd) (N.S.) {\bf 76(90)} (2004), 31--39.

\bibitem{Rosenbook} M. Rosen, \,
\emph{Number theory in function fields.} Graduate Texts in Mathematics, 210.
Springer-Verlag, New York, 2002.

\bibitem{Sel} A. Selberg, \emph{On an elementary method in the theory of primes},
Norske Vid. Selsk. Forh., Trondhjem 19 (1947), 64--67.

\bibitem{SelCollected}  A. Selberg, \emph{Collected papers. II},
Springer Collected Works in Mathematics, Springer, Heidelberg, 2014.


\bibitem{Pollack} P. Pollack, \,
\emph{Bounded gaps between primes with a given primitive root},
Algebra Number Theory 8 (2014), 1769--1786.

\bibitem{tao}
T. Tao, \emph{Polymath8b: Bounded intervals with many primes, after Maynard}, Blog note. \url{https://terrytao.wordpress.com/2013/11/19/polymath8b-bounded-intervals-with-many-primes-after-maynard/}

\bibitem{Ten} G. Tenenbaum,\,
{\it Introduction to analytic and probabilistic number theory.} 
Translated from the second French edition (1995) by C. B. Thomas. Cambridge Studies in Advanced Mathematics, 46. Cambridge University Press, Cambridge, 1995.

\bibitem{Tit} E. C. Titchmarsh,\,
{\it The theory of the Riemann zeta-function.}
Second edition. Edited and with a preface by D. R. Heath-Brown. The Clarendon Press, Oxford University Press, New York, 1986.

\bibitem{Zhang} Y. Zhang,\,
\emph{Bounded gaps between primes},
Ann. of Math. (2) 179 (2014), no. 3, 1121--1174.
\end{thebibliography}
\end{document}